\documentclass[reqno,final,oneside]{amsart}
\usepackage[textwidth=16.1cm]{geometry}
\usepackage{mathrsfs}  
\usepackage{cancel} 
\usepackage[dvipsnames]{xcolor}
\usepackage{enumitem}
\usepackage{mathtools}
\usepackage[final]{hyperref}
\usepackage{bm}
\usepackage{subcaption}
\usepackage{comment}
\usepackage{xargs}
\usepackage{soul}
\usepackage[textsize=tiny,textwidth=2.05cm]{todonotes}
\newcommandx{\task}[2][1=]{\todo[linecolor=blue,backgroundcolor=blue!30,#1]{#2}}
\newcommandx{\note}[2][1=]{\todo[linecolor=green,backgroundcolor=green!30,#1]{#2}}
\newcommandx{\error}[2][1=]{\todo[linecolor=red,backgroundcolor=red!30,#1]{#2}}
\usepackage{marginnote}
\usepackage[dvipsnames]{xcolor}
\usepackage{mathrsfs}
\usepackage{mathtools}
\usepackage{esint}
\usepackage{stackengine}
\usepackage{hyperref}
\usepackage{bbm}
\newcommand{\indic}{\mathbbm{1}}

\newtheorem{theorem}{Theorem}[section]
\newtheorem{lemma}[theorem]{Lemma}
\newtheorem{proposition}[theorem]{Proposition}
\newtheorem{corollary}[theorem]{Corollary}

\theoremstyle{definition}
  \newtheorem{definition}[theorem]{Definition}

\theoremstyle{remark}
  \newtheorem{remark}[theorem]{Remark}

\newcommand{\dt}{\di t}
\newcommand{\dx}{\di x}

\newcommand{\dH}{\di \mathcal{H}^2}

\newcommand{\eps}{\varepsilon}
\newcommand{\N}{\mathbb{N}}
\newcommand{\Z}{\mathbb{Z}}
\newcommand{\R}{\mathbb{R}}
\newcommand{\C}{\mathbb{C}}
\newcommand{\Id}{\mathbf{Id}}
\newcommand{\id}{\mathbf{id}}
\newcommand{\vphi}{\varphi}
\newcommand{\weakly}{\rightharpoonup}
\newcommand{\weaklystar}{\stackrel{*}{\rightharpoonup}}
\newcommand{\defas}{\coloneqq}
\newcommand{\pl}{\partial}
\newcommand{\haus}{\mathcal{H}}
\newcommand{\aC}{C_0}
\newcommand{\ac}{c_0}
\newcommand{\CW}{\mathbb{C}_{W^{\rm el}}}
\newcommand{\CD}{\mathbb{C}_R}
\usepackage{stackengine}
\newcommand{\cdddot}{\mathrel{\Shortstack{{.} {.} {.}}}}
\newcommand*{\di}{\mathop{}\!\mathrm{d}}

\newcommand{\bigo}{\mathcal{O}}

\DeclareMathOperator{\dist}{dist}

\DeclareMathOperator{\diver}{div}

\newcommand{\sym}{{\rm sym}}
\newcommand{\skw}{{\rm skew}}

\newcommand{\EEE}{\color{black}}

\newcommand{\ZZZ}{\color{black}}
\newcommand{\QQQ}{\color{black}}

\newcommand{\LLL}{\color{black}}

\newcommand{\GGG}{\color{black}}

\newcommand{\OOO}{\color{black}}

\newcommand{\REV}{\color{black}}

\newcommand{\cplen}{\mathcal{W}^{\mathrm{cpl}}_h}

\newcommand{\mechen}{\mathcal{M}_h}
\newcommand{\toten}{\mathcal{E}}
\newcommand{\inten}{W^{\mathrm{in}}}

\newcommand{\diss}{\mathcal{R}_h}
\newcommand{\elpot}{W^{\mathrm{el}}}
\newcommand{\cplpot}{W^{\mathrm{cpl}}}
\newcommand{\hypot}{H}
\newcommand{\felpot}{W}
\newcommand{\disspot}{R}
\newcommand{\Wid}{\mathcal{W}_{\id}^{\OOO h \EEE}}

\newcommand{\hc}{\mathbb{K}}

\newcommand{\hcm}{\mathcal{K}}
\newcommand{\hcmnoh}{\mathcal{K}}

\newcommand{\drate}{\xi}

\DeclareMathOperator{\esssup}{ess\sup}
\DeclareMathOperator{\essinf}{ess\inf}

\DeclarePairedDelimiterX\setof[1]\{\}{#1}
\DeclarePairedDelimiterX\abs[1]\lvert\rvert{#1}
\DeclarePairedDelimiterX\norm[1]\lVert\rVert{#1}
\DeclarePairedDelimiterX\sprod[2]\langle\rangle{#1, #2}

\usepackage{enumitem}
\numberwithin{equation}{section}

\setlist[enumerate,1]{font=\normalfont}
\setlist[itemize,1]{font=\normalfont}
\newlist{thmlist}{enumerate}{1}
\setlist[thmlist]{label=(\roman{thmlisti}),
	ref=(\roman{thmlisti}),font=\normalfont,
	noitemsep}

\title{Derivation of   a   von K\'arm\'an plate theory for thermoviscoelastic solids}


\author[R.~Badal]{Rufat Badal}
\address[Rufat Badal]{
  Department of Mathematics \\
  Friedrich-Alexander Universit\"at Erlangen-N\"urnberg \\
  Cauerstr.~11, D-91058 Erlangen, Germany
}
\email{rufat.badal@fau.de}

\author[M.~Friedrich]{Manuel Friedrich} 
\address[Manuel Friedrich]{%
  Department of Mathematics \\
  Friedrich-Alexander Universit\"at Erlangen-N\"urnberg \\
  Cauerstr.~11, D-91058 Erlangen, Germany \\
  \&  Applied \EEE Mathematics M\"{u}nster \\
  University of M\"{u}nster \\
  Einsteinstr.~62, D-48149 M\"{u}nster, Germany
}
\email{manuel.friedrich@fau.de}

\author[L.~Machill]{Lennart Machill}
\address[Lennart Machill]{
 Applied \EEE Mathematics M\"{u}nster \\
  University of M\"{u}nster \\
  Einsteinstr.~62, D-48149 M\"{u}nster, Germany  
  }
\email{lennart.machill@uni-muenster.de}

\begin{document}

\maketitle
\begin{abstract}
We derive a von Kármán plate theory from a three-dimensional quasistatic nonlinear model for nonsimple thermoviscoelastic materials in the Kelvin-Voigt rheology, in which the elastic and the viscous stress tensor comply with a frame indifference principle \cite{MielkeRoubicek20Thermoviscoelasticity}.
In a dimension-reduction limit, we show that weak solutions to the nonlinear system \GGG of equations \EEE converge to weak solutions of an effective two-dimensional system featuring \GGG mechanical equations \EEE for viscoelastic von Kármán plates\GGG, previously \EEE derived in \cite{FK_dimred}, \QQQ coupled with a linear \EEE heat-transfer equation. The main challenge lies in deriving a priori estimates for rescaled displacement fields and temperatures, which requires the adaptation of generalized Korn's inequalities and bounds for heat equations with $L^1$-data to thin domains. 
\end{abstract}

\section{Introduction}  
Understanding and predicting the complex   behavior  of materials undergoing deformation is a pivotal challenge in theoretical and applied sciences. Many three-dimensional models in continuum mechanics are nonlinear and nonconvex, resulting in numerical approximations of high computational cost. Therefore, the derivation of simplified models maintaining the essential features of the original ones plays a significant role in current research. A prominent example in this direction is the variational derivation of plate models, where a rigorous relationship between full three-dimensional descriptions and their lower-dimensional counterparts is achieved by means of $\Gamma$-convergence \cite{DalMaso93AnIntroduction}. This theory has been developed thoroughly in the last two decades starting from the celebrated results by {\sc Friesecke, James, and M\"uller} \GGG \cite{FJM_hierarchy, FrieseckeJamesMueller:02} \EEE on a hierarchy of effective models. Yet, applications to the static setting have largely overshadowed the investigation of evolutionary problems. Based on recent advancements for time-dependent problems  \cite{RBMFMK, FK_dimred, MielkeRoubicek20Thermoviscoelasticity}, in this work we aim \ZZZ at performing \EEE a  dimension reduction of a thermodynamically-consistent model for nonlinear \emph{thermoviscoelastic} solids \GGG in the Kelvin-Voigt rheology \EEE which couples the \GGG balance of momentum in its quasistatic variant \EEE with a \GGG nonlinear \EEE heat-transfer equation.

We start by introducing the large strain setting analyzed in \cite{MielkeRoubicek20Thermoviscoelasticity}. We consider a model for second-grade nonsimple materials in the  Kelvin-Voigt rheology without inertia which is governed by the system of equations
\begin{equation}\label{viscoel_nonsimple}
g^{3D}  =  - \diver\left(
      \partial_F W(\nabla w, \theta)
      - {\rm div}(\partial_G H(\nabla^2 w))
      + \partial_{\dot{F}}R(\nabla w, \partial_t \nabla   w, \theta)
    \right)
\qquad \text{in $[0,T]\times \Omega$.}
\end{equation}
Here, $[0, T]$ is a process time interval for some \GGG time horizon \EEE $T > 0$, $\Omega \subset \R^3$ is a bounded domain representing the reference configuration, $w \colon [0, T] \times \Omega \to \R^3$ indicates a Lagrangian \emph{deformation},  and $\theta$ denotes the \emph{temperature} inside the material. By $g^{3D} \colon [0, T] \times \Omega \to \R^3$ we indicate a volume density of \emph{external forces} \EEE acting on $\Omega$. The elastic stress tensor  
\GGG involves \EEE the first Piola-Kirchhoff stress tensor $\partial_F W$, where $F \in \R^{3\times 3}$ is the placeholder of the deformation gradient $\nabla w$. \GGG This tensor can be \ZZZ deduced from a \EEE \GGG frame indifferent \EEE \emph{energy density} \EEE $W \colon \R^{3\times 3} \times [0, \infty) \to \R \cup \setof{+\infty}$. 
Moreover, $R \colon \R^{3 \times 3} \times \R^{3 \times 3} \times [0, \infty) \to \R $ denotes a (pseudo)potential of dissipative forces \GGG and induces the viscous stress tensor $\partial_{\dot F} R$, \EEE where \GGG $\dot F \in \R^{3\times 3}$ \EEE stands for the time derivative of $F$. \ZZZ It \EEE complies with a  \GGG dynamical \EEE
frame indifference principle, meaning that for all $F$ it holds that
\begin{equation*}
  R(F,\dot F,\theta) = \hat R(C,\dot C,\theta)
\end{equation*}
for some nonnegative function $\hat R$, where $C \defas F^T F$ is the right Cauchy-Green tensor with derivative in time $\dot C \defas \dot F^T F + F^T \dot F$. For a thorough discussion of the latter principle in the context of viscous stresses, we refer to {\sc Antman} \cite{Antmann04Nonlinear}.
Already without a thermodynamical coupling,   existence results \EEE for models respecting frame indifference are  scarce\EEE, we refer here, e.g., to \cite{Lewick} for local in-time existence and to \cite{demoulini} for an existence result in the space of measure-valued solutions.
To date, weak solutions in finite-strain isothermal viscoelasticity can only be guaranteed by \GGG resorting \EEE to higher-order regularizing terms.
Following \cite{FriedrichKruzik18Onthepassage, MielkeRoubicek20Thermoviscoelasticity}, we therefore include an additional term in the elastic stress tensor, usually called \emph{hyperstress} \ZZZ $\partial_G H$, \EEE via a potential $H \colon \R^{3 \times 3 \times 3}\to [0,\infty)$ which depends on the second gradient $\nabla^2 w$ \ZZZ with corresponding placeholder $G$. \EEE The concept of second-grade nonsimple materials goes back to \EEE {\sc Toupin} \cite{Toupin62Elastic,Toupin64Theory} and indeed proved to be useful in mathematical continuum mechanics, see e.g.~\cite{BallCurrieOlver81Null, Batra76Thermodynamics, HealeyKroemer09Injective, MielkeRoubicek16Rateindependent, Podio02Contact}.
The approach of \cite{FriedrichKruzik18Onthepassage, MielkeRoubicek20Thermoviscoelasticity} has been extended in various directions over the last years, including models allowing for self-contact \GGG \cite{gravina, kroemrou}, \EEE a nontrivial coupling with a diffusion equation \cite{liero}, homogenization \cite{gahn}, inertial effects \cite{Schwarzacher}, or applications to fluid-structure interactions \cite{Schwarzacher}.
 While the aforementioned results are formulated using the Lagrangian approach, several recent works employ the alternative Eulerian perspective instead \cite{Roubicek23Eulerian, Roubicek23Eulerian2, Roubicek23Eulerian3, Roubicek23Eulerian4}. \EEE

In contrast \ZZZ to the  models mentioned previously,  \EEE
 we  are interested in a nonlinear coupling of  \eqref{viscoel_nonsimple} \EEE with a heat-transfer
equation of the form
\begin{equation}\label{heat}
  c_V(\nabla w,\theta) \, \partial_t\theta =
    \diver(\mathcal{K}(\nabla w, \theta) \nabla\theta)
    + \xi(\nabla w, \partial_t \nabla w, \theta)
    + \theta \partial_{F \theta} W^{\rm cpl}(\nabla w, \theta) : \partial_t \nabla   w \qquad  \text{in $[0,T]\times \Omega$}. 
\end{equation}
Following \cite{MielkeRoubicek20Thermoviscoelasticity}, we assume that $W$ is the sum of a purely elastic potential $W^{\rm el}$ only depending on the deformation and a coupling potential $W^{\rm cpl}$ depending additionally on the temperature.
To be more precise, we assume
\begin{align*}
W(F,\theta) = W^{\rm el}(F) + W^{\rm cpl}(F,\theta).
\end{align*}
With \GGG regard \EEE to the heat-transfer equation, $c_V(F,\theta) = -\theta \partial^2_\theta W^{\rm cpl}(F, \theta)$ is the \emph{heat capacity} \GGG and \EEE $\mathcal{K}$ denotes the \emph{heat conductivity} tensor, \OOO initially \GGG defined in the deformed configuration and \EEE pulled back to the reference configuration. \GGG The \emph{dissipation \GGG rate} is defined by \EEE
\begin{equation*}
  \xi(\nabla w, \partial_t \nabla w, \theta) \defas \partial_{\dot F} R(\nabla w, \partial_t\nabla   w, \theta) : \partial_t \nabla  w,
\end{equation*}
and the last term \GGG in \eqref{heat} \EEE \GGG corresponds to an \emph{adiabatic} effect, playing the role of a heat source or sink, respectively. \EEE From a technical point of view, the term $\xi$ is particularly delicate as it only has $L^1$-regularity and therefore requires to resort to 
\GGG weak \EEE solution concepts for heat equations.  \EEE
Equation \eqref{heat} can be derived from  \emph{Fourier's law} formulated in the \ZZZ deformed \EEE configuration, see e.g.~\cite{MielkeRoubicek20Thermoviscoelasticity}.
The coupled system \eqref{viscoel_nonsimple}--\eqref{heat} is complemented by suitable initial and boundary conditions, see Section~\ref{sec:model}.
Global-in-time weak solutions to the above system have been derived recently via a minimizing movements scheme \cite{MielkeRoubicek20Thermoviscoelasticity}, see also \GGG \cite{RBMFMK} \EEE for an alternative proof using \GGG an improved \EEE approximation scheme without the use of regularizations.

We are interested in the derivation of effective simplified models of \eqref{viscoel_nonsimple}--\eqref{heat}.
Whereas in \cite{RBMFMK, FriedrichKruzik18Onthepassage}  the linearization of large-to-small strains at small temperatures has been addressed, in the present paper we tackle the more challenging problem of deriving a dimensionally reduced model. Some results on dimension reduction in the isothermal case have been treated recently in the literature. We mention results in nonlinear elastodynamics, both for the case of plates \GGG \cite{Abels3, AMM_vKaccel} \EEE and rods \cite{Abels2}, and models in nonlinear viscoelasticity \cite{FK_dimred, MFLMDimension2D1D, MFLMDimension3D1D}, neglecting inertial effects but allowing for viscous damping.  Our work is closest to \cite{FK_dimred} where a von Kármán plate theory is derived starting from \eqref{viscoel_nonsimple} without temperature dependence. Our goal is to extend this result by performing a dimension reduction for the coupled system \eqref{viscoel_nonsimple}--\eqref{heat}. Although  models of thin thermoviscoelastic plates have been investigated, see e.g.\ \cite{plateref2, plateref1}, to the best of our knowledge, rigorous dimension-reduction results for models with nontrivial mechanical-thermal couplings are currently unavailable in the literature.
 
We now describe our result in more \GGG detail. \EEE We consider the system \eqref{viscoel_nonsimple}--\eqref{heat} on a thin body $\Omega = \Omega_h = \Omega' \times (-h/2,h/2)$ of thickness $h > 0$. As shown in \cite{FK_dimred} (see \cite{FJM_hierarchy} for the purely static case), for forces $g^{3D}$ scaling like $\sim h^3$ and initial deformations suitably close to the identity, which corresponds to an energy per
thickness of $\sim h^4$, the \GGG three-dimensional \EEE problem can be related rigorously to  \GGG two-dimensional \EEE limiting mechanical equations for viscous von \GGG Kármán \EEE plates,   in terms of an  \emph{in-plane displacement field} $u \colon \Omega' \to \R^2 $  and an \emph{out-of-plane displacement} $v \colon \Omega' \to \R^2$, \GGG where the corresponding \GGG three-dimensional \EEE displacements are \EEE suitably rescaled by $\frac{1}{h^2}$ and $\frac{1}{h}$, respectively. Inspired by the linearization result \cite{RBMFMK}, for the \emph{temperature} variable $\theta$, we allow for different scalings $\sim h^\alpha$ in terms of a  parameter $\alpha >0$, and \GGG let
$\mu  \colon \GGG \Omega' \EEE \to [0,\infty)$ be the limit variable of the corresponding rescaled temperature. \EEE \GGG The limiting model  depends on the choice of $\alpha$ and, as we will see later, is only meaningful for $\alpha \in [2,4]$. \EEE
 In the limit of small thickness  $h \to 0$, we identify  \EEE the effective system of equations  \EEE \GGG on \EEE $[0,T]\times \Omega'$ \GGG as \EEE
    \begin{equation}\label{eq:mech2Dintro}
\begin{aligned}
  \diver\Big(\C^2_{\GGG W^{\rm el} \EEE}\big( e(u)  +  \frac{1}{2} \nabla' v \otimes \nabla' v  \big) + \mu (\mathbb{B}^{(\alpha)})'' + \ZZZ \C^2_R \EEE \big( e(\partial_t u) +   \partial_t \nabla'  v   \odot \nabla' v \big) \Big) &= 0, \\
  -\diver\Big(
    \Big(
      \C^2_{\GGG W^{\rm el} \EEE}\big( e(u)  + \frac{1}{2} \nabla' v \otimes \nabla' v  \big)
      + \mu (\mathbb{B}^{(\alpha)})'' 
      + \ZZZ \C^2_R \EEE \big( e( \GGG \partial_t \EEE u) + \partial_t \nabla'  v \odot \nabla' v \big) \Big) \nabla' v
    \Big) \\
  + \tfrac{1}{12}  \diver \diver \Big(
      \C^2_{\GGG W^{\rm el} \EEE} (\nabla')^2 v + \ZZZ \C^2_R \EEE \partial_t(\nabla')^2  v
    \Big)
  &= f^{2D}
\end{aligned}
\end{equation}
and 
\begin{equation}\label{eq:temperature2Dintro}
\begin{aligned}
  & \CD^{2,\alpha}   \Big(
      \sym( \partial_t \nabla'   u)
      +  \partial_t \nabla'   v \odot \nabla' v
    \Big) : \Big(
      \sym( \partial_t \nabla'   u)
      +  \partial_t \nabla'   v \odot \nabla' v
    \Big) \\
  &\quad= \bar c_V \partial_t \mu 
    - \diver(\tilde{\mathbb{K}} \nabla \mu) 
    - \frac{1}{12}\CD^{2,\alpha} \partial_t (\nabla')^2  v : \partial_t (\nabla')^2.
\end{aligned}
\end{equation}
\ZZZ Here,  we write $\nabla'$ and $(\nabla')^2$ \EEE for the in-plane gradient and Hessian, respectively. \EEE  
The    mechanical evolution  \eqref{eq:mech2Dintro} features membrane and bending contributions both in the elastic and the viscous stress, in terms of the linear strain  $e(u) \defas  (\nabla' u + (\nabla' u)^T) / 2$, where  the tensors of elastic constants $\C^2_{\GGG W^{\rm el} \EEE}$ and viscosity coefficients $\CD^{2}$ are given by the \EEE  second order derivatives \EEE of ${\GGG W^{\rm el} \EEE}$ and $R$ evaluated at \GGG $\Id$ and \EEE $(\Id, 0)$, \GGG respectively. \EEE 
Moreover,   $(\mathbb{B}^{(\alpha)})'' $ represents a thermal expansion matrix which is only active in the case $\alpha = 2$ \GGG and suitably related to $W^{\rm cpl}$, \EEE and $f^{2D}$ denotes an effective force\GGG, which in turn is linked to $g^{3D}$, \EEE see Sections \ref{a new subsection} and \ref{sec:quadraticforms}  for  details. In the  heat-transfer equation \eqref{eq:temperature2Dintro}, \GGG $\tilde{\mathbb{K}}$ \EEE represents the limiting heat conductivity tensor, \GGG $\overline{c}_V$ \EEE denotes the constant heat capacity at zero strain and temperature,  and $\CD^{2,\alpha}$ is given by $\CD^{2}$ for $\alpha = 4$ and zero otherwise. As in the linearization result \cite{RBMFMK}, we observe that  the equations are only  one-sided coupled for $\alpha \in \lbrace 2, \, 4\rbrace$ and that there is no coupling  for $\alpha \in (2,4)$. Moreover,  we mention that the limiting problem \OOO contains \EEE no  spatial gradients of \GGG $\nabla' u$, \EEE although the initial nonlinear model was  formulated \EEE for a nonsimple material. \EEE

 Our main goal consists in showing that weak solutions to \eqref{viscoel_nonsimple}--\eqref{heat} converge to  \eqref{eq:mech2Dintro}--\eqref{eq:temperature2Dintro} in a suitable sense, see  Theorem~\ref{maintheorem} below. As a byproduct, we also obtain existence of weak \GGG solutions \EEE to the limiting \GGG two-dimensional \EEE problem.
 \REV We mention that, as in  \cite{RBMFMK}, we perform  a linearization at zero temperature.  The case of temperatures in the vicinity of a positive critical temperature is more challenging. \REV Addressing this would require further technical estimates, which we omit to avoid exceeding the scope of the paper. However, we note that recent work \cite{RBMFMKLM} provides a detailed linearization for positive temperatures on thick domains. \EEE Moreover, strictly speaking, except for the case $\alpha = 4$, we need to slightly regularize the heat equation \eqref{heat} in order to rigorously perform the dimension reduction, \ZZZ see \eqref{diss_rate_truncated} for details. \EEE As in \cite{RBMFMK}, however, this  regularization does not affect the limiting problem. Indeed, on a formal level, the regularized as well as the \GGG nonregularized \EEE equation converge towards the same effective equation \eqref{eq:temperature2Dintro}.

Let us comment on the main \GGG difficulties \EEE and \GGG novelties \EEE of this paper.
The geometric rigidity result \cite{FrieseckeJamesMueller:02} is the cornerstone for proving $\Gamma$-converging results in the static setting, see e.g.\ \cite{DalMasoNegriPercivale02Linearized, Freddi2012, Freddi2013, FJM_hierarchy, FrieseckeJamesMueller:02, FriJaMoMu, Mora, Mora2},  or for passing to the limit in  equilibrium equations \cite{mora-scardia, MP_thin_plates}. Indeed, it  allows  to control the local deviations of deformations from approximating rigid motions which implies a compactness result for rescaled displacement fields.   In the  evolutionary setting, this estimate is still of utmost importance but the situation is more delicate as suitable a priori estimates with optimal scaling in $h$ are also needed for the time derivatives $\partial_t u$ and $\partial_t v$, as well as \GGG for \EEE the temperature. In the case of a \emph{fixed} domain, such estimates have been derived in \EEE \cite{RBMFMK} and \cite{MielkeRoubicek20Thermoviscoelasticity}. Our challenge lies in refining them to the setting of thin domains in order to ensure the correct scaling of all quantities with respect to the \ZZZ  thickness $h$. \EEE Since we deal with a heat equation with $L^1$-data, on several occasions we need to resort  to test functions in the spirit of {\sc Boccardo and Gallouët} \cite{BoccardoGallouet89Nonlinear}, see also \cite{RBMFMK, MielkeRoubicek20Thermoviscoelasticity}. In order to derive a priori bounds on the strain rates, we  ought \EEE to employ a generalization of Korn's inequality due to {\sc Pompe} \cite{Pompe} in the version of \cite[Corollary 3.4]{MielkeRoubicek20Thermoviscoelasticity}. For both issues, we  must \EEE \GGG investigate \EEE the dependence of  constants on the \ZZZ thickness $h$. \EEE In particular, in Theorem~\ref{pompescaling} we derive a generalized Korn's inequality on thin domains with optimal scaling of the constant. This result might be of independent interest.

 Concerning \EEE the limiting passage, \EEE we rely on the techniques employed in \cite{RBMFMK} as well as the ones from \GGG \cite{AMM_vKaccel, MP_thin_plates}. Since the latter \ZZZ work \EEE does not account for viscous effects, a further novelty of our \ZZZ paper \EEE compared to \cite{AMM_vKaccel} is the derivation of the limit $h \to 0$ of the viscous stress $\partial_{\dot{F}}R(\nabla w^{h}, \partial_t \nabla w^{h}, \theta^{h})$.
Formally, one can proceed similarly to \GGG \cite{AMM_vKaccel, MP_thin_plates}  \EEE in the case of the elastic stress. However, in the mathematically rigorous derivation, several technical difficulties arise. Among others, the most severe one is the necessity of strong convergence of the rescaled viscous stress. 
The same challenge has already been encountered in \GGG \cite{RBMFMK, MielkeRoubicek20Thermoviscoelasticity} \EEE in the passage from time-discrete to time-continuous solutions and we solve the issue by a careful adaptation of the arguments therein to the setting of dimension reduction. \ZZZ In particular, as an auxiliary step, \EEE we pass to the limit in an energy balance related to the mechanical equations \eqref{viscoel_nonsimple} and \eqref{eq:mech2Dintro}.
\REV Notice that nonsimple materials are used not only for the existence of weak solutions, but also for the derivation of a priori estimates as well as for the limiting passage, see e.g.\ Lemma~\ref{lemma:rigidity} and \eqref{conv:stress:visc}. \EEE

\REV
Let us highlight that our techniques serve as a starting point for a dimension reduction and may be applicable to other rheological models.
While linearization has been performed for energetic solutions of the Poynting–Thomson model \cite{Poynting} including Maxwell \emph{and} Kelvin-Voigt elements, 
one might also perform a dimension reduction of this model for nonsimple materials by also including thermal effects, see e.g.\ \cite{Meo}.
As a further step, plastic effects \cite{MielkeRossiSavare,Ulisse} could also be incorporated.
\EEE

Before closing this introduction, we mention that in the isothermal setting our result reduces to the  purely viscoelastic case. In this \GGG paper,  we \EEE provide   an alternative \EEE proof of the results in \cite{FK_dimred} where the proof techniques were based on evolutionary  $\Gamma$-convergence   \cite{Mielke, S1,S2}. \EEE   In the present setting, however,  it is not clear \EEE how this technique can be generalized to systems of equations allowing also for  thermodynamical coupling in \eqref{viscoel_nonsimple}--\eqref{heat}.   Therefore, we use a different proof strategy here which consists in \EEE  deriving the  effective model  by passing to the limit directly on the PDE level without resorting to  (evolutionary) $\Gamma$-convergence. \EEE

The plan of the paper is as follows. In Section~\ref{sec:model}, we introduce the three- and two-dimensional models in more detail and state our main results.
Section~\ref{sec:korn} is devoted to the proof of the optimal scaling of the generalized Korn's inequality. Section~\ref{sec:existenceapriorisection} provides  a priori estimates for solutions in the thin domain \ZZZ with precise dependence on  $h$. \EEE Finally, Section~\ref{sec:proofofcompactness} addresses the dimension reduction.

\section{The model and main results}\label{sec:model}
\subsection*{Notation}

In what follows, we use standard notation for Lebesgue and Sobolev spaces. If the target space is a Banach space $E \neq \R$, we use the usual notion of Bochner-Sobolev spaces, written $W^{k,p}(\Omega;E)$, containing weak derivatives up to the $k$-th order, that are integrable with the $p$-th power (if $1\leq p < +\infty$) or essentially bounded (if $p = +\infty$). Denoting by $d \in \{2, 3\}$ the dimension,   the $d$-dimensional Lebesgue measure of a measurable set $U\subset \R^d$ is indicated by $\vert U \vert$, and the mean integral is written as $\fint_U$. By   $\nabla$ and $\nabla^2$ we \EEE denote the spatial gradient and Hessian, respectively.  If a function only depends on two  spatial  variables $x' \defas (x_1,x_2)$, we use the notation $\nabla'$ and $(\nabla')^2$. Frequently, we extend such  functions constantly to the third dimension without relabeling.

  The lower index $_+$  indicates \EEE nonnegative elements and functions, respectively. Given $a, \, b \in \R$ we set $a \wedge b \defas \min\setof{a, b}$ and $a \vee b \defas \max\{a, b\}$.  We let $\Id \in \R^{d \times d}$  be the identity matrix, and   $\id(x) \defas x$ stands for the identity map on $\R^d$. We define   the subsets  $SO(d) \defas \setof{A \in \R^{d \times d} \colon A^T A = \Id, \, \det A = 1  }$, $GL^+(d) \defas \setof{F \in \R^{d \times d} \colon \det(F) > 0}$,  and  $\R^{d \times d}_\sym \defas \setof{A \in \R^{d \times d} \colon A^T = A}$. Furthermore, the inverse of the transpose of $F$ will be shortly written as $F^{-T} \defas (F^{-1})^T=(F^T)^{-1}$.  The symbol \EEE $\vert A \vert$   stands for   the Frobenius norm  of a matrix $A \in \R^{d\times d}$,   and ${\rm sym}(A) = \frac{1}{2}(A^T + A)$ and ${\rm skew}(A) = \frac{1}{2}(A - A^T)$ indicate the symmetric and skew-symmetric part, respectively.   By $\delta_{ij}$ we denote \EEE the Kronecker delta function.
We write the scalar product between vectors, matrices, or 3rd-order tensors as $\cdot$, $:$, and $\cdddot$, respectively.
Given $a, \, b \in \R^d$, the symmetrized tensor product is defined as $a \odot b = (a \otimes b + b \otimes a)/2$, where
$a \otimes b = a b^T  \in \R^{d \times d} $.
We use $\{e_1,e_2,e_3 \}$ for standard unit vectors in $\R^3$.
As usual, in the proofs, a generic  constant \EEE $C$ may vary from line to line.
 In the following,  $0 < c_0 <C_0 < \infty$ denote  fixed \EEE   constants.

\subsection{The three-dimensional  setting\EEE}
Given a bounded Lipschitz domain $\Omega' \subset \R^2$ and \ZZZ a \EEE  thickness \EEE $h \GGG > 0 \EEE$, the reference configuration of a 3-dimensional thin \ZZZ plate \EEE is denoted by $\Omega_h \defas \Omega' \times (-h/2, h/2)$.
Let $\Gamma'_D \subset \Gamma' \defas \partial \Omega'$ be an open subset.
We then set $ \Gamma_h \defas \Gamma' \times (-h/2, h/2)$ and $\Gamma_D^h \defas \Gamma'_D \times (-h/2 , h/2)$.
We will prescribe Dirichlet boundary conditions on $\Gamma_D^h$.
More precisely, given $\OOO p > 4\EEE$, the space of \emph{admissible deformations} of the thin plate is given by
\begin{equation}\label{admissibledeformations}
{   \mathcal{W}^h_\id  \defas \{ w \in W^{2, p}(\Omega_h;\R^3) \colon w = \id \text{ on } \Gamma_D^h \}.}
\end{equation}
 By $T > 0$ we denote a fixed time horizon and shortly write $I \defas [0,T]$ for the time interval. We start by introducing a  variational model of \emph{thermoviscoelasticity} \ZZZ studied \EEE in \cite{RBMFMK, MielkeRoubicek20Thermoviscoelasticity}, where in the present setting the space dimension is chosen to be $d = 3$.

\subsection*{Mechanical and coupling energy}
The \emph{mechanical energy} associated \ZZZ to \EEE a deformation $w \in \mathcal{W}^h_\id$ is given by
\begin{equation}\label{mechanical}
\mathcal{M}_h\EEE(w)\defas \int_{\Omega_h} W^{\rm el} (\nabla w(x) )  + \EEE H(\nabla^2 w(x)) \dx.
\end{equation}
The above energy depends on an elastic potential $W^{\rm el}$, as well as on a strain gradient term $H$, adopting the concept of 2nd-grade nonsimple materials, see \cite{Toupin62Elastic, Toupin64Theory}. Given $\OOO p>4\EEE$, the \emph{elastic potential} $\elpot \colon GL^+(d) \to \R_+$ satisfies usual assumptions in nonlinear (hyper-)elasticity, i.e., we require:
\begin{enumerate}[label=(W.\arabic*)]
  \item \label{W_regularity} $\elpot$ is continuous and $C^3$ in a neighborhood of $SO(3)$.
  \item \label{W_frame_invariace} Frame indifference: $\elpot(QF) = \elpot(F)$ for all $F \in GL^+(3)$ and $Q \in SO(3)$.
  \item \label{W_lower_bound} Lower bound: $W^{\rm el}(F) \ge \ac \big(|F|^2 + \det(F)^{-q}\big) - \aC$ for all $F \in GL^+(3)$, where $q \ge \frac{3p}{p-3}$.
  \item \label{W_lower_bound_spec} $\elpot(F) \geq \ac \dist^2(F, SO(3))$ for all $F \in GL^+(3)$ and $\elpot(F) = 0$ if $F \in SO(3)$.
\end{enumerate}

The \emph{strain gradient  energy} term $\hypot \colon \R^{3 \times 3 \times 3} \to \R_+$ has the following properties:
\begin{enumerate}[label=(H.\arabic*)]
  \item \label{H_regularity} $\hypot$ is convex and $C^1$.
  \item \label{H_frame_indifference} Frame indifference: $\hypot(QG) = \hypot(G)$ for all $G \in \R^{3 \times 3 \times 3}$ and $Q \in SO(3)$.
  \item \label{H_bounds2} $H(0) = 0$.
  \item \label{H_bounds} $\ac \abs{G}^p \leq H(G) \leq \aC (1+ \abs{G}^p)$ and $\abs{\pl_G H(G)} \leq \aC \abs{G}^{p-1}$ for all $G \in \R^{3 \times 3 \times 3}$.
\end{enumerate}

As described in the introduction, the energy also depends on a temperature variable $\vartheta \in L^1_+(\Omega_h)$.
We now introduce a \emph{coupling energy} $ \cplen \EEE \colon \Wid \times L^1_+(\Omega_h) \to \R$ given by
\begin{equation*}
  \cplen(w, \vartheta) \defas \int_{\Omega_h} \cplpot(\nabla w, \vartheta) \di x,
\end{equation*}
where $\cplpot \colon GL^+(3) \times \R_+ \to \R$ describes mutual interactions of mechanical and thermal effects, see e.g.~\cite{GurtinFriedAnand10Themechanics}, and satisfies:
\begin{enumerate}[label=(C.\arabic*)]
  \item \label{C_regularity} $\cplpot$ is continuous  and   $C^2$ in $GL^+(3) \times (0, \infty)$\EEE.
  \item \label{C_frame_indifference} $\cplpot(QF, \vartheta) = \cplpot(F, \vartheta)$ for all $F \in GL^+(3)$, $\vartheta \geq 0$, and $Q \in SO(3)$.
  \item \label{C_zero_temperature} $\cplpot(F, 0) = 0$ for all $F \in GL^+(3)$.
  \item \label{C_lipschitz} $|\cplpot(F,\vartheta) - \cplpot(\tilde{F}, \vartheta)| \le \aC(1 + |F| + |\tilde{F}|)|F - \tilde{F}|$ for all $F, \, \tilde F \in GL^+(3)$, and $\vartheta \geq 0$.
  \item \label{C_bounds} For all $F \in GL^+(3) $ and $\vartheta \OOO > 0\EEE$ it holds that
	\begin{align*}
		\abs{\partial_F^2 W^{\rm cpl}(F,\vartheta)} & \le \aC, &
		\abs{\pl_{F \vartheta} \cplpot(F, \vartheta)}
			& \leq \frac{\aC(1+|F|)}{\vartheta \vee 1}, &
		\ac & \leq -\vartheta \pl_\vartheta^2 \cplpot(F, \vartheta) \leq \aC.
	\end{align*}
	  \item \label{C_heatcap_cont} The \emph{heat capacity} $c_V(F, \vartheta) \defas - \vartheta \pl_\vartheta^2 \cplpot(F, \vartheta)$ for $F \in GL^+(3)$ and $\vartheta > 0$ as well as $\pl_{F \vartheta} \cplpot$ can be continuously extended to $GL^+(3) \times \R_+$. 
\end{enumerate}
Notice that, by \ref{C_zero_temperature} and the second bound in \ref{C_bounds}, $\pl_F \cplpot$ can be continuously extended to zero temperatures via $\pl_F W^{\rm cpl}(F, 0) = 0$. For $F \in GL^+(3)$ and $\vartheta\geq 0$, we define the \emph{total free energy potential}
\begin{align}\label{eq: free energy}
  \felpot(F, \vartheta) \defas \elpot(F) + \cplpot(F, \vartheta).
\end{align}

\subsection*{Internal   energy}
We define the \textit{internal energy} $\inten \colon GL^+(3) \times (0, \infty) \to \R$ as
\begin{equation}\label{Wint}
  \inten(F, \vartheta) \defas \cplpot(F, \vartheta) - \vartheta \pl_\vartheta \cplpot(F, \vartheta).
\end{equation}
By the third bound in \ref{C_bounds}, the internal energy is controlled by the temperature in the  sense  \EEE
\begin{align}\label{sec_deriv}
  \partial_{\vartheta} \inten (F, \vartheta)
  = -\vartheta \pl_\vartheta^2 \cplpot(F, \vartheta) \in [\ac, \aC]
  \qquad \text{for all $F \in GL^+(3)$ and $\vartheta > 0$}
\end{align}
which along with \ref{C_zero_temperature} yields
\begin{equation}\label{inten_lipschitz_bounds}
  \ac \vartheta \leq \inten(F, \vartheta) \leq \aC \vartheta.
\end{equation}
 Then,  we can continuously extend  $\inten$ \EEE by setting $\inten(F, 0) = 0$ for all $F \in GL^+(3)$.

 
\subsection*{Dissipation mechanism}
 We introduce a \emph{potential of dissipative forces}  \EEE $\disspot \colon  GL^{\OOO +\EEE}(3) \EEE \times \R^{3 \times 3} \times \R_+ \to \R_+$  satisfying
\begin{enumerate}[label=(D.\arabic*)]
  \item \label{D_quadratic} $\disspot(F, \dot F, \vartheta) \defas \frac{1}{2} D(C, \vartheta)[\dot C, \dot C] \defas \frac{1}{2} \dot C : D(C, \vartheta) \dot C$, where $C \defas F^T F$, $\dot C \defas \dot F^T F + F^T \dot F$, and $D \in C(\R^{3 \times 3}_\sym \times \R_+; \R^{3 \times 3 \times 3 \times 3})$ with $D_{ijkl} = D_{jikl}= D_{klij}$ for $1 \le i,j,k,l \le 3$.
  \item \label{D_bounds} $\ac \abs{\dot C}^2 \leq \dot C : D(C, \vartheta) \dot C \leq \aC \abs{\dot C}^2$ for all $C, \, \dot C \in \R^{3 \times 3}_\sym$, and $\vartheta \geq 0$.
\end{enumerate}
The associated \emph{dissipation functional} $ \diss \EEE \colon \Wid \times H^1(\Omega_h) \times  L^1_+(\Omega_h)  \to \R_+$  defined on time-dependent deformations and temperatures \EEE is given by
\begin{equation*}
  \diss(w,  \partial_t \EEE w, \vartheta)
  \defas \int_{\Omega_h} \disspot(\nabla w,   \partial_t \EEE \nabla w, \vartheta) \di x.
\end{equation*}
\EEE Notice that the fact that $\disspot$ can be rewritten as a function depending on the right Cauchy-Green tensor $C = F^T F$ and its time derivative $\dot C$ is equivalent to \emph{dynamic frame indifference}, see e.g.~\cite{Antmann98Physically}. As a consequence  of \EEE \ref{D_quadratic},  we can express the viscous stress tensor $\partial_{\dot F} R$ as \EEE
\begin{equation}\label{chain_rule_Fderiv}
  \pl_{\dot F} \disspot(F, \dot F, \vartheta) = 2 F (D(C, \vartheta) \dot C),
\end{equation}
see e.g.~\cite[Equation~(2.8)]{RBMFMK} for further details.
Although   $\pl_{\dot F} \disspot(F, \dot F, \vartheta)$ is linear in the time derivative $\dot C$, we stress that nonlinearities arise due to $D(C,\vartheta)$, $F$, and $\dot C$ itself.
Multiplying the viscous stress tensor from the right with the strain rate yields the \emph{dissipation rate} $\drate \colon \R^{3 \times 3} \times \R^{3 \times 3} \times \R_+ \to \R_+$  \OOO given by \EEE
\begin{equation}\label{diss_rate}
  \drate(F, \dot F, \vartheta)
  \defas \pl_{\dot F} \disspot(F, \dot F, \vartheta) : \dot F
  = D(C, \vartheta) \dot C : \dot C = 2 R(F, \dot F,\vartheta),
\end{equation}
 where the last two identities  follow from \EEE    \eqref{chain_rule_Fderiv} and the symmetries in \ref{D_quadratic}.

\subsection*{Heat flux and heat conductivity}  The  \emph{heat \OOO transfer}  is \OOO governed by the Fourier   law in an Eulerian description and transformed to a Lagrangian formulation via a pull-back operator, see  \cite[Equation~(2.24)]{MielkeRoubicek20Thermoviscoelasticity}. \EEE More precisely, given  the  \emph{heat-conductivity tensor} $\hc \colon  \R_+ \to \R^{3 \times 3}_\sym$,  \GGG  a  deformation $w$, and temperature $\vartheta$, \EEE      \OOO the heat flux \EEE $q$ \ZZZ is given by \EEE 
$  q = \EEE -  \ZZZ \hcm( \EEE\nabla w, \vartheta) \nabla \vartheta$, where  \ZZZ $\hc$ \EEE is transformed to Lagrangian coordinates by \EEE
\begin{equation}\label{hcm}
 \ZZZ  \hcm( \EEE F, \vartheta) \defas \det(F) F^{-1} \hc( \vartheta) F^{-T}.
\end{equation}
 Here, \EEE we assume that $\hc$ is continuous, symmetric, uniformly positive definite, and bounded,  i.e., \EEE for all $\vartheta \geq 0$ \ZZZ  we \EEE require that
\begin{equation}\label{spectrum_bound_K}
  \ac \leq \hc(\vartheta) \leq \aC,
\end{equation}
where the inequalities are meant in the eigenvalue sense. 

\ZZZ We highlight again that the described model coincides with the one studied in \cite{RBMFMK, MielkeRoubicek20Thermoviscoelasticity}. The only difference is that we  neglect $x$-dependence  of $\hc$ for simplicity and  we ask for the stronger assumption $p>4$ instead of   $p > 3$. In principle, our arguments  would also work for $p>3$ at the expense of an $h$-dependent prefactor for $H$ in \eqref{mechanical}\OOO, see also \cite[Equation (2.14)]{FK_dimred} \ZZZ \GGG for \QQQ a model \GGG in the isothermal case up to a change of notation. \EEE As this  would lead to heavier notation throughout the paper, we refrain from treating the case $3 < p \le 4$. \EEE

\subsection*{Equations of nonlinear thermoviscoelasticity} 
 We are now in the position to formulate the system of PDEs for which we intend to perform a dimension reduction. We consider the coupled system \EEE
\begin{align}\label{quasistaticequation3D}
-{\rm div}\Big( \partial_F W(\nabla w,\vartheta) +   {\rm div} (\partial_G H(\nabla^2 w)) + \partial_{\dot{F}}R(\nabla w, \partial_t \EEE \nabla w, \vartheta)  \Big) =  g^{3D}_h e_3 \qquad \text{in $I \times \Omega_h$,}
\end{align}
\begin{align}\label{thermoequation3D}
c_V(\nabla w, \vartheta)  \partial_t \EEE \vartheta = {\rm div} \big(   \hcmnoh (\nabla w, \vartheta)  \nabla \vartheta \big) +  \xi(\nabla w,  \partial_t \EEE \nabla w, \vartheta) +\vartheta \partial_{F \vartheta} W^{\rm cpl}(\nabla w,\vartheta) :  \partial_t \EEE \nabla w  \quad  \text{ in $I \times \Omega_h$},
\end{align}
 where $ g^{3D}_h \EEE \colon I \times \Omega_h \to \R$ denotes a time-dependent \emph{body-force} acting vertically on the material.  The \emph{mechanical equation }\eqref{quasistaticequation3D} is a quasistatic version of the Kelvin-Voigt rheological model \EEE (neglecting inertia), and corresponds to the sum of elastic and viscous stress. \ZZZ As it is customary for dimension-reduction problems in the von Kármán regime \cite{FJM_hierarchy}, we focus on purely  vertical body forces,  referring to \cite{lecumberry} for a thorough discussion of other scenarios. \EEE
 
%
%

 The \emph{heat-transfer equation} \eqref{thermoequation3D}  \EEE is derived from the  entropy equation
\begin{align*}
\vartheta  \partial_t \EEE s = \xi - \diver q \qquad \text{in $I \times \Omega_h$},
\end{align*}
where $s= - \partial_\vartheta W^{\rm cpl}(\nabla w,\vartheta)$ denotes the \emph{entropy} \OOO and \EEE $\xi$ the dissipation rate  introduced in \eqref{diss_rate}\OOO. \EEE  The term  $c_V(\nabla w, \vartheta) = -\vartheta \partial_{\vartheta^2} W^{\rm cpl}(\nabla w,\vartheta)$ defined in \ref{C_heatcap_cont}  corresponds to the \emph{heat capacity}    and the last term in \eqref{thermoequation3D} is an \emph{adiabatic  term},  playing the role \ZZZ of \EEE a heat source or sink, respectively. \EEE
We  refer to \cite{MielkeRoubicek20Thermoviscoelasticity} or to \cite[Section 8.1]{KruzikRoubicek19Rate-independent}   for \OOO further \EEE details.
Notice that \GGG the \EEE purely mechanical stored energy  $W^{\rm el}$, see \eqref{mechanical},  does not influence the heat production \OOO nor the \EEE transfer in \eqref{thermoequation3D}. \EEE

We complete the above  equations \ZZZ by \EEE boundary conditions on  $I \times \Gamma_h$. \EEE Besides Dirichlet boundary conditions on $ I \times  \Gamma^h_D\EEE $, see \eqref{admissibledeformations}, we assume zero Neumann boundary conditions \GGG for the stress and hyperstress \EEE on $ I \times \EEE (\partial \Omega_h \setminus  \Gamma^h_D)\EEE $ \GGG since no surface forces are applied. \EEE \GGG Due to  the second deformation gradient, there arise additional natural Neumann conditions on $I \times \partial \Omega_h$  \EEE and, for \EEE the heat flux, we suppose that  \EEE  $-  \hcmnoh \EEE(\nabla w, \vartheta) \nabla \vartheta \cdot \nu = \kappa (\vartheta -  \vartheta_\flat^h\EEE)$  on $I \times \Gamma_h$. \ZZZ Here, \EEE $\nu$ is the outward pointing unit normal on $ \GGG \Gamma_h \EEE$, $\kappa \geq 0 $ is a phenomenological \emph{heat-transfer coefficient}, and $\vartheta_\flat^h\EEE \in \ZZZ L^{2}(I;L^2_+(\Gamma_h)) \EEE $  denotes an \emph{external temperature}.   We refer to  \cite[Equation (2.14)]{MielkeRoubicek20Thermoviscoelasticity} for details. Note that we do not assign boundary conditions at the top and the bottom of $\Omega_h$, i.e., on $ \Omega' \EEE \times \{-h, h\}$, neither for the deformations nor for the temperature.

\subsection*{Weak formulation in three dimensions}
  We now introduce the weak formulation of \eqref{quasistaticequation3D}--\eqref{thermoequation3D}. For purely technical reasons \OOO arising in \EEE the derivation of a priori estimates, \EEE
we introduce truncated versions of the dissipation rate by

\begin{align}\label{diss_rate_truncated}
\xi^{(\alpha)}(F, \dot F, \vartheta) \defas \begin{cases}
\xi(F,\dot F,\vartheta) &\xi \leq 1, \\
\xi(F,\dot F,\vartheta)^{\alpha/4} &\text{otherwise},
\end{cases}
\end{align}
 where the parameter $\alpha \in [2,4]$ is related to the scaling exponent of the temperature, see \ZZZ the discussion preceding \eqref{eq:mech2Dintro}.  \EEE
 We emphasize that no regularization is applied in the case $\alpha = 4$ \ZZZ whereas  for \EEE $\alpha \in [2,4)$ the dissipation is changed for large strain rates. Since in the von Kármán regime we deal with small strains and strain rates, we  heuristically \EEE have $\xi \le 1$ and  thus this regularization does essentially not   affect the system.  In particular, it has no influence on the effective model in   \eqref{eq:temperature2Dintro} \ZZZ as the latter is deduced from a linearization at $F = \Id$ and $\dot{F} = 0$. \EEE

\begin{definition}[Weak solution of the nonlinear system]\label{def:weak_formulation}
\ZZZ Consider \EEE initial values $w_{0}^h \in \mathcal{W}^h_\id$ and $\vartheta_{0}^h \in L^2_+(\Omega_h)$, \ZZZ and data $g^{3D}_h \in W^{1,1}(I;L^2( \Omega_h))$ and $\vartheta_\flat^h  \in   L^{2}(I;L^2_+(\Gamma_h))$. Then, \EEE a pair $(w^h, \vartheta^h) \colon I \times \Omega_h \to \R^3 \times \R$ is called a \emph{weak solution} to \eqref{quasistaticequation3D} and \eqref{thermoequation3D} with associated natural boundary conditions  if \EEE $w^h \in L^\infty(I; \Wid) \cap H^1(I; H^1(\Omega_h; \R^3))$ with $w^h(0, \cdot)=  w^h_0$, $\vartheta^h \in L^1(I; W^{1,1}(\Omega_h))$ with $\vartheta^h \ge 0$ a.e., and if the following \ZZZ equations \EEE are satisfied:
\begin{subequations}
\begin{equation}\label{weak_form_mech_res_unrescaled}
\begin{aligned}
  &\int_I \int_{\Omega_h}  \Big(
      \pl_F \felpot(\nabla w^h, \vartheta^h)
      + \pl_{\dot F} \disspot(\nabla w^h,  \partial_t \EEE \nabla w^h, \vartheta^h)
    \Big) : \nabla {\varphi_w} + \pl_G
    \hypot(\nabla^2 w^h) \cdddot \nabla^2 {\varphi_w}
     \di x \di t \\
    &\quad =  \int_I \int_{\Omega_h} g_h^{3D}  ({\varphi_w})_3 \EEE \di x \di t
\end{aligned}
\end{equation}
for any test function ${\varphi_w} \in C^\infty(I \times   \overline{\Omega_h}; \EEE \R^3)$ with ${\varphi_w} = 0$ on $I \times \Gamma_D^h$, as well as
\begin{equation}\label{weak_form_heat_unrescaled}
\begin{aligned}
  &\int_I \int_{\Omega_h}
     \hcmnoh \EEE(\nabla w^h, \vartheta^h) \nabla \vartheta^h \cdot \nabla {\varphi_{\vartheta}}
    - \Big(
      \xi^{(\alpha)}(\nabla w^h,  \partial_t \EEE \nabla w^h, \vartheta^h)
      + \partial_F W^{\rm{cpl}}(\nabla w^h, \vartheta^h) :  \partial_t \EEE \nabla w^h
    \Big) {\varphi_{\vartheta}} \di x \di t \\
  & \ \, - \int_I \int_\Omega W^{\rm{in}}(\nabla w^h, \vartheta^h)   \partial_t \EEE {{\varphi_{\vartheta}}}  \di  x \di t
  - \int_{\Omega_h} W^{\rm in}(\nabla w_{0}^h, \vartheta_{0}^h) {\varphi_{\vartheta}}(0) \di x \GGG = \kappa \int_I \int_{\Gamma_h} ( \vartheta_\flat^h -  \vartheta^h )\EEE {\varphi_{\vartheta}} \di \haus^2  \di t  \EEE \\ 
\end{aligned}
\end{equation}
for any test function ${\varphi_{\vartheta}} \in C^\infty(I \times  \overline{\Omega_h})\EEE$ with ${\varphi_{\vartheta}}(T) = 0$.
\end{subequations}
\end{definition}

\ZZZ The identities \EEE \eqref{weak_form_mech_res_unrescaled} and \eqref{weak_form_heat_unrescaled} arise naturally from the classical formulation \eqref{quasistaticequation3D}--\eqref{thermoequation3D}  by replacing the \EEE dissipation rate $\xi$   with \EEE its regularized version $\xi^{(\alpha)}$. Then, one can indeed show that sufficiently regular weak solutions  coincide with solutions to \eqref{quasistaticequation3D}--\eqref{thermoequation3D} \EEE along with the   imposed  \EEE  boundary conditions,  we refer to  \cite{MielkeRoubicek20Thermoviscoelasticity} for details. \EEE Existence of weak solutions described above was already shown in \GGG \cite[Theorem~2.2]{MielkeRoubicek20Thermoviscoelasticity} and \EEE \cite[Proposition~2.5(ii)]{RBMFMK}, \GGG where the latter result \GGG explicitly \EEE takes the truncation of $\xi$ into account. \EEE  (Note that in \cite{RBMFMK} the parameter $\alpha$ is \ZZZ replaced \EEE by $\alpha/2$.) \EEE
 
\begin{proposition}[Existence of weak solutions]
    For any $h > 0$ \ZZZ there \EEE exists a weak solution $(w^h, \vartheta^h)$ to \eqref{quasistaticequation3D} and \eqref{thermoequation3D} in the sense of Definition \ref{def:weak_formulation}.
\end{proposition} 
We stress that the  above \EEE existence of weak solutions in the large-strain setting is not reliant on the regularized dissipation rate $\xi^{(\alpha)}$.
However, for the rigorous  dimension reduction, \EEE  the regularization seems to be unavoidable.

\subsection{Rescaling  to \GGG a \EEE fixed domain} As customary in dimension-reduction problems, it is convenient to reformulate the model on a fixed domain. \EEE 
We shortly write $\Omega \coloneqq \Omega_1 = \Omega' \times (-1/2, 1/2)$, $\Gamma_D \defas \Gamma_D^1$, and $\Gamma \coloneqq \GGG \Gamma_1 \EEE$.
For any $x = (x_1, x_2, x_3) \in \R^3$, we  write \EEE  $x'$ for the first two components $(x_1, x_2)$ of $x$.
Given $h > 0$, a deformation $w^h \colon I \times \Omega_h \to \R^3$, and  a \EEE temperature $\vartheta^h \colon I \times  \Omega_h \EEE \to \R_+$, we denote by $y^h \colon I \times \Omega \to \R^3$ the rescaled deformation and by $\theta^h \colon I \times \Omega \to \R_+$ the rescaled temperature, defined by
\begin{align*}
  y^h(t, x', x_3) &\defas w^h(t, x', h x_3), & \theta^h(t, x', x_3) &\coloneqq \vartheta^h(t, x', hx_3).
\end{align*}
 Similarly, we define the rescaled initial  data, external temperature, and body force  by   
\begin{align}\label{eq: data}
 y^h_0(x', x_3)  &\defas w^h_0(x', h x_3),  &  \theta^h_0(x', x_3) &\defas \vartheta^h_0(x', hx_3),\notag \\ 
     f_h^{3D}(t, x) &\defas g_h^{3D}(t, x', hx_3), & \theta_\flat^h\EEE &\defas \vartheta_\flat^h\EEE(t, x', h x_3). \EEE
\end{align}
\EEE The set of admissible configurations, see \eqref{admissibledeformations}, takes the form
\begin{align}\label{eq: nonlinear boundary conditions}
\mathscr{S}^{3D}_h = \Big\{ y \in W^{2,p}(\Omega;\R^3) \colon y(x',x_3) = \begin{pmatrix} x' \\ hx_3 \end{pmatrix}
\text{ for } x \in \Gamma_D \Big\}.
\end{align} 
 In the isothermal case \cite{FK_dimred},  slightly more general boundary conditions are considered. \EEE
We restrict ourselves to functions that coincide with $\id$ at the boundary as this is the setting for which existence results  in \GGG three-dimensional \EEE nonlinear thermoviscoelasticity are available, see  \EEE \cite{RBMFMK, MielkeRoubicek20Thermoviscoelasticity}. For a smooth function $y\colon \Omega \to \R^3$, the scaled gradient  of $y$ is given by $\nabla_h y = ({y,}_1, {y,}_2, \tfrac{{y,}_3}{h})$, where the subscript indicates the directional derivative along the $i$-th unit vector.
Moreover, $\nabla_h^2$ denotes the scaled Hessian defined by
\begin{align*}
  (\nabla_h^2y)_{ijk} &\defas h^{-\delta_{3j} - \delta_{3k}} (\nabla^2y)_{ijk} \quad\text{for $i,j,k \in \{1,2,3\}$}, & (\nabla^2 y)_{ijk} &\defas (\nabla^2 y_i)_{jk}.
\end{align*}
In order to avoid possible confusion,  we \EEE denote the gradient and  Hessian \EEE  of  functions \EEE defined on the \GGG two-dimensional \EEE domain \EEE $\Omega'$ by $\nabla'$ and $(\nabla')^2$, respectively.   For convenience, we denote the mechanical energy of the rescaled deformation by
\begin{align*}
 \mathcal{M}(y^h) :=      \int_\Omega W^{\rm el} (\nabla_h y^h  ) \dx +  \int_\Omega H(\nabla_h^2 y^h) \dx. 
 \end{align*}

\begin{remark}[Weak formulation \ZZZ of \EEE the rescaled problem]\label{rem:weakformresc}
 If $(w^h,\vartheta^h)$ are solutions in the sense of Definition~\ref{def:weak_formulation} for initial values $(w_0^h, \vartheta_0^h)$, the rescaled pair $(y^h, \theta^h)$ satisfies the identities
\begin{subequations}
\begin{equation} \label{weak_form_mech_res}
\begin{aligned}
  &\int_I \int_\Omega
      \partial_F W(\nabla_h y^h, \theta^h) : \nabla_h {\varphi_{y}}
      + \partial_G H(\nabla^2_h y^h) \cdddot \nabla^2_h {\varphi_{y}}
      + \partial_{\dot F} R(\nabla_h y^h,   \partial_t \EEE  \nabla_h \EEE  y^h, \theta^h) : \nabla_h {\varphi_{y}}
    \di x \di t \\
  &\quad= \int_I \int_\Omega  f^{3D}_h  ({\varphi_{y}})_3 \EEE \di x \di t
\end{aligned}
\end{equation}
for all ${\varphi_{y}} \in C^\infty(I \times \overline{\Omega};\R^3)$ with ${\varphi_{y}} = 0$ on $I \times \Gamma_D$ and
\begin{align}\label{weak_form_heat_res} 
  &\int_I \int_\Omega
    \mathcal{K}(\nabla_h y^h, \theta^h) \nabla_h \theta^h \cdot \nabla_h {\varphi_{\theta}}
    - \Big(
      \xi^{(\alpha)}(\nabla_h y^h,  \partial_t \EEE \ZZZ \nabla_h \EEE y^h, \theta^h)
      + \partial_F W^{\rm{cpl}}(\nabla_h y^h, \theta^h) :  \partial_t \EEE \ZZZ \nabla_h \EEE  y^h 
    \Big) {\varphi_{\theta}} \di x \di t \notag \\
  & \quad - \int_I \int_\Omega
    W^{\rm{in}}(\nabla_h y^h,\theta^h)  \partial_t \EEE {\varphi_{\theta}} \di x \di t
    -  \int_\Omega W^{\rm in}(\nabla_h y_0^h, \theta_0^h) {\varphi_{\theta}}(0) \di x   = \kappa \int_I \int_\Gamma (\theta_\flat^h - \theta^h )\EEE {\varphi_{\theta}} \di \haus^2 \di t   
\end{align}
\end{subequations}
for all ${\varphi_{\theta}}\in C^\infty( I\times \overline{\Omega})$ with ${\varphi_{\theta}}(T) = 0$, where  we refer to \eqref{eq: data} for the definition of the rescaled data. \EEE  
\end{remark}

\subsection{Compactness and limiting variables}\label{a new subsection}
The limiting variables are identified via a compactness argument. Following \cite{FJM_hierarchy}, we derive compactness for the rescaled deformations $(y^h)_h$ in terms of averaged  and scaled \ZZZ in-plane \EEE and out-of-plane displacements, denoted by
\begin{align}\label{def:displacements}
  u^h(t, x')
  &\defas \frac{1}{h^2} \int_{-\frac{1}{2}}^{\frac{1}{2}} \left( \ZZZ \begin{pmatrix}
	y_1^h(t,x',x_3) \\
	y_2^h(t,x',x_3)
	\end{pmatrix} \EEE - x' \right)\, \di x_3, &
  v^h(t, x')
  &\defas \frac{1}{h} \int_{-\frac{1}{2}}^{\frac{1}{2}} y^h_3(t, x', x_3) \di x_3.
\end{align}
 Here, the different scaling in terms of $\frac{1}{h^2}$ and $\frac{1}{h}$ corresponds to the  von K\'arm\'an scaling regime. We consider different regimes for the temperature in terms of an exponent $\alpha \in [2,4]$: given      temperatures $(\theta^h)_h$,  we define  the averaged and scaled temperature as \EEE
\begin{align}\label{def:scaledtemp}
  \mu^h(t, x')
  &\defas \frac{1}{h^\alpha}  \int_{-1/2}^{1/2} \EEE \theta^h(t,  x', x_3) \di x_3 \EEE.
\end{align}
 As mentioned in the introduction, the  definition of  $\mu^h$ corresponds to a linearization around temperature zero. 
\OOO For the sake of simplicity,  we consider external body forces $f^{3D}_h$ independent  \ZZZ  of $x_3$, and \OOO require \EEE  
\begin{align}
\sup_{h > 0} h^{-3} \Vert f^{3D}_h \Vert_{W^{1,1}(I;L^2(\ZZZ \Omega') \EEE )} < \infty \quad \text{ and } \quad
\ZZZ \frac{1}{h^3} \EEE f_h^{3D} \to f^{2D} \quad \text{strongly in } L^2(I \times \ZZZ \Omega') \EEE \label{forcebound} \tag{E.1} 
\end{align}
for some $f^{2D} \in L^2 (I \times\Omega')$. \EEE
Finally, we define  the averaged scaled \ZZZ external \EEE temperature as 
\begin{equation*}
    \mu_\flat^h(t, x') \defas h^{-\alpha}  \int_{-1/2}^{1/2}  \theta_\flat^h\EEE(t, x', x_3) \di x_3,
\end{equation*}
 and \OOO suppose  that there exist $\mu_\flat \in L^2(I; \L^2_+(\Gamma'))$ such that
\begin{align}
\sup_{h > 0} h^{-\alpha} \lVert\theta_\flat^h\EEE\rVert_{\ZZZ L^{2}(I;L^2_+(\Gamma)) \EEE}< \infty \quad \text{ and }\quad
\mu_\flat^h \rightharpoonup \mu_\flat \qquad \text{weakly in }  L^2(I \times \Gamma') \EEE. \label{assumption:temp} \tag{E.2}
\end{align}
\EEE
 
 \begin{proposition}[Compactness]\label{prop:compactness}
Suppose that $ \sup\nolimits_{h>0} \GGG ( \EEE h^{-4}\mathcal{M}(y_0^h) + \ZZZ  h^{-2\alpha} \EEE \Vert \theta_0^h \Vert_{L^{ 2\EEE}(\Omega)}^{ 2 \EEE} \GGG ) \EEE  < + \infty $, \ZZZ and that \eqref{forcebound}--\eqref{assumption:temp} hold. \EEE   Then, there exists a sequence of weak solutions    $((y^h,\theta^h))_h$   to \eqref{weak_form_mech_res} and \eqref{weak_form_heat_res} in the sense of Definition \ref{def:weak_formulation} and   in-plane and out-of-plane   displacements \EEE $u \in \ZZZ  H^1(I;H^1(\Omega';\R^2))$ and $v \in \ZZZ H^1(I; H^2(\Omega'))$\OOO, respectively, \EEE satisfying
\begin{align}\label{eq: bccond}
    u(t, x') &=0, & v(t,x') &= 0, & \nabla' v(t,x') &= 0
             & &\text{for almost every } (t,x') \in I \times \Gamma_D'
\end{align}
such that, up to selecting a subsequence,  the mappings $u^h$ and  $v^h$  defined in \eqref{def:displacements}  satisfy \EEE  
\begin{subequations}
\begin{align}
u^h &\weaklystar u \qquad  \quad \; \text{weakly* in }  L^\infty(I; H^1(\Omega'; \R^2)), \label{convuh} \\
 \partial_t \EEE u^h &\rightharpoonup  \partial_t \EEE u \qquad \ \text{weakly in }  L^s(I; W^{1, s}(\Omega';\R^2)), \label{convuhdot} \\
v^h & \weaklystar v \qquad \quad \; \text{weakly* in } L^\infty(I; H^1(\Omega')), \label{convvh} \\
 \partial_t \EEE v^h &\rightharpoonup  \partial_t \EEE v \qquad \ \text{weakly in } L^2(I;H^1(\Omega')), \label{convvhdot}
\end{align}
\end{subequations}
for $s= 1+ (3-8/p)^{-1} \in [1,2)$. Moreover, there  exists a \GGG temperature \EEE $\mu \in L^q(I \times \Omega')$ \GGG with \EEE $\nabla' \mu \in L^r(I \times \Omega'; \R^2)$ for any $q \in [1, 5/3)$ and $r \in [1, 5/4)$ such that, up to selecting a further subsequence, \ZZZ the mappings $\mu^h$ defined in \eqref{def:scaledtemp} satisfy \EEE
\begin{subequations}
\begin{align}
\mu^h &\to \mu \qquad \quad \text{strongly in } L^q(I \times \Omega'), \label{convmuh} \\
\mu^h &\rightharpoonup \mu \qquad \quad \text{weakly in } L^r(I; W^{1,r}(\Omega')). \label{convmuhnabla}
\end{align}
\end{subequations}
\end{proposition}
 Note \GGG that \EEE the regularity of the limits $ \GGG  u \EEE$ and $ \GGG v \EEE$ cannot be deduced \GGG directly from \GGG \eqref{convuh}--\eqref{convvhdot}. \EEE
\OOO Instead, we will further \ZZZ exploit \EEE \GGG compactness properties of the rescaled strain and stress, see \EEE Lemma~\ref{lem:conv_strain_stress}. \EEE

\subsection{The two-dimensional model}\label{sec:two-dimensional models}
 The main goal of this paper is to identify a limiting  system \EEE of equations governing the evolution of $u$, $v$, and $\mu$,  see Theorem~\ref{maintheorem} below. We proceed by introducing the effective two-dimensional problem. \EEE  

\subsection*{Effective tensors}\label{sec:quadraticforms}
As a preparation,  we \EEE introduce effective lower dimensional tensors related to $W^{\rm el}$, $R$, and $W^{\rm cpl}$ , respectively. \EEE
We define $Q^3_{W^{\rm el} }\colon \R^{3\times3} \to \R$ and $Q^3_R\colon \R^{3\times3} \to \R$ by
\begin{align}\label{eq:quadraticformsnotred}
 	Q^3_{W^{\rm el}}(A) &\defas \partial^2_{F^2} {W^{\rm el} }(\Id)[A,A],  &
  Q^3_R(A) &\defas 2 R(\Id,A,0)
\end{align}
 for  any \EEE $A \in \R^{3 \times 3}$. The quadratic forms $Q_{\elpot}^3$ and $Q_R^3$ induce fourth-order tensors denoted by $\C_{W^{\rm el}}^3$ and $\C_R^3$, respectively. \ZZZ In particular, recalling  \ref{D_quadratic} it holds that  
 \begin{align}\label{DDDDD}
 \C_R^3 = 4D(\Id,0).
 \end{align} \EEE
 Moreover, we define    
 \begin{align}\label{alpha_dep}
   \mathbb{B}^{(\alpha)} \EEE &= \begin{cases}
    \partial_{F\theta} \cplpot(\Id,0) & \text{if } \alpha=2, \\
    0  & \text{if } \alpha \in (2,4],
  \end{cases} &
   \CD^{3, \alpha}\EEE &= \begin{cases}
    0 & \text{if } \alpha \in [2,4), \\
    \CD^3  & \text{if } \alpha =4,
  \end{cases}
\end{align}
where  the second-order tensor $\mathbb{B}^{(\alpha)}$ plays the role of a \emph{thermal expansion matrix}.    The dependence of the tensors on the scaling parameter $\alpha$ has already been noticed in  \cite[Equation~(2.34)]{RBMFMKLM}. \EEE
 We further define reduced quadratic forms by minimizing among stretches in the vertical direction.
 More precisely, for any  $A \in \R^{2 \times 2}$ \EEE let
 \begin{equation}\label{def:quadraticforms}
 	Q^2_S( A \EEE) \defas \min \left\{Q^3_S( A^* \EEE) \colon  A^* \EEE \in \R^{3\times 3}_\sym, \,  A^*_{ij} \EEE = A_{ij} \, \text{ for } i,j =1,2 \right\} \qquad \text{for }  S \in \{W^{\rm el},R\}.
 \end{equation}\EEE
The associated tensors are denoted by $\C_{W^{\rm el}}^2$ and $\C_R^2$, respectively.
 In a similar fashion, we can also define the reduced tensor $\CD^{2, \alpha}$ associated to $\CD^{3, \alpha}$.  Moreover, we set \EEE
\begin{align}\label{barcvandk}
 \overline c_V \defas c_V(\Id, 0),
\end{align}
 where the definition above is well-defined due to \ZZZ \ref{C_heatcap_cont}. \ZZZ We \EEE shortly write $\mathbb K \defas \mathcal{K}(\Id, 0)$ for the heat conductivity tensor  at the identity and zero \EEE temperature, \ZZZ and introduce \EEE  the effective $2$-dimensional heat conductivity $\tilde{\mathbb{K}}$ \ZZZ as \EEE
 \begin{equation}\label{heattensorredd}
  \tilde{\mathbb{K}} \coloneqq \mathbb{K}''
    - \frac{1}{\mathbb{K}_{33}} \begin{pmatrix} \mathbb{K}_{31} \\ \mathbb{K}_{32} \end{pmatrix} \otimes \begin{pmatrix} \mathbb{K}_{31} \\ \mathbb{K}_{32} \end{pmatrix},
\end{equation}
where  here and in the following \EEE for any matrix $A \in \R^{3 \times 3}$ we denote its upper-left $2\times 2$-minor as $A''$.

By Taylor expansion, polar decomposition, and frame indifference (see \ref{W_regularity}, \ref{W_frame_invariace}\OOO, \EEE and \ref{D_quadratic}) one can observe that all the above  introduced quadratic forms \EEE only depend on the symmetric part of  the strain and strain rate, respectively\EEE. Furthermore, the quadratic forms $Q_S^i$, $i \in \OOO\{\EEE 2,3 \}\EEE$, $S  \in \{\EEE \elpot, R\}\EEE$, are positive definite whenever restricted to symmetric matrices, see \ref{W_lower_bound_spec} and \ref{D_bounds}.  Similarly, \ref{C_frame_indifference} implies that $\mathbb{B}^{(\alpha)} $ is symmetric. \EEE

\subsection*{Compatibility conditions}
In  the  spirit of \EEE  \cite{FK_dimred, MFLMDimension2D1D,MFLMDimension3D1D},  we require some compatibility conditions of the quadratic forms  and their reduced versions. In this regard, w\EEE e assume that we can decompose $Q_S^3$, for $S  \in \{ \EEE \elpot, R  \}\EEE$, in the following way: there exist  quadratic forms $Q_S^*$ \EEE such that for all $A \in \R^{3\times 3}_\sym$ it holds that  
 \begin{align}
 Q_S^3(A) = Q_S^2(A'') + Q_S^*(\tilde A), \tag{F.1} \label{quadraticformsassumption}
 \end{align}
 where $\tilde A \in \R^{3\times 3}_\sym$ is given by $\tilde A_{ij}=0$ for $i,j  \in \{\EEE 1,2\}\EEE$, and $\tilde A_{km}= A_{km} $ for $k=3$ and $m \in \{ \EEE1,2,3 \}\EEE$.

This  is restrictive \EEE from a modeling point of view since the above condition is only satisfied for materials with \textit{ zero Poisson ratio}.  As first noted in  \cite[Section 2.2]{FK_dimred}, such an assumption \EEE is crucial  for \EEE our analysis, as the limiting strain component of the upper-left $2\times 2$ minor  must   not influence the remaining components.   (We will exploit this fact in \EEE Step 1 of the proof of Theorem~\ref{maintheorem}, \GGG see \EEE \eqref{CWCD} below.) \EEE 
For the  very \EEE same reason,  in the case $\alpha = 2$, \ZZZ we \EEE require    that  
\begin{equation}\label{heattensorassumption}
  \mathbb{B}^{(\alpha)} = \begin{bmatrix}
     (\mathbb{B}^{(\alpha)})''  & 0 \\
    0 & 0
  \end{bmatrix}.\tag{F.2}
\end{equation}
Intuitively speaking, this guarantees that the induced stress generated by  changes in \EEE temperature does not  influence \EEE the vertical displacement.

\subsection*{Equations of thermoviscoelasticity for von Kármán plates} \EEE Depending on the scaling $\alpha \in [2,4]$ of the temperature, we derive different limit evolutions for thermoviscoelastic von Kármán plates for a triple $(u, v, \mu)\colon I\times \Omega' \to \R^2 \times \R \times \R_{\OOO +\EEE}$.
Let us first describe \emph{the mechanical equations} in their strong formulation: in $ I \EEE \times \Omega'$, we have  
\begin{equation}\label{eq:mechanical2d}
\left\{
\begin{aligned}
  \diver\Big(\C^2_{W^{\rm el}}\big( e(u)  +  \tfrac{1}{2} \nabla' v \otimes \nabla' v  \big) + \mu (\mathbb{B}^{(\alpha)})'' + \C^2_D \big( e( \partial_t \EEE u) +    \partial_t \EEE \nabla' v   \odot \nabla' v \big) \Big) &= 0, \\
  -\diver\Big(
    \Big(
      \C^2_{W^{\rm el}}\big( e(u)  + \tfrac{1}{2} \nabla' v \otimes \nabla' v  \big)
      + \mu (\mathbb{B}^{(\alpha)})''
      + \C^2_D \big( e( \partial_t \EEE u) +  \partial_t \EEE \nabla' v \odot \nabla' v \big) \Big) \nabla' v
    \Big) \\
  + \tfrac{1}{12}  \diver \diver \Big(
      \C^2_{W^{\rm el}} (\nabla')^2 v + \C^2_D  \partial_t \EEE (\nabla')^2 v
    \Big)
  &= f^{2D},
\end{aligned}
\right.
\end{equation}
where $e(u) \defas \sym (\nabla' u)$ is the symmetrized gradient.  The system is complemented with initial and boundary conditions, \ZZZ namely \EEE
\begin{equation*}
\left\{
\begin{aligned}
  u&= 0, & v &= 0, & \nabla'v &=0 & &\text{on } I \times \Gamma_D', \\  
  u(0,\cdot) &= u_0, &  v(0,\cdot) &= v_0 & &&  &\text{in } \Omega',  
\end{aligned}
\right.
\end{equation*}
for some $u_0\colon I \times \Omega' \to \R^2  $ and $v_0 \colon I \times \Omega' \to \R$, along with a natural Neumann boundary condition on $I \times  \GGG \Gamma' \EEE \setminus \Gamma'_D \EEE$ which we do not state explicitly for convenience.  

%
%
%

With regard to the thermal evolution in $I \times \Omega'$,  we obtain the effective \emph{heat-transfer equation} \EEE 
\begin{equation*}
\begin{aligned}
\overline c_V  \partial_t \EEE \mu
    - \diver(\tilde{\mathbb{K}}  \nabla' \EEE \mu)  =   &  \CD^{2, \alpha} \EEE   \Big(
      \sym( \partial_t \EEE \nabla' u)
      +  \partial_t \EEE \nabla' v \odot \nabla' v
    \Big) : \Big(
      \sym( \partial_t \EEE \nabla' u)
      +  \partial_t \EEE \nabla' v \odot \nabla' v
    \Big)  \\ & \quad  +   \frac{1}{12} \EEE \CD^{2,\alpha} \partial_t \EEE (\nabla')^2 v :  \partial_t \EEE (\nabla')^2 v,
\end{aligned}
\end{equation*}
 complemented with  the initial and boundary conditions \EEE 
\begin{equation}\label{eq:temp2D_bdry_init}
\left\{
\begin{aligned}
  \tilde{\mathbb{K}} \nabla \mu \cdot \nu + \kappa \mu
  &= \kappa \mu_\flat  & &\text{on } I \times \GGG \Gamma', \EEE \\
  \mu(0,\cdot) &= \mu_0  & &\text{in } \Omega', \\
\end{aligned}
\right.
\end{equation}
for some $\mu_0 \colon I \times \Omega' \to \R$,  i.e., \EEE we find a standard linear heat equation with Robin boundary conditions.  Here, by $\nu$  we now denote the outward pointing unit normal on $\GGG \Gamma' \EEE$. \EEE

We will prove that the system \eqref{eq:mechanical2d}--\eqref{eq:temp2D_bdry_init} admits a weak solution.
In fact, we will show that appropriate rescaling of weak solutions of the \GGG three-dimensional \EEE problems (as given in \GGG Remark~\ref{rem:weakformresc} and \EEE Proposition~\ref{prop:compactness})
converge to a weak solution of the two-dimensional problem in a suitable sense.
The latter ones are defined as follows.

\begin{definition}[Weak solution of thermoviscoelastic von Kármán plates]\label{def:weak_form_vK}
\ZZZ Consider \EEE initial values $u_0 \in H^1(\Omega';\R^2) $, $v_0 \in \ZZZ H^2(\Omega') \EEE $, and $\mu_0 \in L^2(\Omega')$, \ZZZ as well as data $f^{2D} \in L^2 (I \times\Omega')$ and  $\mu_\flat \in L^1(I; L^1_+(\Gamma'))$.  \EEE A triple $(u, v, \mu) \colon I \times   \Omega' \EEE \to \R^2 \times \R \times \R_+$ is called a \emph{weak solution} to the initial-boundary-value problem \eqref{eq:mechanical2d}--\eqref{eq:temp2D_bdry_init} \ZZZ if \EEE  
\begin{enumerate}
\item $u \in \ZZZ H^1(I;H^1(\Omega';\R^2)) \EEE $ with $u=0$ a.e.~on $I \times \Gamma_D'$  and $u(0) = u_0$ a.e.~in $\Omega'$\EEE , 
\item $v \in \ZZZ H^1(I; H^2(\Omega') ) \EEE $ with $v = 0$, $\nabla' v = 0$ a.e.~on $I \times \Gamma_D'$  and $v(0) = v_0$ a.e.~in $\Omega'$\EEE ,
\item $\mu \in L^1(I; W^{1,1}( \Omega') )$  with $\mu \geq 0$  a.e.~on $I\times \Omega'$, \EEE
\end{enumerate}
and if it satisfies
\begin{subequations}\label{eq: weak equation}
\begin{equation}
\int_{ I} \int_{\Omega'} \Big(\C^2_{W^{\rm el}}\big( e(u)  + \tfrac{1}{2} \nabla' v \otimes \nabla' v  \big)  +\mu (\mathbb{B}^{(\alpha)})''+  \C^2_D \big( e( \partial_t \EEE u ) +   \partial_t \EEE \nabla' v   \odot  \nabla' v \big) \Big) : \nabla' \varphi_u \di x' \di t = 0 \label{eq: weak equation1}
\end{equation}
for all $\varphi_u \in C^\infty(I \times   \overline{\Omega'} ; \EEE \R^2)$ with $\varphi_u = 0$ on $I \times \Gamma_D'$,
\begin{align}\label{eq: weak equation2} 
 &\int_{ I} \int_{\Omega'} \Big(\C^2_{W^{\rm el}}\big( e(u)  + \tfrac{1}{2} \nabla' v \otimes \nabla' v  \big)  + \mu (\mathbb{B}^{(\alpha)})''+  \C^2_D \big( e( \partial_t \EEE u ) +   \partial_t \EEE \nabla' v  \odot\nabla' v \big) \Big) : \big(\nabla' v  \odot  \nabla' \varphi_v  \big) \di x' \di t \notag \\
 &\quad +\frac{1}{12} \int_0^T \int_{\Omega'} \Big(\C^2_{W^{\rm el}} (\nabla')^2 v + \C^2_D  \partial_t \EEE (\nabla')^2 v \Big) : (\nabla')^2 \varphi_v \di x' \di t = \int_0^T \int_{\Omega'} f^{2D}\varphi_v \di x' \di t
\end{align}
\end{subequations}
for all $\varphi_v \in C^\infty(I \times  \overline{\Omega'})$ with $\varphi_v = 0$ on $I \times \Gamma_D'$, and
\begin{align}\label{eq:temperature2dweak}
  &\int_I \int_{\Omega'} \hspace{-0.1cm}  \CD^{2, \alpha} \EEE \Big(
      \sym( \partial_t \EEE \nabla' u)
      +  \partial_t \EEE \nabla' v \odot \nabla' v
    \Big)\hspace{-0.1cm}  : \hspace{-0.1cm}  \Big(
      \sym( \partial_t \EEE \nabla' u)
      +  \partial_t \EEE \nabla' v \odot \nabla' v
    \Big) {\varphi_{\mu}} \di x' \di t \notag \\
    &\quad + \int_I \int_{\Omega'} \frac{1}{12} \Big(\CD^{2,\alpha}   \partial_t \EEE (\nabla')^2 v
      : \partial_t \EEE (\nabla')^2 v  
    \Big)  {\varphi_{\mu}} \di x' \di t
    + \kappa \int_I \int_{\GGG \Gamma' \EEE}  (\mu_\flat -\mu){\varphi_{\mu}} \di \haus^1(x') \di t + \overline c_V \int_{\Omega'} \mu_0 {\varphi_{\mu}}(0) \di x' \notag \\
    &\qquad = \int_I \int_{\Omega'} \Big( \tilde{\mathbb{K}}  \nabla' \mu :
    \nabla' {\varphi_{\mu}} - \overline c_V \mu  \partial_t \EEE {\varphi_{\mu}}  \Big)\di x' \di t
\end{align}
for all ${\varphi_{\mu}} \in C^\infty(I \times \overline{\Omega'})$ with ${\varphi_{\mu}}(T) = 0$.
\end{definition}

It is a standard matter to check that sufficiently smooth weak solutions  coincide with classical solutions of the system \eqref{eq:mechanical2d}--\eqref{eq:temp2D_bdry_init}, complemented with  additional natural Neumann \ZZZ  conditions \EEE for $(u,v)$.

\subsection{Main convergence result}

We consider the two dimensional elastic energy, defined by
\begin{align}\label{eq: phi0}
{\phi}^{\rm el}_0(u,v) \defas \int_{\Omega'} \bigg( \frac{1}{2}Q_{W^{\rm el}}^2\Big( e(u) + \frac{1}{2} \nabla'v \otimes \nabla'v \Big) + \frac{1}{24}Q_{W^{\rm el}}^2((\nabla')^2 v) \bigg) \, {\rm d}x',
\end{align}
\GGG and recall the definition of $\mathscr{S}^{3D}_h$ in \eqref{eq: nonlinear boundary conditions}. \EEE
 The main result of this paper is as follows: \EEE
\begin{theorem}[Convergence to the two-dimensional system]\label{maintheorem}
Let $(\EEE(y_0^h, \theta_0^h))_h\EEE$ be a sequence of initial data with $y_0^h \in \mathscr{S}^{3D}_h$.  Denote by $((u_0^h,v_0^h))_h \EEE $  and  $( \mu^h_0 )_h \EEE$  \EEE their rescaled versions given by \eqref{def:displacements} and \eqref{def:scaledtemp}, respectively. Let $(u_0,v_0,\mu_0) \ZZZ \in H^1(\Omega';\R^2) \times   H^2(\Omega') \times  L^2(\Omega') \EEE $ be limiting initial values and assume that  $(u_0^h,v_0^h) \rightharpoonup (u_0,v_0) $ in $H^1(\Omega';\R^3)$,   $\mu_0^h \to \mu_0$  strongly in $L^2(I \times \Omega')$, and \EEE
\begin{equation}\label{well-preparednessinitial}
h^{-4}  \mathcal{M} (y_0^h) \EEE  \to \phi_0^{\rm el} (u_0,v_0), \quad \quad     \Vert h^{-\alpha} \theta^h_0\Vert_{L^2(\Omega)} \to \Vert \mu_0 \Vert_{L^2(\Omega')}. \EEE  
\end{equation}
\ZZZ Suppose that   \eqref{forcebound}--\eqref{assumption:temp} \GGG and \eqref{quadraticformsassumption}--\eqref{heattensorassumption} \EEE hold.  \EEE   Then, there exists a sequence of  solutions \EEE $(\EEE (y^h,\theta^h))_h\EEE$ to \eqref{weak_form_mech_res}--\EEE\eqref{weak_form_heat_res}, converging in the sense of Proposition~\ref{prop:compactness}  to  some \EEE $(u,v,\mu)$, and $(u,v,\mu)$  is a weak solution to \eqref{eq:mechanical2d}--\eqref{eq:temp2D_bdry_init}  as described  in   Definition~\ref{def:weak_form_vK}. \OOO Moreover, it satisfies the energy balance \EEE
\begin{align}\label{balance2Dmaintheorem}
&\phi_0^{\rm el}(u(t),v(t)) - \phi_0^{\rm el}(u(0),v(0)) + \int_0^t \int_{\Omega'} \Big( Q_R^2 \big( \sym( \partial_t \EEE \nabla' u) +   \partial_t \EEE \nabla' v \odot \nabla' v
    \big)   + \frac{1}{12}  Q_R^2 \big(  \partial_t \EEE (\nabla')^2 v \big) \Big) \, {\rm d}x \di s \notag  \\
    &\quad + \int_0^t \int_{\Omega'} \mu (\mathbb{B}^{(\alpha)})'' \big( \sym( \partial_t \EEE \nabla' u) +  \partial_t \EEE \nabla' v \odot \nabla' v
     \big) \, {\rm d}x \di s =\int_0^t \int_{\Omega'} f^{2D}  \partial_t \EEE v \, {\rm d}x' \di s
\end{align}
for  every $t\in I$. \EEE 
\end{theorem}

  Note that condition \eqref{well-preparednessinitial} ensures  \emph{well-preparedness of initial data}.
 In other words,  it corresponds to the existence of a recovery sequence  in \EEE the $\Gamma$-convergence result of $h^{-4} \mechen$  to \EEE $\phi_0^{\rm el}$, \ZZZ see  \EEE \cite[Theorem 5.6]{FK_dimred}.   

We close this section by giving a further outline of the paper. In Section~\ref{sec:korn}, \ZZZ  we address a generalized version of Korn's inequality \cite{Pompe} and derive an optimal scaling of the \GGG Korn's \EEE constant \GGG in \EEE thin domains.  \EEE   Section~\ref{sec:existenceapriorisection} is devoted to \ZZZ deriving a priori estimates  of solutions in the \GGG three-dimensional \EEE setting in terms of the thickness $h$.  In Section~\ref{sec:proofofcompactness}, we treat   the dimension reduction and show convergence of solutions from the \GGG three-dimensional \EEE to the \GGG two-dimensional \EEE setting. \EEE

\EEE

\section{A generalized Korn's inequality on thin domains}\label{sec:korn}

 This section is devoted to a generalized Korn's inequality on thin domains. More precisely, we will revisit the estimate established in \cite{MielkeRoubicek20Thermoviscoelasticity, Neff, Pompe} by  investigating the scaling of the constant in thin domains.  The inequality  will be instrumental \OOO for the derivation of a priori estimates in the next section  (see Proposition~\ref{prop:Existenceof3dsolutions}), but  may also  be of independent interest beyond our application to a model in thermoviscoelasticity. \ZZZ We highlight that identifying optimal scalings of Korn's constants in thin domains \GGG is an issue motivated and studied in the context of linear elastostatics (see e.g.\ \cite{Korn:Scalingreference, kohnvogelius1085}) and fluid mechanics (see e.g.\ \cite{LewickaMueller}), \EEE but has not been performed yet for the generalized version needed for our  \EEE  model.

 As in Section \ref{sec:model}, \EEE given  $h > 0$ and a   Lipschitz \EEE domain $\Omega' \subset \R^2$, we consider $\Omega_h \defas \Omega' \times (-h/2, h/2)$  and \EEE set $\Gamma_D^h \defas \Gamma_D' \times (-h/2, h/2)$, where $\Gamma_D' \subset \Gamma' \defas \partial \Omega'$ is an arbitrary but fixed open subset.
It is a well-known result, see e.g.~\cite[Theorem A.1(ii)]{Korn:Scalingreference}  or \cite{ kohnvogelius1085}, \EEE   that there exists a constant $C = C(\Omega')$ independent of the thickness $h$ such that for all $u \in H^1(\Omega_h; \R^3)$ with $u = 0$ on $\Gamma_D^h$ it holds that
\begin{equation}\label{standardkornscaling}
  \int_{\Omega_h} \vert \nabla u(x) \vert^2 \di x
    \leq\frac{C}{h^2} \int_{\Omega_h} \vert \sym ( \nabla u ) \vert^2 \di x,
\end{equation}
where $\sym ( \nabla u ) \defas (\nabla u^T + \nabla u)/2$. \ZZZ Note \EEE that the \EEE scaling $h^{-2}$ is optimal. \ZZZ In our work, we need the   following   generalization   of Korn's inequality. \EEE

\begin{theorem}[Generalized \ZZZ Korn's \EEE inequalities on thin domains]\label{pompescaling}
Let $\Omega'$, $\Omega_h$, $\Gamma_D'$, and $\Gamma_D^h$ be as described in the beginning of this section.
\ZZZ Given \EEE $p > 3$ and $\rho > 0$, let $z \in W^{2,p}(\Omega_h; \R^3)$ be such that $\det(\nabla z) \geq \rho$ in $\Omega_h$, $\Vert \nabla z \Vert_{L^\infty({\Omega_h})} \leq 1/\rho$, and $\Vert \nabla^2 z \Vert_{L^p({\Omega_h})} \leq 1/\rho$.
Then, the following holds true:
\begin{enumerate}[label=(\roman*)]
  \item \label{item:kornsecondthin} There exists a constant $C = C(\Omega', \rho, p) > 0$ such that for all $h$ sufficiently small and $u \in H^1({\Omega_h}; \R^3)$ we can find a matrix $A \in \R^{3 \times 3}_\skw$ satisfying
  \begin{equation}\label{kornsecondthin}
    \left\Vert \nabla u - A \nabla z \right\Vert_{L^2(\Omega_h)} \leq	\frac{C}{h} \Vert \sym ((\nabla z)^T \nabla u ) \Vert_{L^2(\Omega_h)}.
  \end{equation}
  \item \label{item:kornsecondthin_dirichlet} Moreover, there exists a constant $C = C(\Omega', \Gamma_D', \rho, p) > 0$ such that for all $h$ sufficiently small and $u \in H^1({\Omega_h}; \R^3)$ with $u = 0$ on $\Gamma_D^h$ it holds that
  \begin{equation}\label{kornsecondthin_dirichlet}
    \left\Vert \nabla u \right\Vert_{L^2(\Omega_h)} \leq	\frac{C}{h} \Vert \sym ((\nabla z)^T \nabla u ) \Vert_{L^2(\Omega_h)}.
  \end{equation}
\end{enumerate}
\end{theorem}
\GGG  It is worth noting that the asymptotic  behavior of the generalized versions \eqref{kornsecondthin} and \eqref{kornsecondthin_dirichlet} is consistent with \eqref{standardkornscaling}. Inequality \eqref{kornsecondthin_dirichlet} has been addressed in \cite[Theorem~3.3]{MielkeRoubicek20Thermoviscoelasticity} and \cite[Corollary~4.1]{Pompe} without analyzing the constant in terms of  the thickness. Inequality \eqref{kornsecondthin} in turn emerges as a byproduct of our analysis and has, to our knowledge, not been proved in this specific form. Our proof   crucially relies on the following generalization of Korn's inequality.  \EEE
\begin{proposition}[Generalized Korn's inequality]\label{lem:kernelchar}
   Given a Lipschitz domain \EEE $\Omega$, $\rho > 0$, and $\lambda \in (0, 1]\EEE$, there \EEE exists a constant $\GGG C \EEE>0$ \ZZZ depending on  $\Omega$, $\rho$, and $\lambda$ \EEE such that for all $u \in H^1(\Omega; \R^3)$ and $F \in C^{0, \lambda}(\Omega; \R^{3\times 3})$ satisfying $\det F \geq \rho$ in $\Omega$ and $\Vert F \Vert_{C^{ 0,\EEE \lambda}(\Omega)} \leq 1/\rho$ it holds that
  \begin{align*}
    \Vert u \Vert_{H^1(\Omega)}
      \leq  \GGG C \EEE \Big(
        \Vert u \Vert_{L^2(\Omega)}
        + \Vert \sym (F^T \nabla u ) \Vert_{L^2(\Omega)}
      \Big).
  \end{align*}
\end{proposition}
\begin{proof}
 The proof follows by combining \cite[Theorem 2.2]{Pompe} and \cite[Theorem 3.3]{MielkeRoubicek20Thermoviscoelasticity}. \EEE
\end{proof} 

\GGG
We will first address    \eqref{kornsecondthin} on cubes and afterwards  we pass to $\Omega_h$ \EEE via a covering argument. 
 
\begin{proposition}[Generalized Korn's second inequality on cubes]\label{thm:secondpompegeneralized}
  Given an open cube $Q$ of side length $h \in (0, 1]$, $\rho > 0$, and $p > 3$, let $z \in W^{2,p}(Q; \R^3)$ be such that $\det(\nabla z) \geq \rho$ in $Q$, $\Vert \nabla z \Vert_{L^\infty(Q)} \leq 1/\rho$, and $\lVert \nabla^2 z \rVert_{L^p(Q)} \leq 1/\rho$.
  Then, there exists a constant $C = C(\rho, p)$, independent of $h$, such that for all $u \in H^1(Q; \R^3)$  it holds that \EEE
  \begin{align}
     \left\Vert \nabla v \right\Vert_{L^2(Q)}
      &\leq C \Vert \sym ((\nabla z)^T \nabla u) \Vert_{L^2(Q)}, \label{secondkorn_grad} \\
    \left\Vert v - \fint_Q v \di x \right\Vert_{L^2(Q)}
      &\leq C h \Vert \sym ((\nabla z)^T \nabla u) \Vert_{L^2(Q)}, \label{secondkorn_l2}
  \end{align}
  where
  \begin{equation*}
    v \defas u  - \left( \fint_Q \skw (\nabla u (\nabla z)^{-1}) \di x \right) z.
  \end{equation*}
\end{proposition}
 For the proof, we recall the following \EEE  characterization of the kernel of $\sym(\OOO(\EEE\nabla z\OOO)\EEE^T \nabla u)$ for $u$ and $z$ as given in   Proposition   \ref{thm:secondpompegeneralized}.
\begin{lemma}[Infinitesimal rigid \ZZZ displacements\EEE]\label{rigiddisplacementlemma}
Let $\Omega \subset \R^3$ be a Lipschitz domain and $u \in H^1(\Omega; \R^3)$.
Moreover, for $p > 3$ and $\rho > 0$, let $z \in W^{2,p}(\Omega; \R^3)$ with $\det \nabla z \geq \rho$ in $\Omega$ and
  \begin{equation*}
    \sym \left((\nabla z)^T \nabla u \right) = 0 \quad \text{a.e.~in $\Omega$}.
  \end{equation*}
  Then, there exist \GGG some \EEE $a \in \R^3$ and  $A \in \R^{3\times 3}_{\rm skew}$ such that $u = A z + a$ a.e.~in $\Omega$.
\end{lemma}
\begin{proof}
  The proof can be found in \cite[Theorem 2.7]{Lankeit2013}, see also \cite[Theorem 2.5]{Anicic}.
\end{proof}


\GGG
We now prove   Proposition~\ref{thm:secondpompegeneralized} by performing a usual argument    by contradiction, which in this case relies on   Proposition~\ref{lem:kernelchar}  and  Lemma~\ref{rigiddisplacementlemma}.  \EEE

\begin{proof}[Proof of  Proposition \EEE \ref{thm:secondpompegeneralized}]
The proof is divided into two parts.
 We will first prove \EEE the statement for the  unit cube\EEE.  By a rescaling argument we will extend the result to an arbitrary cube.\EEE

  \textit{Step 1 (Special case of the unit cube)}:
  Let $Q = (0, 1)^3$ be  the unit cube\EEE.
  \EEE Assume that there exist sequences  $(u_k)_k \subset H^1(Q; \R^3)$ and $(z_k)_k \subset W^{2,p}(Q; \R^3)$ \EEE such that \OOO for all $k \in \N$ we have that \EEE $\det \nabla z_k \geq \rho$  in $Q$, \EEE $\Vert \nabla z_k \Vert_{L^\infty(Q)} \leq 1/\rho$, $\Vert \nabla^2 z_k \Vert_{L^p(Q)} \leq 1/\rho$, and
  \begin{equation}\label{contr}
    \left\Vert v_k - \fint_Q v_k \di x \right\Vert_{H^1(Q)} \geq k \Vert \sym ((\nabla z_k)^T \nabla u_k) \Vert_{L^2(Q)},
  \end{equation}
  where
  \begin{equation*}
    v_k \defas u_k  - M_k z_k, \qquad M_k \defas \fint_Q \skw (\nabla u_k (\nabla z_k)^{-1}) \di x \in \R^{3 \times 3}_{\rm skew}.
  \end{equation*}
  Setting
  \begin{equation*}
    \tilde v_k \defas \frac{v_k - \fint_Q v_k \di x}{ \lVert v_k - \fint_Q v_k \di x  \rVert_{H^1(Q)}},
  \end{equation*}
  by Rellich-Kondrachov, we can find $\tilde v \in H^1(Q; \R^3)$ such that, up to taking a subsequence,
  \begin{equation}\label{weakstrongwk}
    \nabla \tilde v_k \weakly \nabla \tilde v \text{ weakly in } L^2(Q ; \R^{3 \times 3} \EEE), \qquad
    \tilde v_k \to \tilde v \text{ strongly in } L^2(Q; \R^3).
  \end{equation}
  By definition of $\tilde v_k$ we have that $\int_Q \tilde v_k \di x = 0$ for all $k$ and therefore
  \begin{equation}\label{meanwk}
    \int_Q \tilde v \di x = 0
  \end{equation}
  by the strong convergence of $(\tilde v_k)_k$ in $L^2(Q; \ZZZ \R^3) \EEE $.  As $M_k$ is skew-symmetric,  we get \EEE   
  \begin{align}
    \int_Q \skw(\nabla v_k (\nabla z_k)^{-1}) \di x
    &= \int_Q \skw((\nabla u_k - M_k \nabla z_k) (\nabla z_k)^{-1}) \di x  = 0. \EEE \label{zeroskew}
  \end{align}
   We  define \EEE $\tilde z_k \defas z_k - \fint_Q z_k \di x$.
  As $\nabla \tilde z_k = \nabla z_k$, by our assumptions on $z_k$ and Poincaré's inequality we see 
  \begin{equation*}
    \lVert \tilde z_k \rVert_{W^{2, p}(Q)} \leq C \lVert \nabla z_k \rVert_{W^{1, p}(Q)} \leq \frac{C}{\rho},
  \end{equation*}
  where $C$ is a constant independent of $k$.
  By Morrey's embedding and by possibly increasing $C$ we have $\lVert \tilde z_k \rVert_{C^{1, \lambda}(Q)} \leq C / \rho$ for all $k$, where $\lambda \defas 1 - 3/p$.
  Hence, up to taking a further subsequence, $\tilde z_k \to z$ strongly in  $W^{1, \infty}(Q;\R^3)$ \EEE for some $z \in W^{2, p}(Q; \R^3)$ with $\det \nabla z \geq \rho$ in $Q$.
  In particular, this implies that $(\nabla z_k)^{-1} \to (\nabla z)^{-1}$ strongly in $L^\infty(Q; \R^{3 \times 3})$.
  Using weak-strong convergence,  \eqref{weakstrongwk}, \EEE and \eqref{zeroskew}, we  derive \EEE
  \begin{equation}\label{zeroskewlim}
    \int_Q \skw(\nabla \tilde v (\nabla z)^{-1}) \di x = 0.
  \end{equation}
 By \EEE the definition of $v_k$,  the identity $\sym((\nabla z_k)^T M_k \nabla z_k) = (\nabla z_k)^T \sym(M_k) \nabla z_k$, and the \GGG skew-symmetry \EEE of $M_k$, \EEE   it follows that
  \begin{equation*}
    \sym((\nabla z_k)^T \nabla u_k)
    = \sym((\nabla z_k)^T \nabla v_k + (\nabla z_k)^T M_k \nabla z_k)
    = \sym((\nabla z_k)^T \nabla v_k).
  \end{equation*}
  Dividing \eqref{contr} by \ZZZ $ \GGG k \ZZZ \Vert v_k - \fint_Q v_k \di x \Vert_{H^1(Q)}$ \EEE  then leads to
  \begin{equation}\label{contr2}
    \Vert \sym ((\nabla z_k)^T \nabla \tilde v_k) \Vert_{L^2(Q)} \leq 1/k.
  \end{equation}
\OOO As $\|\nabla \tilde v_k\|_{L^2(Q)} \leq 1$ for all $k \in \N$ and  $\nabla z_k \to \nabla z$ in $W^{1,\infty}(Q;\R^{3\times 3})$ \OOO this \EEE shows
  \begin{equation*}
 \limsup_{k, \, l\to \infty}   \Vert \sym ((\nabla z_l)^T \nabla \tilde v_k) \Vert_{L^2(Q)} =  0. 
  \end{equation*}
  \EEE   This \OOO limit \EEE allows us to improve the first convergence in \eqref{weakstrongwk} to strong convergence in $L^2(Q;\R^{3 \times 3})$.
  In fact, by the second convergence in \eqref{weakstrongwk} \OOO and \EEE \ZZZ Proposition \EEE \ref{lem:kernelchar} for $u \defas \tilde v_k -  \tilde v_l\EEE$  for $k, \, l\in \N$ \EEE and $F \defas \nabla z_l$ we derive that $(\nabla \tilde v_k)_k$ is a  Cauchy sequence \EEE in $L^2(Q; \R^{3 \times 3})$.   Then, strong convergence follows. As a consequence, \EEE since $\Vert \tilde v_k \Vert_{H^1(Q)} = 1$ for each $k \in \N$, we  obtain \EEE $\Vert \tilde v\Vert_{H^1(Q)} = 1$. \EEE
  Additionally, with \eqref{contr2}   and the fact that $\nabla z_k \to \nabla z$ in $\ZZZ L^\infty \EEE (Q;\R^{3\times 3})$   it follows that
  \begin{equation*}
    \sym((\nabla z)^T \nabla \tilde v) = 0 \quad \text{a.e.\ in } Q.
  \end{equation*}
    Then, \EEE b\EEE y Lemma \ref{rigiddisplacementlemma} there exist $a \in \R^3$ and a   skew-symmetric \EEE \EEE $A \in \R^{3 \times 3}_{\rm skew}$ such that $\tilde v = A z + a$ a.e.~in $Q$.
  Taking the gradient on both sides, multiplying by $(\nabla z)^{-1}$ from the right and using \eqref{zeroskewlim}, it follows that $A = \fint_Q \skw(\nabla \tilde v (\nabla z)^{-1}) \di x = 0$.
  In particular, $\tilde v$ is constant.
  With \eqref{meanwk} this leads to $\tilde v = 0$ a.e.~in  $Q$ \EEE which contradicts $\lVert \tilde v \rVert_{H^1(Q)} = 1$.

  \textit{Step 2 (General case by rescaling)}: In the first step, we have shown that there exists a constant  $\tilde C>0$ \EEE  such that \eqref{secondkorn_grad} and \eqref{secondkorn_l2}  hold true in the case of the unit cube $\tilde Q \defas (0, 1)^3$.
  Consider now a general cube $Q = a + h \tilde Q$ of side length $h$ for some $a \in \R^3$.
  Let $u$, $z$, and $v$ be as in the statement.
  We define $\tilde u(x) \defas u(a + hx)$ and $\tilde z(x) \defas h^{-1} z(a + hx)$ for $x \in \tilde Q$. Note that Step 1 applies to this choice of $\tilde u$ and $\tilde z$  since \EEE by a change of coordinates
  \begin{gather*}
    \det(\nabla \tilde z(x)) = \det(\nabla z(a + hx)) \geq \rho > 0 \quad \text{for all } x \in \tilde Q, \\
    \lVert \nabla \tilde z \rVert_{L^\infty(\tilde Q)} =  \lVert \nabla z \rVert_{L^\infty(Q)} \leq \frac{1}{\rho}, \quad \text{\OOO and} \quad
    \lVert \nabla^2 \tilde z \rVert_{L^p(\tilde Q)} =  h^{(p-3)/p} \EEE \lVert \nabla^2 z \rVert_{L^p(Q)} \leq \frac{1}{\rho}.
  \end{gather*}
  Let us further define $\tilde v(x) \defas v(a + hx)$ for $x \in \tilde Q$.
  Again, changing coordinates we derive that
  \begin{equation*}
    \fint_{\tilde Q} \skw(\nabla \tilde u (\nabla \tilde z)^{-1}) \di x
    = h \fint_{Q} \skw(\nabla u (\nabla z)^{-1}) \di y
  \end{equation*}
  and therefore $\tilde v = \tilde u - (\fint_{\tilde Q} \skw(\nabla \tilde u (\nabla \tilde z)^{-1}) \di x ) \tilde z$ in $\tilde Q$.
  Then, by Step 1  and change of coordinates \EEE it follows that
  \begin{equation*}
    \lVert \nabla v \rVert_{L^2(Q)}
    = \sqrt{h} \lVert \nabla \tilde v \rVert_{L^2(\tilde Q)}
    \leq \tilde C \sqrt{h} \lVert \sym((\nabla \tilde z)^T \nabla \tilde u) \rVert_{L^2(\tilde Q)}
    = \tilde C \lVert \sym((\nabla z)^T \nabla u) \rVert_{L^2(Q)},
  \end{equation*}
  which shows \eqref{secondkorn_grad}.
  The estimate in \eqref{secondkorn_l2} can be shown using Poincaré's inequality
  \begin{equation*}
    \left\lVert v - \fint_Q v \di x \right\rVert_{L^2(Q)} \leq C_P h \lVert \nabla v \rVert_{L^2(Q)} \leq \tilde C C_P h \lVert \sym((\nabla z)^T \nabla u) \rVert_{L^2(Q)},
  \end{equation*}
   where $C_P$ denotes the constant of Poincaré's inequality \EEE in the unit cube.
\end{proof}

Our next goal is to transfer our generalization of Korn's second inequality from $cubes$ to sets $U$ that are Bilipschitz equivalent to cubes. \ZZZ Here, a \EEE set $U$ is  called \EEE Bilipschitz equivalent to a cube $Q$ if there exists a Lipschitz bijection $\Phi \colon Q \to U$ with Lipschitz inverse.  Notice that even in the classical case of Korn's second inequality this is \emph{not} purely a matter of changing coordinates  since, given $u\colon U\to \R^3$, ${\rm sym}(\nabla (u \circ \Phi))$ is in general not controlled in terms of ${\rm sym}(\nabla u)$.   The statement in the generalized setting is as follows. \EEE

\begin{proposition}[Generalized Korn's second inequality for sets Bilipschitz equivalent to cubes]\label{lem:secondpompebilipschitz}
Let $U  \subset \R^3\EEE$ be Bilipschitz equivalent to a cube of side length $h > 0$ with controlled Lipschitz constants  independent of $h$. \EEE
 Let \EEE  $z \in W^{2,p}(U; \R^3)$ be such that $\det(\nabla z) \geq \rho$ in $U$, $\Vert \nabla z \Vert_{L^\infty(U)} \leq 1/\rho$, and $\Vert \nabla^2 z \Vert_{L^p(U)} \leq 1/\rho$ for some $\rho > 0$ and $p > 3$.
Then, there exists a constant $C = C(\rho, p) > 0$, independent of $h$, such that for all $u \in H^1(U; \R^3)$ we can find a matrix $ A \EEE\in \R^{3\times 3}_\skw$ satisfying
 \begin{equation*}
   \Vert \nabla u - A \nabla z \Vert_{L^2(U)}
      \leq C  \Vert \sym ((\nabla z)^T \nabla u ) \Vert_{L^2(U)}.
  \end{equation*}
\end{proposition}

 The proof of Proposition \ref{lem:secondpompebilipschitz} follows by using a Whitney covering argument and a weighted Poincaré inequality from \cite[Theorem B.4]{LewickaMueller}. As  it  is rather standard and similar to the proof of Theorem \ref{pompescaling}   (or the proof of \cite[Theorem 3.1]{FrieseckeJamesMueller:02}),   we omit it here. We proceed with the proof of Theorem~\ref{pompescaling}. \EEE 

\begin{proof}[Proof of Theorem~\ref{pompescaling}]
The proof is divided into five  steps. \EEE   We first introduce \EEE a suitable partition of the domain $\Omega_h$ such that Proposition~\ref{lem:secondpompebilipschitz}  applies \EEE on each element of the partition. In Step 2, we construct  a \EEE skew-symmetric matrix serving as a suitable \OOO candidate \EEE for  $A$ in \EEE \eqref{kornsecondthin}.  Step 3 and Step 4 are devoted to the proof of \EEE \eqref{kornsecondthin}.  Lastly,  in Step 5, \EEE we use  \eqref{kornsecondthin} \EEE to show \eqref{kornsecondthin_dirichlet}.

\textit{Step 1 (Covering of  $\Omega_h$):}
 For \EEE $x' \in \R^2$ and $r > 0$  we \EEE write $Q(x', r) \defas (x', 0)^T + (-r, r)^3$ and $U(x', r) \defas Q(x', r) \cap \Omega_h$  for shorthand. \EEE Given $h>0$, let $ J_h  \EEE \defas \{i \in h\Z^2 \colon Q(i, h/2) \cap \Omega_h \neq \emptyset \}$.
For each $j \in J_h$, we fix some $(x_j',0) \in (\Omega' \times \{0\}) \cap Q(j, h/2)$.
By the Lipschitz regularity of $\Omega'$ and by passing to a sufficiently small $h>0$, \EEE the sets $U_j^h \defas U(x_j',  h \EEE)$ are Bilipschitz equivalent to a cube of side length $h$ with controlled Lipschitz constants \ZZZ  independently \EEE of $j$ and $h$.
 The family  $(U_j^h)_j$ \EEE satisfies
\begin{enumerate}[label=(\roman*)]
\item \label{u_cover1} $\Omega_h = \bigcup_{j \in J_h} U_j^h$;
\item \label{u_intersec_number1} For all $j$ we  have  $ |M_j| \leq \OOO 25$, \EEE \ZZZ where \EEE  $\ZZZ M_j \defas \EEE \big\{ k \in J_h \colon  \vert {U_j^h} \cap {U_k^h} \vert >0 \big\} $.

\end{enumerate}

\textit{Step 2  (Definition \EEE  of $A$):}
Applying Proposition \ref{lem:secondpompebilipschitz} in each set $U_j^h$ yields a matrix $A_j \in \R^{3\times 3}_\skw$ such that
\begin{align}\label{coveringsecondpompe2}
\Vert \nabla u - A_j \nabla z \Vert_{L^2(U_j^h)} \leq C  \Vert \sym ( (\nabla z)^T \nabla u) \Vert_{L^2(U_j^h)},
\end{align}
for a constant $C>0$ independent of $j$  and $h$. \EEE
We  smoothly \EEE interpolate between  the \EEE matrices $( A_j)_j$  as follows. \OOO By Property \ref{u_cover1} we can find a smooth partition of unity \EEE $(\zeta_j)_{\OOO j\EEE}$ subordinate to the open sets $(U_j^h)_{j}$, i.e.,
\begin{align}\label{partitionofunity2}
\left\{
\begin{aligned}
&0 \leq \zeta_j \leq 1, \quad \zeta_j \in C_c^\infty(U_j^h; [0, 1]), \\
&\sum_{j\in  J_h\EEE} \zeta_j = 1 \quad \text{on } \Omega_h, \\
&\vert \nabla \zeta_j \vert \leq C h^{-1}.
\end{aligned}
\right.
\end{align}
 We \EEE define ${\tilde A} \colon \Omega_h \to \R^{3\times 3}_\skw$  by \EEE ${\tilde A} \defas \sum_{j \in J_h} \zeta_j A_j$.
Using Poincar\'e\OOO's \EEE inequality, there exists $A \in \R^{3\times 3}_\skw$ such that
\begin{equation}\label{poincare2}
\int_{\Omega_h} \vert {\tilde A} - A \vert^2 \dx \leq C \int_{\Omega_h} \vert \nabla {\tilde A} \vert^2 \dx,
\end{equation}
 where it is well-known that the constant depends on $\Omega'$ but is \EEE independent of $h>0$.

\textit{Step 3 (Proof of \eqref{kornsecondthin} with $\tilde A$ \ZZZ in place \EEE of $A$):}
The desired estimate \eqref{kornsecondthin}  with $\tilde A$ in place \EEE of $A$ directly follows from  \eqref{coveringsecondpompe2},  \eqref{partitionofunity2}, and Properties \ref{u_cover1} and \ref{u_intersec_number1}. In fact, we have
\begin{align*}
 \int_{\Omega_h} \vert \nabla u - {\tilde A} \nabla z \vert^2 \dx
  &= \int_{\Omega_h} \Big\vert \sum_{j \in J_h} \zeta_j \nabla u - \sum_{j \in J_h}\zeta_j A_j \nabla z \Big\vert^2 \dx   \leq C \sum_{ j \in J_h}  \int_{U_j^h} \vert  \nabla u - A_j \nabla z \vert^2 \dx \\
  & \leq  C \sum_{ j \in J_h} \int_{U_j^h} \vert  \sym ( (\nabla z)^T \nabla u ) \vert^2 \dx
  \leq C  \int_{\Omega_h} \vert  \sym ( (\nabla z)^T \nabla u ) \vert^2 \dx,
\end{align*}
where Property \ref{u_intersec_number1} was used in the  first and last inequality. \EEE  

\textit{Step 4 (Bound on  the $L^2$-distance between $A$ and $\tilde A$)\EEE:}
Notice that the desired result follows once we have shown that
\begin{equation}\label{AminustildeA2}
  \lVert \tilde A - A \rVert_{L^2(\Omega_h)} \leq C h^{-1}\lVert \sym((\nabla z)^T \nabla u) \rVert_{L^2(\Omega_h)}.
\end{equation}
In fact, by Step 3 and $\lVert \nabla z \rVert_{L^\infty(\Omega_h)} \leq 1/\rho$ we  then \EEE derive that
\begin{equation*}
   \Vert \nabla u - A \nabla z \Vert_{L^2(\Omega_h)} \leq  \Vert \nabla u - {\tilde A} \nabla z \Vert_{L^2(\Omega_h)} +  \Vert ({\tilde A} - A )\nabla z \Vert_{L^2(\Omega_h)}
      \leq C  h^{-1}\Vert \sym ((\nabla z)^T \nabla u ) \Vert_{L^2(\Omega_h)}.
\end{equation*}
It remains to show \eqref{AminustildeA2}.
 Let \EEE $j \neq k$ be such that $U_j^h \cap U_k^h \neq \emptyset$ with corresponding   skew-symmetric \EEE \EEE matrices $A_j$ and $A_k$.
 Our \EEE assumptions on $z$ directly  give \EEE $\lVert (\nabla z)^{-1} \rVert_{L^\infty( \Omega_h)\EEE} \leq C$  for a constant depending on $\rho$ but independent of $h$\EEE. 
In view of \eqref{coveringsecondpompe2}, we obtain
\begin{align*}
\vert A_j - A_k \vert^2
&= \frac{1}{\vert U_j^h \cap U_k^h \vert} \int_{U_j^h \cap U_k^h} \vert A_j - A_k \vert^2 \dx  \leq \frac{C}{\vert U_j^h \cap U_k^h \vert} \int_{U_j^h \cap U_k^h} \vert  A_j \nabla z - \nabla u  + \nabla u -  A_k \nabla z \vert^2 \dx \\
&\leq \frac{C}{\vert U_j^h \cap U_k^h \vert} \int_{U_j^h \cup U_k^h} \vert \sym ( (\nabla z)^T \nabla u ) \vert^2 \dx.
\end{align*}
Consequently, with \eqref{partitionofunity2}, \eqref{poincare2},  and  \EEE Properties \ref{u_cover1} and \ref{u_intersec_number1}, it follows that
\begin{align*}
\int_{\Omega_h} \vert {\tilde A} - A \vert^2 \dx
&\leq C \int_{\Omega_h} \Big\vert \nabla \Big(\sum_{j \in J_h}\zeta_j A_j \Big) \Big\vert^2 \dx   \leq C \sum_{k \in J_h} \int_{U_k^h}\Big\vert \nabla \Big(\sum_{j \in J_h} \zeta_j A_j \Big) - \nabla \Big( \sum_{j \in J_h} \zeta_j A_k\Big) \Big\vert^2 \dx \\
&\leq C \sum_{ j, \, k \in J_h} \int_{U_j^h \cap U_k^h}   \vert\nabla \zeta_j\vert^2 \vert A_j - A_k\vert^2 \dx  \leq C h^{-2}\sum_{ j, \, k \EEE \in J_h} \int_{U_j^h \cap U_k^h} |A_j - A_k|^2 \di x \\
&\leq C h^{-2} \sum_{\substack{ j, \, k \EEE \in J_h, \\  U_j^h \cap U_k^h \neq \emptyset \EEE}}    \int_{U_j^h \cup U_k^h} \vert \sym ( (\nabla z)^T \nabla u ) \vert^2 \dx
\leq C h^{-2} \int_{ \Omega_h} \vert \sym ( (\nabla z)^T \nabla u ) \vert^2 \dx,
\end{align*}
\ZZZ where in the second step we used that $ \nabla ( \sum_{j \in J_h} \zeta_j) = 0$. \EEE
 This \EEE concludes the proof of  \eqref{kornsecondthin}\EEE.

\textit{Step 5 (Proof of  \eqref{kornsecondthin_dirichlet}\EEE):}
We use  \eqref{kornsecondthin} in order \EEE to prove the generalized version of Korn's first inequality, see \eqref{kornsecondthin_dirichlet}.
 As $\Gamma_D'$ is an open subset of $ \GGG \Gamma' = \EEE \partial \Omega'$, \EEE we can find $r > 0$ sufficiently small and $x' \in \Gamma_D'$ such that $B'_r( x')  \EEE \setminus \Omega'$ is connected,  $B'_r(x') \cap \Gamma' \subset \Gamma_D'$, \ZZZ and $\tilde \Omega_h \defas \Omega_h \cup \big(B'_r(x') \times (-h/2, h/2)\big)$ is Lipschitz, \EEE where $B'_r(x')$ denotes the two-dimensional ball centered at $x'$ with radius $r$.
Let $u \in H^1(\Omega_h; \R^3)$ with $u = 0$ on $\Gamma_D^h$.  We extend $u$  to \EEE the set $\tilde \Omega_h $ by zero.
\OOO Up to possibly decreasing $r$\EEE, by  a Sobolev extension argument,  see \OOO e.g.~\EEE \cite[Chapter~6.3,~Theorem~5]{Stein1970}, \EEE we can extend $z$ to $\tilde \Omega_h$ such that $z$ still satisfies the assumptions  of \EEE the statement with $\Omega_h$ replaced by $\tilde \Omega_h$ and $\rho$ replaced by $\rho/2$.  

By \eqref{kornsecondthin} applied on the set $\tilde \Omega_h$ there exists a matrix $A \in \R^{3\times 3}_{\skw}$ such that
\begin{align}\label{kornsecondthinXXXXX}
    \left\Vert \nabla u - A \nabla z \right\Vert_{L^2(\tilde \Omega_h)} \leq	\frac{C}{h} \Vert \sym ((\nabla z)^T \nabla u ) \Vert_{L^2(\tilde \Omega_h)}.
\end{align}
\ZZZ Combining \EEE the \OOO previous  inequality \EEE  with $\Vert (\EEE\nabla z)^{-1} \EEE \Vert_{L^\infty(\Omega_h)} \leq C$,  we discover that \EEE
\begin{align*}
\vert A \vert^2 &\leq \frac{C}{\vert \tilde \Omega_h \setminus \Omega_h\vert}\int_{\tilde \Omega_h \setminus \Omega_h} \vert   A\nabla z\EEE \vert^2 \di x    \le  \frac{C}{\vert \tilde \Omega_h \setminus \Omega_h\vert}\int_{\tilde \Omega_h} \vert \nabla u -  A\nabla z  \vert^2 \di x  \EEE \\
&\leq \frac{C}{\vert \tilde \Omega_h \setminus \Omega_h\vert}  \frac{1}{h^2} \int_{\tilde \Omega_h} \vert  \sym ((\nabla z)^T \nabla u ) \vert^2 \di x \EEE \leq \frac{C}{h^3} \int_{\Omega_h} \vert \sym ( (\nabla z )^T \nabla u ) \vert ^2 \di x,
\end{align*}
where  in the second and the last step \EEE we have used that $u=0 $ on $\tilde \Omega_h \setminus \Omega_h$. We can then use  \eqref{kornsecondthinXXXXX}, \EEE the  triangular inequality, \EEE and the fact that $\Vert \nabla z \Vert_{L^\infty(\GGG \tilde \Omega_h \EEE )} \leq \GGG 2\EEE/\rho$ to  derive \EEE 
\begin{align*}
\Vert \nabla u \Vert_{L^2(\Omega_h)} & = \Vert \nabla u - A \nabla z + A\nabla z \Vert_{L^2(\tilde \Omega_h)}  \leq C h^{-1} \Vert \sym ((\nabla z)^T \nabla u ) \Vert_{L^2( \GGG \tilde \Omega_h \EEE)}+ C h^{1/2}\vert A \vert.
\end{align*}
\GGG By using that $u=0 $ on $\tilde \Omega_h \setminus \Omega_h$ once again, this implies \EEE \eqref{kornsecondthin_dirichlet}.
\end{proof}

\section{ A priori bounds for solutions of the three-dimensional problem}\label{sec:existenceapriorisection}

In this section, we  derive a priori bounds  for solutions given  in Remark \ref{rem:weakformresc}  which are needed to pass to the dimension-reduction limit.  Although it should \ZZZ also \EEE be possible to derive  these  bounds   on the time-discrete level, as done in the case of linearization \cite{RBMFMK}, we prefer  here    to work purely in the time-continuous setting  for   the sake of convenience and easier presentation. Nevertheless, certain technical difficulties will arise requiring us to first work in a regularized setting.
More precisely, we will start by deriving a priori bounds for solutions of an associated regularized problem, \OOO inspired by the one from   \cite[Section 4]{MielkeRoubicek20Thermoviscoelasticity}. \EEE It will be crucial to ensure that the derived bounds are independent of the regularization parameter.
Once this is shown,  all bounds   will be  inherited by the weak solutions of the nonregularized system.  We start with the formulation of the main statement.    Recall the definition of $\mathscr{S}_h^{3D}$ in \eqref{eq: nonlinear boundary conditions}. \EEE

\begin{proposition}[A priori estimates]\label{prop:Existenceof3dsolutions}
Let $y_0^h \in \mathscr{S}^{3D}_h$ and $\theta_0^h \in L^2_+(\Omega)$ be such that $ \sup\nolimits_{h>0} \GGG( \EEE h^{-4}\mathcal{M}(y_0^h) + \ZZZ  h^{-2\alpha} \EEE \Vert \theta_0^h \Vert_{L^{ 2\EEE}(\Omega)}^{ 2 \EEE} \GGG ) \EEE < + \infty $. Moreover, suppose that $f^{3D}_h \in W^{1,1}(I;L^2(\Omega'))$ and \ZZZ $\theta_\flat^h \in L^2(I;L_+^{2}(\Gamma))$ \EEE are given and satisfy \eqref{forcebound}--\eqref{assumption:temp}. \EEE Then,  for \EEE sufficiently small $h > 0$ there exists  a weak solutions $(y^h, \theta^h)$  in the sense of Remark \ref{rem:weakformresc} \EEE  and a constant $C > 0$ \ZZZ independently of $h$ \EEE  such that the following bounds hold true: 
\begin{subequations}\label{boundres}
\begin{align}
  \esssup_{t \in I}  \mathcal{M}(y^h(t)) \EEE  &\leq Ch^4, \label{boundres:mech} \\
  \int_I \int_{\Omega} R(\nabla_h y^h,  \partial_t  \nabla_h \EEE y^h,  \EEE \theta^h) \dx \dt &\leq C h^4, \label{boundres:dissipation} \\
  \Vert  \partial_t  \nabla_h \EEE y^h \Vert_{L^2(I\times \Omega)} &\leq C h. \label{boundres:strainrate}
\end{align}
\end{subequations}  
Moreover, for any $q \in [1, 5/3)$ and $r \in [1, 5/4)$ we can find constants $C_q$ and $C_r$ \ZZZ independently \EEE of $h$ such that
\begin{subequations}\label{temp_apriori}
\begin{align}
  \Vert \theta^h \Vert_{L^q(I\times \Omega)} + \Vert \zeta^h \Vert_{L^q(I\times \Omega)} &\leq C_q h^\alpha, \label{boundres:temperature} \\
  \Vert \nabla_h \theta^h \Vert_{L^r(I \times \Omega)} + \Vert \nabla_h \zeta^h \Vert_{L^r(I \times \Omega)} &\leq C_r h^\alpha,  \label{boundres:temperaturegrad} \\
  \Vert  \partial_t \EEE \zeta^h \Vert_{L^1(I; \OOO(\EEE H^3\EEE(\Omega)\OOO)\EEE^*)} &\leq C h^\alpha, \label{forAubin-Lion3}
\end{align}
\end{subequations}
where $\zeta^h \defas W^{\rm in}(\nabla_h y^h,\theta^h)$.
\end{proposition}

We remark that the bounds \eqref{boundres:mech}--\eqref{forAubin-Lion3} do not directly follow from the  bounds   already derived  in   \cite[Lemma 6.2 and  Proposition   6.3]{MielkeRoubicek20Thermoviscoelasticity}  (see also \cite[Lemma~3.18 and Theorem~3.20]{RBMFMK})  since  in  \cite{RBMFMK, MielkeRoubicek20Thermoviscoelasticity} the estimates are provided   for a \emph{fixed} domain.
 A naive application  to the present setting of domains with thickness $h$ might lead to constants depending on $h$.
 The crucial point of Proposition \ref{prop:Existenceof3dsolutions} is that   all constants appearing in the estimates are independent of the \ZZZ  thickness $h$. \EEE
In this regard, all constants throughout this sectio\OOO n m\EEE ay  vary from line to line,   but  are \EEE \emph{independent} of $h$.
 For convenience,  we will prove all  bounds on the thin domain $\Omega_h$, and the stated bounds in Proposition \ref{prop:Existenceof3dsolutions}  then easily follow by a change of coordinates.
\EEE

  In Section  \ref{sec: 4.1}, we introduce the regularized problem, and state some first auxiliary results. The relevant a priori estimates for the regularized problem are established in Sections \ref{sec: 4.2}--\ref{sec: 4.3}. The proof is concluded in Section \ref{sec: 4.4} by transferring the bounds to the original system of equations.   \EEE

\subsection{Regularization and auxiliary lemmas}\label{sec: 4.1}
We remind the reader that, similarly to the linearization  result in  thermoviscoelasticity \cite{RBMFMK}, in order to perform the dimension reduction we require a regularization of the dissipation rate $\xi$ depending on the temperature scale $\alpha$, see the definition of $\xi^{(\alpha)}$ in \eqref{diss_rate_truncated}\OOO. The regularization improves the a priori integrability of $\xi$, a fact that will be employed  in the proof of Proposition~\ref{lem: first a priori estimates} below. \EEE
To keep the argument concise, \OOO we further regularize $\xi^{(\alpha)}$ as \GGG it \OOO was done for $\xi$ \EEE in \cite[Section 4]{MielkeRoubicek20Thermoviscoelasticity}\OOO. \EEE
More precisely, given $\eps > 0$, we \OOO define \EEE for every $F \in GL^+(3)$, $\dot F \in \R^{3 \times 3}$, \ZZZ and \EEE $\vartheta \geq 0$: 
\begin{equation}\label{dissipation:trucated2}
\xi_{\eps, \alpha}^{\rm reg}(F, \dot F, \vartheta) \defas
\frac{\xi^{(\alpha)}(F,\dot F,\vartheta)}{1+ \eps \xi^{(\alpha)}(F,\dot F,\vartheta)}.
\end{equation}

Note that $\xi_{\eps, \alpha}^{\rm reg} \leq \eps^{-1}$ ensures even an $L^\infty$-bound on the regularized dissipation rate.  Moreover, we have  $\ZZZ \xi^{\rm reg}_{\eps, \alpha} \EEE \nearrow \xi^{(\alpha)}$ \ZZZ pointwise \EEE as $\eps \searrow 0$. For the reader's \GGG convenience, \EEE  let us start by \EEE repeating the notion of weak solutions in the $\eps$-regularized setting which is  similar \EEE to the  one in  \cite[Equations~(4.1)--(4.2)\EEE]{MielkeRoubicek20Thermoviscoelasticity}. \EEE

\begin{definition}[Weak solution of the $\eps$-regularized nonlinear system]\label{def:weak_formulation_reg}
Given $\eps > 0$, $w_0 \in \mathcal{W}^h_\id$, $\vartheta_0 \in L_+^{2}(\Omega_h)$, and  $\vartheta_\flat^h\EEE\in \ZZZ L^{2}(I;L^2_+(\Gamma_h)) \EEE$, let $w_{0,\eps}\defas w_0$, $\vartheta_{0,\eps}\defas \vartheta_0 (1+\eps \vartheta_0)^{-1}$, and $\vartheta_{\flat, \eps} \defas \vartheta_\flat^h\EEE(1 + \eps \vartheta_\flat^h\EEE)^{-1}$.
A pair $(w_\eps, \vartheta_\eps) \colon I \times \Omega_h \to \R^3 \times \R$ is said to be an $\eps$-\emph{regularized weak solution} to \eqref{quasistaticequation3D} and \eqref{thermoequation3D}  if \EEE $w_\eps \in L^\infty(I; \Wid) \cap H^1(I; H^1(\Omega_h; \R^3))$ with $w_\eps(0, \cdot) = w_{0,\eps}$, $\vartheta_\eps \in  L^{2}(I; H^1(\Omega_h)) $ with $\vartheta_\eps \ge 0$ a.e.\ and $\vartheta_\eps(0) = \vartheta_{0, \eps}$,  $ \partial_t \EEE m_\eps \in L^2(I;\REV (H^1(\Omega_h))^* \EEE)$, where $m_\eps \defas W^{\rm{in}}(\nabla w_\eps, \vartheta_\eps)$, \EEE and if it satisfies the identities
\begin{subequations}
\begin{equation}\label{weak_form_mech_res_unrescaled_reg}
\begin{aligned}
  &\int_I \int_{\Omega_h} \pl_G
    \hypot(\nabla^2 w_\eps) \cdddot \nabla^2 {\varphi_{w}}
    + \Big( \eps  \partial_t \EEE \nabla w_\eps +
      \pl_F \felpot(\nabla w_\eps, \vartheta_\eps)
      + \pl_{\dot F} \disspot(\nabla w_\eps,  \partial_t \EEE \nabla w_\eps, \vartheta_\eps)
    \Big) : \nabla   {\varphi_{w}} \di x \di t \\
    &\quad =  \int_I \int_{\Omega_h}  g_h^{3D} {(\varphi_{w})}_3 \EEE \di x \di t
\end{aligned}
\end{equation}
for any test function ${\varphi_{w}} \in C^\infty(I \times   \overline{\Omega_h}; \EEE \R^3)$ with ${\varphi_{w}} = 0$ on $I \times \Gamma_D^h$, as well as
\begin{align}\label{weak_form_heat_unrescaled_reg}
  &\int_I \int_{\Omega_h}
     \hcmnoh \EEE(\nabla w_\eps, \vartheta_\eps) \nabla \vartheta_\eps \cdot \nabla {\varphi_{\vartheta}}
    - \Big(
      \xi_{\eps, \alpha}^{\rm reg}(\nabla w_\eps,  \partial_t \EEE \nabla w_\eps, \vartheta_\eps)
      + \partial_F W^{\rm{cpl}}(\nabla w_\eps, \vartheta_\eps) :  \partial_t \EEE \nabla w_\eps
    \Big) {\varphi_{\vartheta}} \REV  \di x  + \langle \partial_t  m_\eps , {\varphi_{\vartheta}} \rangle \EEE \di t \nonumber\\
    &\quad + \kappa \int_I \int_{ \Gamma_h} \vartheta_\eps {\varphi_{\vartheta}} \di  \haus^2 \di t  = \kappa   \int_I \int_{ \Gamma_h} \vartheta_{\flat, \eps} {\varphi_{\vartheta}} \di \haus^2 \di t
    ,
\end{align}
for any test function ${\varphi_{\vartheta}} \in L^2(I;H^1(\Omega_h))$, \REV where $\langle \cdot, \cdot \rangle$ in \eqref{weak_form_heat_unrescaled_reg} denotes the dual pairing of $H^1(\Omega_h)$ and $(H^1(\Omega_h))^*$. \EEE
\end{subequations}
\end{definition}

\ZZZ For convenience,  in the definition above and in the remainder of this section we do not include the $h$-dependence of solutions $(w_\eps,\vartheta_\eps)$ in the  notation. \EEE \OOO We note that we choose for the second equation \eqref{weak_form_heat_unrescaled_reg} a class of test functions that is larger than the one in \cite[Equation (5.12)]{MielkeRoubicek20Thermoviscoelasticity}.  The equivalence of both definitions follows by a standard density argument. \EEE
In \cite[Propostion 5.1]{MielkeRoubicek20Thermoviscoelasticity}, \EEE existence of weak solution  to \EEE the regularized problem obeying an energy balance was shown.  The \EEE results  therein  naturally pass over to the present setting\EEE.
\OOO In fact\EEE, we have \ZZZ  the following. \EEE
\begin{proposition}\label{prop:existenceregularized}
For any $\eps, \ZZZ h \EEE > 0$, there exists a solution $(w_\eps,\vartheta_\eps)$ in the sense of Definition~\ref{def:weak_formulation_reg}.
Moreover, given any weak solution, the following energy balance holds true for    a.e.~\EEE $t \in I$:
\begin{align}\label{energybalanceregularized}
&\mathcal{M}_h(w_\eps(t)) + \int_0^t \int_{\Omega_h} \xi(\nabla w_\eps,  \partial_t \EEE \nabla w_\eps, \vartheta_\eps) + \eps \vert  \partial_t \EEE \nabla w_\eps \vert^2 \dx \di s\nonumber \\
&\quad= \mathcal{M}_h(w_\eps(0)) + \int_0^t \int_{\Omega_h}  g\EEE_h^{3D}(s) ({ \partial_t \EEE{w_\eps}})_3  \di x \EEE \di s - \int_0^t \int_{\Omega_h} \partial_F W^{\rm cpl} (\nabla w_\eps, \vartheta_\eps) :  \partial_t \EEE \nabla w_\eps \dx \di s.
\end{align}
\end{proposition}
\begin{proof}
This is exactly \cite[Proposition~5.1]{MielkeRoubicek20Thermoviscoelasticity} in the case $\alpha = 4$  as then $\xi^{(\alpha)} = \xi$\EEE.
 For $\alpha < 4$, we   have $\xi^{(\alpha)} \leq \xi$.
Hence, the proof of \cite[Proposition~5.1]{MielkeRoubicek20Thermoviscoelasticity} still applies as the regularized dissipation rate  is only required to be in $L^\infty$ and pointwise below $\xi$.
The energy balance follows by an application of a chain rule\EEE, see also \cite[Proposition~3.6~and~Equation~(5.9)]{MielkeRoubicek20Thermoviscoelasticity}.
\end{proof}

\ZZZ Next, \EEE we also collect some  helpful properties that will be instrumental later on.

\begin{lemma}
For all $F \in GL^+(3)$ and $ \vartheta \EEE \geq 0$ it holds that
\begin{equation}\label{est:coupl}
\vert \partial_F W^{\rm cpl} (F,  \vartheta \EEE) \vert + \vert \partial_F W^{\rm in}(F, \vartheta \EEE) \vert \leq C ( \vartheta \EEE \wedge 1 ) (1 + \vert F \vert).
\end{equation}
\end{lemma}
\begin{proof}
The statement without the second term on the left-hand  side \EEE has  already been shown in \cite[Lemma~3.4]{RBMFMK}.
The same bound also holds true for the internal \ZZZ energy. \EEE
In fact, using \eqref{Wint}, \ref{C_bounds}, and \cite[Lemma~3.4]{RBMFMK} we have  that \EEE
\begin{align*}
\vert \partial_F W^{\rm in}(F, \vartheta \EEE) \vert \leq \vert \partial_F W^{\rm cpl}(F, \vartheta \EEE) \vert + \vert \vartheta \partial_{F \vartheta} W^{\rm cpl}(F,  \vartheta \EEE) \vert \leq C ( \vartheta \EEE \wedge 1 ) (1 + \vert F \vert),
\end{align*}
as desired.
\end{proof}

\begin{lemma}
  There exists a constant $C>0$ such that for all $F \in GL^+(3)$, $\dot F \in \R^{3 \times 3}$, and $ \vartheta \EEE \geq 0$ it holds
   \begin{align}
    \partial_F W^{\rm cpl} (F, \vartheta \EEE)  : \dot F &= \frac{1}{2} F^{-1}  \partial_F W^{\rm cpl} ( F,  \vartheta \EEE) : (F^T \dot F + \dot F^T F), \label{writecplviadotc} \\
      \vert \partial_F W^{\rm cpl} (F, \vartheta \EEE): \dot F \vert
    &\leq C ( \vartheta \EEE \wedge 1) |F^{-1}| (1 + |F|) \ZZZ (\xi(F, \dot F,  \vartheta))^{1/2}. \EEE \label{avoidKorn}
  \end{align}
\end{lemma}

\begin{proof}
  By frame indifference, see \ref{C_frame_indifference}, there exists a potential  $\hat W^{\rm cpl} \colon (GL^+(3) \cap \EEE \R^{3 \times 3}_{\rm sym})  \times \R_+ \to \R$  such that for any $F \in  GL^+(3)\EEE$ and $\vartheta \geq 0$ it holds that \EEE
  \begin{equation*}
    \partial_F W^{\rm cpl} (F, \vartheta \EEE) = 2F \partial_C \hat W^{\rm cpl}(C, \vartheta \EEE),
  \end{equation*}
  where $C \ZZZ = \EEE F^T F$ is the Cauchy-Green tensor,   see \cite[ Equation~(3.17)\EEE]{RBMFMK}. Along these lines, by the symmetry of $\partial_C \hat W^{\rm cpl} $, \EEE we  can then show for \EEE any $\dot F \in \R^{3 \times 3}$
  \begin{equation*}
    F \partial_C \hat W^{\rm cpl}(C,  \vartheta \EEE) : \dot F = \frac{1}{2} \partial_C \hat W^{\rm cpl}(C, \vartheta \EEE) : (F^T \dot F + \dot F^T F).
  \end{equation*}
 Combining the previous \OOO two equalities \EEE we get  \EEE \eqref{writecplviadotc}.
  Then, \eqref{avoidKorn} follows by \EEE taking absolute values, using \eqref{est:coupl}, \eqref{diss_rate}, and the lower bound in \ref{D_bounds}.
\end{proof}

Recall \eqref{mechanical} \GGG and \eqref{eq: data}. \EEE \ZZZ Given \GGG $t \in I$, \EEE define $ \mathcal F^t_h(w) \EEE \defas \mathcal{M}_h(w) - \int_{\Omega_h}  g_h^{3D} (t) \EEE (w_3-x_3) \di x$  for all \OOO $w \in \mathcal{W}_{\id}^h$\EEE.
\begin{lemma}\label{lem:forceestimate}
  There exists a constant $C > 0$ such that for all $h > 0$, $t \in I$, and $w \in \mathcal{W}^h_\id$ it holds that  
  \begin{align}
\Vert  w_3 - x_3 \EEE \Vert_{H^1(\Omega_h)}^2 &\leq C h^{-2} \mathcal{M}_h(w), \label{ineq:force} \\
 \left\vert \int_{\Omega_h}  g^{3D}_h \EEE(t) (w_3  - x_3) \dx \right\vert
    &\leq  \min\{  \mathcal  F^t_h(w),  \EEE \mathcal{M}_h(w) \} +  C h^{5} .\label{forceestimate}
\end{align}
\end{lemma}
\begin{proof}
\OOO Fix $w \in \mathcal{W}_{\id}^h$. In  \cite[Equation~(35)]{lecumberry}  (\EEE see also \cite[Theorem~6]{FJM_hierarchy}) \EEE it is shown that there exists a rotation $Q \in SO(3)$ \OOO satisfying \EEE  
\begin{align*}
\Vert \nabla w - Q \Vert_{L^2(\Omega_h)}^2 \leq C h^{-2} \mathcal{M}_h(w ).
\end{align*}
Moreover, by the definition of $\mathcal W_\id^{\OOO h\EEE}$, see \eqref{admissibledeformations} \OOO a\EEE nd \cite[Equation~ (53)]{lecumberry} we derive that
\begin{align*}
\Vert Q - \Id \Vert_{L^2(\Omega_h)}^2 \leq Ch^{-2} \mathcal{M}_h(w).
\end{align*}
Thus, by Poincaré's inequality it  follows that \EEE 
\begin{align*}
\Vert w_3 - x_3 \Vert_{H^1(\Omega_h)}^2 \leq C \Vert \nabla w - \Id \Vert_{L^2(\Omega_h)}^2 \leq C h^{-2} \mathcal{M}_h(w),
\end{align*}
which  is \EEE \eqref{ineq:force}.    We now show \eqref{forceestimate}. \EEE  By the fundamental theorem of calculus in Bochner spaces,  \eqref{eq: data}, and \EEE  the  first bound \EEE in \eqref{forcebound} we have for all $t \in I$  that \EEE 
\begin{align}\label{forceonthindomain}
\Vert  g_h^{3D} \EEE(t) \Vert_{L^2(\Omega_h)}
&= \Big\|  g_h^{3D}\EEE(0) + \int_0^t  \partial_t g_h^{3D} \EEE(s) \di s \Big\|_{L^2(\Omega_h)} \nonumber \\
& \leq \Vert f_h^{3D}(0) \Vert_{L^2(\Omega)} h^{1/2} + \int_I \Vert  \partial_t \EEE f_h^{3D}(s) \Vert_{L^2(\Omega)} h^{1/2} \di s
\leq C h^3 h^{1/2} = C h^{7/2}.
\end{align}
By Hölder's inequality, \eqref{ineq:force}, \eqref{forceonthindomain}, and Young's inequality  we derive that \EEE
\begin{align}\label{forceonthindomain_part}
\left \vert \int_{\Omega_h}  g_h^{3D} \EEE(t) (w_3 - x_3) \dx \right \vert
&\leq \Vert  g_h^{3D} \EEE(t) \Vert_{L^2(\Omega_h)} \Vert w_3 - x_3 \Vert_{L^2(\Omega_h)} \notag \\
&\leq C  h^{7/2} h^{-1} \mathcal{M}_h(w)^{1/2}
\leq \OOO \frac{1}{2} \EEE \mathcal{M}_h(w) + \OOO C \EEE h^5.
\end{align}
 Therefore, \OOO using the definition of $\mathcal F_h^t$, we have \EEE
\begin{equation*}
\mathcal{M}_h(w)
= \mathcal F^t_h(w) +  \int_{\Omega_h}  g^{3D}_h\EEE(t) (w_3 - x_3) \dx  \leq  \mathcal F^t_h(w) + \frac{1}{2} \mathcal{M}_h(w) + C h^5.
\end{equation*}
 This shows $\mathcal{M}_h(w) \le 2\mathcal F^t_h(w) + Ch^5$, and employing \eqref{forceonthindomain_part} once again, we have \EEE
\begin{equation*}
\left\vert \int_{\Omega_h}  g^{3D}_h\EEE(t) (w_3 - x_3) \dx \right\vert
\leq \OOO \frac{1}{2}  \mathcal{M}_h(w)\EEE + C h^5 \leq \OOO\min\EEE\{ \mathcal F^t_h(w), \mathcal{M}_h(w) \} +  C h^{5}.
\end{equation*}
\ZZZ Thus, \eqref{forceestimate} holds. \EEE
\end{proof}

\subsection{A priori estimates for the regularized solution}\label{sec: 4.2}


  For the following, it is convenient to introduce the  \ZZZ $\alpha$-dependent \EEE \emph{total energy}:  for $\alpha \in [2,4]$ fixed, we  define \EEE $\toten^{(\alpha)}_h \colon \Wid \times L^{4/\alpha}_+(\Omega_h) \to \R_+$ by
\begin{equation}\label{toten}
  \toten^{(\alpha)}_h(w, \vartheta) \defas \mechen(w) +  \mathcal{W}^{{\rm in}, \alpha}_h(w,\vartheta) \quad \text{ with } \quad \mathcal{W}^{{\rm in}, \alpha}_h(w,\vartheta) \defas \frac{\alpha}{4} \int_{\Omega_h} \inten(\nabla w, \vartheta)^{4/\alpha} \di x.
\end{equation}

We emphasize that the exponent $4/\alpha$  in the internal energy \REV is of a purely technical nature. It is \EEE introduced to ensure that $\mechen$ and $\mathcal{W}^{{\rm in}, \alpha}_h$ are of the same order in $h$. In fact, the mechanical energy is of order $h^4$ per unit volume, \REV see \eqref{boundres:mech}, \EEE whereas $ \inten(\nabla w, \vartheta) \sim h^\alpha$, \REV see \eqref{inten_lipschitz_bounds} and \eqref{boundres:temperature}. \EEE Due to the exponent $4/\alpha$  and the fact that $\OOO|\EEE\Omega_h\OOO|\EEE \sim h$, we can therefore expect that both terms of  $\toten^{(\alpha)}_h(w, \vartheta)$ scale like $h^5$. 
\REV Additionally, the integrability of the temperature is improved, see Remark~\ref{rem:improved_temp_bounds} below, needed for the limiting passage. \EEE For $\alpha = 4$, we shortly \OOO write \EEE $\toten_h \OOO\defas\EEE \toten^{(4)}_{ h \EEE}$. We also refer to \cite[Section 3.3]{RBMFMK} for a further discussion in this direction. 

\OOO After a change of coordinates, \ZZZ the bound $ \sup\nolimits_{h>0} \GGG ( \EEE h^{-4}\mathcal{M}(y_0^h) +   h^{-2\alpha}  \Vert \theta_0^h \Vert_{L^{2}(\Omega)}^{2} \GGG ) \EEE < + \infty $,  \eqref{forcebound}--\eqref{assumption:temp}, \GGG and \eqref{inten_lipschitz_bounds} \EEE directly lead  to the following bounds on the rescaled quantities: \EEE
\begin{align}
      \mathcal{E}^{(\alpha)}_h(w_0^h,\vartheta_0^h) &\leq C_0 h^5 \label{improved_bounds_energy_initial}, \\
      \Vert  g^{3D}_h \EEE \Vert_{W^{1,1}(I;L^2( \Omega_h))} &\leq C_0 \ZZZ h^{1/2 + 3}, \EEE \qquad
      \ZZZ \Vert  \vartheta_\flat^h  \Vert_{L^2(\GGG I \times \Gamma_h \EEE)} \leq C_0 h^{1/2 + \alpha}, \EEE \label{improved_bounds_external}
    \end{align}
    for a constant $C_0 > 0$ independent of $h$.  All constants we encounter in the rest of the section are always \EEE independent of $h$, $\eps$,   and the time $t$, \EEE but might depend on $T$. 
\begin{proposition}[Bounds on the total energy]\label{lem: first a priori estimates}
Given $h \in (0, 1]$, let $w_0^h \in \mathcal W_\id$, $\vartheta_0^h \in L_+^2(\Omega_h)$,  $g^{3D}_h \in W^{1, 1}(I; L^2(\Omega_h))$, \EEE and  $\vartheta_\flat^h\EEE \in \ZZZ L^{2}(I;L^2_+(\Gamma_h)) \EEE$ such that  \eqref{improved_bounds_energy_initial}--\eqref{improved_bounds_external} hold. \EEE
Then, there exist constants $C>0, \, \rho>0$\OOO, \EEE and $h_0 \in (0, 1)$ such that for every $\eps \in (0, 1)$, $h \in (0,  h_0)\EEE$, and \ZZZ any \EEE weak solution $(w_\eps,\vartheta_\eps)$ of the $\eps$-regularized problem in the sense of Definition~\ref{def:weak_formulation_reg} it holds that
\begin{subequations}
\begin{align}
  \esssup_{t \in I} \mathcal{E}^{(\alpha)}_h(w_\eps(t),\vartheta_\eps(t)) &\leq C h^5, \label{bound:total} \\
    \Vert\nabla w_\eps\Vert_{L^\infty(I \times \Omega_h)}
  + \Vert(\nabla w_\eps)^{-1}\Vert_{L^\infty(I \times \Omega_h)} &\leq C, \label{Linfty_strain_inverse_bound}\\
    \essinf\nolimits_{ t \in I}  \inf\nolimits_{x \in \Omega_h } \EEE \det (\nabla w_\eps  (t,x)) \EEE  &\geq \rho,  \label{bound:determinant} \\
  \Vert    \vartheta_\eps \EEE \Vert_{L^\infty(I;L^1(\Omega_h))} & \leq C h^{1+\alpha}. \label{bound:temperaturel1}
\end{align}
\end{subequations}
\end{proposition}

\begin{proof} 
In \ZZZ Step 1  we derive a suboptimal   a priori bound on the total energy. In Step 2, we deduce uniform bounds on the strain (see \eqref{Linfty_strain_inverse_bound}--\eqref{bound:determinant}) which then in Step 3 allows us to obtain the optimal control \eqref{bound:total}. Eventually, in Step 4 we address the bound \eqref{bound:temperaturel1} on the temperature. \EEE 

\emph{Step 1 (Preliminary bound on the total energy):}
 Let us \EEE proceed similarly to \cite[Lemma 6.2]{MielkeRoubicek20Thermoviscoelasticity}.
 \GGG For a.e.~$t \in I$, \EEE we can test \eqref{weak_form_heat_unrescaled_reg} with $\vphi(s, x) \defas \indic_{[0, t]}(s)$ resulting in
\begin{align}\label{balancereg:testwith1}
  &\mathcal{W}^{\rm{in}}( w_\eps(t), \vartheta_\eps(t)) - \int_0^t \int_{\Omega_h}
      \xi_{\eps, \alpha}^{\rm reg}(\nabla w_\eps,  \partial_t \EEE \nabla w_\eps, \vartheta_\eps)
      + \partial_F W^{\rm{cpl}}(\nabla w_\eps, \vartheta_\eps) :  \partial_t \EEE \nabla w_\eps \di x \di s \nonumber\\
    &\quad=
    \mathcal{W}^{\rm{in}}(w_\eps(0), \vartheta_\eps(0)) + \kappa \int_0^t \int_{ \Gamma_h} (\vartheta_{\flat, \eps} - \vartheta_\eps)  \di \haus^2  \di s.
\end{align}
\ZZZ \GGG We now define for a.e.~$t \in I$  \EEE
\begin{equation}\label{def_Ealpha}
   E(t) \defas \mathcal{E}_h\EEE (w_\eps(t),\vartheta_\eps(t)) -   \int_{\Omega_h}  g\EEE^{3D}_h(t) (  ({w_\eps})_3 (t) \EEE - x_3) \dx.
 \end{equation}
Adding \eqref{balancereg:testwith1} to \eqref{energybalanceregularized}, integrating by parts, and using $\vartheta_\eps \geq 0$, $\xi_{\eps, \alpha}^{\rm reg} \leq \xi$, and $\vartheta_{\flat, \eps} \leq \vartheta_\flat^h\EEE$ we find that
\begin{equation}\label{Ebound}
 E\EEE(t)
  \leq  E\EEE(0)
    + \int_0^t \left( \int_{ \Gamma_h} \kappa \vartheta_\flat^h\EEE \di \mathcal{H}^2 - \int_{\Omega_h}  \partial_t g\EEE^{3D}_h(s) (({w_\eps})_3 - x_3) \dx  \right) {\rm d} s.
\end{equation}
 Consider now \EEE $m \geq 0$ and $h \geq 0$.
If $m \geq h^3$ it  follows \EEE that $\sqrt{h^{-3}m} \leq h^{-3} m$  and hence \EEE $\sqrt{m} \leq h^{-3/2} m$.
If $m \leq h^3$ we \OOO instead \EEE have $\sqrt{m} \leq h^{3/2}$.
Combining both statements leads to $\sqrt{m} \leq h^{-3/2} m + h^{3/2}$, and therefore,  with  $m = \frac{1}{h^2} \mechen(w_\eps(s)) $, \ZZZ we get  \EEE
\begin{equation*}
   \frac{\sqrt{\mechen(w_\eps(s))}}{h} \EEE \leq \frac{\mechen(w_\eps(s))}{h^{7/2}} + h^{3/2}
\end{equation*}
for  \GGG a.e.~$s \in I$. \EEE
In view of \eqref{ineq:force}, we thus get
\begin{align} \label{forceboundcalc}
  \int_0^t \int_{\Omega_h}  \partial_t g\EEE^{3D}_h(s) (({w_\eps})_3 - x_3) \dx  \di s
  &\leq  \int_0^t \Vert   \partial_t g\EEE^{3D}_h(s) \Vert_{L^2(\Omega_h)} \Vert w_\eps(s) - x_3 \Vert_{\OOO L^2\EEE(\Omega_h)} \di s \nonumber \\
  &\leq  \int_0^t \Vert  \partial_t g\EEE^{3D}_h(s) \Vert_{L^2(\Omega_h)} h^{-1} \sqrt{\mechen(w_\eps(s))} \di s \nonumber \\
  &\leq C \int_0^t h^{-7/2} \Vert  \partial_t g\EEE^{3D}_h(s) \Vert_{L^2(\Omega_h)} (\mathcal{M}_h(w_\eps(s)) + h^5) \di s.
\end{align}
 Moreover, by  assumption   \eqref{improved_bounds_external}  and Hölder's inequality \ZZZ  we have 
\begin{equation}\label{hatten wir schon}
  \Vert \vartheta_\flat^h \Vert_{L^1(I; L^1(\Gamma_h))}
   \leq   \big(T \,  \mathcal{H}^2(\Gamma_h)\big)^{1/2} \Vert \vartheta_\flat^h \Vert_{L^{2}(I; L^{2}(\Gamma_h))} \EEE
  \leq C h^{1 + \alpha}. \EEE
\end{equation}
\EEE
Shortly writing $\tilde g_h(s) \defas h^{-7/2} \Vert  \partial_t \EEE g^{3D}_h(s) \Vert_{L^2(\Omega_h)}$ we derive \EEE using \eqref{forceestimate}, \eqref{Ebound}, and \eqref{forceboundcalc} that
\begin{equation}\label{for once later}
   E\EEE(t) \leq  E\EEE(0) + C h^{1+\alpha} + C \int_0^t  \tilde g_h\EEE(s) (E(s) + h^5) \di s.\EEE
\end{equation}
Note that \ZZZ by \EEE  \eqref{improved_bounds_external} \EEE we have $\int_0^t  \tilde g_h\EEE(s) \di s \leq C$. \OOO Moreover,  \eqref{forceestimate} \EEE and  \eqref{improved_bounds_energy_initial} \ZZZ (for $\alpha = 4$) \EEE also \OOO show \EEE  $E(0) \leq \ZZZ C \EEE h^5$. \EEE
Then, by Gronwall's inequality (in integral form)    we derive that  
\begin{equation*}
 E\EEE(t)
  \leq \left( E\EEE(0) + C h^{1+\alpha} + C  h^5 \EEE \int_0^t \ZZZ \tilde g_h \EEE(s) \di s\right)
   \exp\left(C \int_0^t \ZZZ \tilde g_h \EEE(s) \di s \right)  
  \leq C (h^5 + h^{1+\alpha}). \EEE
\end{equation*}
 T\EEE he previous estimate  together with \EEE \eqref{forceestimate}  implies 
\begin{align}\label{preliminary_mech_bound}
\mathcal{E}_h(w_\eps(t),\vartheta_\eps(t)) &\leq C (h^5 + h^{1+\alpha}) \qquad \GGG \text{for a.e.~} t\in I\EEE,
\end{align}
which gives \EEE  \eqref{bound:total} in the case $\alpha = 4$.   \ZZZ We still need \EEE to prove \eqref{bound:total} for $\alpha \in [2, 4)$.  To this end, we need to repeat the estimates with \OOO $\mathcal{W}^{{\rm in}, \alpha}_h$ \EEE   instead of \OOO $\mathcal{W}^{{\rm in}, 4}_h\EEE$ in order to obtain the right scaling $h^5$ for the energy $\mathcal{E}^{(\alpha)}_h$. \EEE  Before we deal with this  task, \EEE let us first derive \eqref{Linfty_strain_inverse_bound} \ZZZ and \eqref{bound:determinant} \EEE which can be  in fact shown  already  with  \eqref{preliminary_mech_bound}\EEE.

\emph{Step 2 ($L^\infty$-bound on the strain and its inverse):}  
 We  \EEE  use  \eqref{preliminary_mech_bound} \EEE   to conclude \eqref{Linfty_strain_inverse_bound} for sufficiently small $h$.
To this end,   let us \EEE fix $t \in I$  for which \eqref{preliminary_mech_bound} holds true\EEE.
Let us shortly write $F_\eps(t) \defas \fint_{\Omega_h} \nabla w_\eps(t, x\EEE) \di x$, $G_\eps(t, x) \defas \nabla w_\eps(t, x) - F_\eps(t)$  for $x \in \Omega_h$ \EEE as well as $\tilde G_\eps(t,  \tilde x \EEE) \defas G_\eps(t, \tilde x_1 \EEE, \tilde x_2\EEE, h \tilde x_3\EEE)$ for $\tilde x \EEE \in \Omega$. \EEE
Using $p > 3$, Morrey's and Poincaré's inequality, a change of variables, assumption \ref{H_bounds}, and  \eqref{preliminary_mech_bound} \EEE we derive  that \EEE
\begin{align*}
  \left\Vert \nabla w_\eps(t)\EEE - F_\eps(t)\EEE \right\Vert _{L^\infty(\Omega_h)}
  &=  \Vert \tilde G_\eps  (t) \EEE  \Vert _{L^\infty(\Omega)}
  \leq C \Vert \tilde G_\eps (t)\EEE \Vert_{W^{1,p}(\Omega)}
  \leq C \Vert \nabla \tilde G_\eps (t)\EEE \Vert_{L^p(\Omega)} \nonumber \\
  &\leq C h^{-1/p} \Vert \nabla G_\eps (t)\EEE \Vert_{L^p(\Omega_h)} = C h^{-1/p} \Vert \nabla^2 w_\eps (t)\EEE \Vert_{L^p(\Omega_h)} \\
  &\leq C h^{-1/p}   \mechen(w_\eps(t))^{1/p} \EEE \OOO \leq C h^{-1/p} (h^{5/p} + h^{(1 + \alpha)/p}) \EEE \leq C h^{2/p},
\end{align*}
 where in the last step we also used that $\alpha\ge 2$. \EEE Let $\bar Q_\eps \in SO(3)$ be such that $|F_\eps\OOO(t)\EEE - \bar Q_\eps| = \dist(F_\eps\OOO(t)\EEE, SO(3))$.
Then, with the aforementioned bound, \ref{W_lower_bound_spec},   \eqref{preliminary_mech_bound},  and $\alpha \ge 2$ \EEE we see that
\begin{align*}
  \lVert \nabla w_\eps (t)\EEE - \bar Q_\eps \rVert _{L^\infty(\Omega_h)}
  &\leq
    \lVert \nabla w_\eps (t)\EEE - F_\eps (t)\EEE \rVert _{L^\infty(\Omega_h)}
    + |F_\eps (t)\EEE - \bar Q_\eps| \\
  &= C h^{2/p} + \Big(\fint_{\Omega_h} \dist^2(F_\eps (t)\EEE, SO(3)) \di x\Big)^{1/2} \\
  &\leq C h^{2/p} +
   \Big(2 \lVert \nabla w_\eps (t)\EEE - F_\eps (t)\EEE  \rVert _{L^\infty(\Omega_h)}^2
    +  2\fint_{\Omega_h}\EEE \dist^2(\nabla w_\eps (t)\EEE, SO(3)) \di x\Big)^{1/2} \\
  &\leq C \big(h^{2/p} + h^{-1/2} \sqrt{\mechen(w_\eps (t)\EEE)} \big)
  \leq C (h^{2/p} + h).
\end{align*}
 As $\det(\bar Q_\eps) = 1$, by the  local  Lipschitz-continuity of the determinant  we derive that \EEE there exists some $h_0>0$ independent of $\eps$ such that
$\Vert \det(\nabla w_\eps) - 1 \Vert_{L^\infty(\Omega_h) } \leq \frac{1}{2}$ for $h \in (0, h_0)$.  This implies \EEE  \eqref{bound:determinant} and
\begin{align*}
\Vert (\nabla w_\eps)^{-1} \Vert_{L^\infty(\Omega_h) } &\leq \Vert \det (\nabla w_\eps)^{-1} \Vert_{L^\infty(\Omega_h)} \Vert {\rm adj}(\nabla w_\eps) \Vert_{L^\infty(\Omega_h)} \leq C ,
\end{align*}
where ${\rm adj}(F)\in  \R^{3\times 3} \EEE $ denotes the adjugate matrix of $F\in  \R^{3\times 3} \EEE  $. This shows \eqref{Linfty_strain_inverse_bound}.
\EEE

\emph{Step 3 \GGG (Optimal b\EEE ound on the total energy):}
We now show the total energy bound \eqref{bound:total} with optimal scaling in $h$  in the case \EEE $\alpha \in [2,4)$.
Without further  notice, \EEE we assume that $h \in (0, h_0)$ with $h_0$ as in Step~2.
In the derivation of \eqref{bound:total}, we will follow  the lines of \EEE the proof of \cite[Lemma 3.15]{RBMFMK}\EEE.
 Nevertheless, t\EEE here are two main differences: On the one hand, we need to make sure that all constants are independent of the \ZZZ thickness $h$. \EEE
On the other hand, our setting is time-continuous while the one in \cite{RBMFMK} is time-discrete.

Consider the scalar function $\chi(s) \defas \alpha / 4 (h^\alpha + s)^{4/\alpha}$ for $s \geq 0$ and define $\vphi(s, x) \defas \indic_{[0, t]}(s)\EEE\chi'(m_\eps(s, x))$ for $m_\eps \defas \inten(\nabla w_\eps, \vartheta_\eps)$.
 Our goal is to show \EEE that $\vphi$ is an admissible test function for \eqref{weak_form_heat_unrescaled_reg},  i.e., $\varphi \in L^2(I; H^1(\Omega_h))$.  Indeed,  as $\chi'(m_\eps) = (h^\alpha + m_\eps)^{4/\alpha - 1}$ and $4/\alpha - 1 \leq 1$, we directly get \EEE $\chi'(m_\eps) \in L^2(I \times \Omega_h)$  by \eqref{inten_lipschitz_bounds}  and  $\vartheta_\eps \in  L^{2}(I; H^1(\Omega_h)) $ (see Definition~\ref{def:weak_formulation_reg}). \EEE 
Moreover, using \EEE $\chi''(m_\eps) = (4/\alpha - 1) (h^\alpha + m_\eps)^{4/\alpha - 2}$ and $4/\alpha - 2 \leq 0$,  it holds that  \EEE $\chi''(m_\eps) \in L^\infty(I \times \Omega_h)$.
\OOO With the chain rule we compute
\begin{equation}\label{chain_rule_nabla_inten}
  \nabla m_\eps = \partial_F W^{\rm in}(\nabla w_\eps, \vartheta_\eps) \GGG : \EEE\nabla^2 w_\eps + \partial_\vartheta W^{\rm in} (\nabla w_\eps, \vartheta_\eps) \nabla \vartheta_\eps.  
\end{equation}
\EEE
Consequently, \EEE we obtain $\nabla \vphi = \indic_{[0, t]} \chi''(m_\eps) \nabla m_\eps \in L^2(I\times\Omega_h)$ \OOO using  \eqref{sec_deriv}, \eqref{est:coupl}, \eqref{Linfty_strain_inverse_bound}, \ref{H_bounds}, and $\vartheta_\eps \in L^{2}(I; H^1(\Omega_h)) $. \EEE
This shows the  admissibility of \EEE $\vphi$ as a test function in \eqref{weak_form_heat_unrescaled_reg}.

 We continue by \EEE applying the chain rule  from \cite[Proposition 3.5]{MielkeRoubicek20Thermoviscoelasticity} \EEE to the convex functional $\mathcal{J}(\cdot) = \int_{\Omega_h} \chi(\cdot) \di x$ for $m_\eps$, \ZZZ where we recall that  \EEE$m_\eps \in L^2(I;H^1(\Omega_h)) \cap H^1(I; \REV (H^1(\Omega_h))^* \EEE)$ \ZZZ by Definition \ref{def:weak_formulation_reg}. \EEE This along with  \eqref{weak_form_heat_unrescaled_reg} (tested  with $\varphi$)  leads to 
\begin{align}\label{balancetestwithchi}
  \int_{\Omega_h} \chi(m_\eps(t)) \dx \EEE
  &= \int_{\Omega_h} \chi(m_\eps(0)) \dx
    + \kappa \int_0^t \int_{ \Gamma_h}
      (\vartheta_{\flat, \eps} \EEE - \vartheta_\eps) \chi'(m_\eps) \di \haus^2 \di s \notag \\
  &\phantom{\leq}\quad + \int_0^t \int_{\Omega_h}
    \Big(
      \xi_{\eps, \alpha}^{\rm reg}(\nabla w_\eps,  \partial_t \EEE \nabla w_\eps, \vartheta_\eps)
      + \partial_F W^{\rm{cpl}}(\nabla w_\eps, \vartheta_\eps) :  \partial_t \EEE \nabla w_\eps
    \Big) \chi'(m_\eps) \di x \di s \notag \\
  &\phantom{\leq}\quad - \int_0^t \int_{\Omega_h}  \hcmnoh \EEE(\nabla w_\eps, \vartheta_\eps) \nabla \vartheta_\eps \cdot \nabla \chi'(m_\eps) \di x \di s
\end{align}
\GGG for a.e.~$t \in I$. \EEE
\ZZZ By \EEE the definition of $\chi$ and the definition of  $\OOO\mathcal{W}^{\rm in, \alpha}_h\EEE$ in \eqref{toten} \ZZZ it holds that \EEE 
\begin{align}\label{eq: to conlc}
\mathcal{W}^{{\rm in}, \alpha}_h\big(w_\eps(t),\vartheta_\eps(t)\big) \le   \int_{\Omega_h} \chi(m_\eps(t)) \dx \qquad \GGG \text{for a.e.~} t \in I. \EEE
\end{align}  \EEE
\ZZZ  We now \EEE   derive a bound for every term appearing on the right-hand side of \eqref{balancetestwithchi}. \ZZZ With this, \eqref{eq: to conlc} allows us to control $\mathcal{W}^{{\rm in}, \alpha}_h(w_\eps(t),\vartheta_\eps(t))$. \EEE By the definition of $\chi$  and the fact that $\OOO|\Omega_h\OOO| \le Ch$, \EEE  we first obtain
\begin{equation}\label{reg:est:1}
  \int_{\Omega_h} \chi(m_\eps(0)) \dx
  = \int_{\Omega_h} \frac{\alpha}{4} (h^\alpha + m_\eps(0))^{4/\alpha} \di x
  \leq C \left(  h^5  + \mathcal{W}^{{\rm in}, \alpha}_h(w_\eps(0), \vartheta_\eps(0)\GGG ) \EEE \right).
\end{equation}
 Using \EEE $\vartheta_{\flat,\eps}\leq \vartheta_\flat^h\EEE$\OOO, $\chi' \geq 0$,  \eqref{inten_lipschitz_bounds},   \eqref{improved_bounds_external} \ZZZ (with H\"older's inequality), \EEE   Young's inequality with powers $4/\alpha$ and $4/(4-\alpha)$ and constant $\lambda = C_0^{-1} 2^{-4/\alpha}$, and   the inequality $\vert a+b \vert^{ q} \leq 2^{ q}( \vert a\vert^{ q} + \vert b \vert^{ q} )$ for $q \geq  1$,  it holds that
\begin{align}\label{dont remove}
  &\kappa \int_0^t \int_{ \Gamma_h}
    (\vartheta_{\flat,\eps} - \vartheta_\eps) \chi'(m_\eps)
  \di \haus^2 \di s
  \leq \int_0^t \int_{ \Gamma_h}
    (\vartheta_\flat^h\EEE - C_0^{-1} m_\eps) (h^\alpha + m_\eps)^{4/\alpha - 1}
  \di \haus^2 \di s \notag \\
  &\quad\leq  \int_0^t \int_{ \Gamma_h} \Big(\vartheta_\flat^h\EEE (h^\alpha + m_\eps)^{4/\alpha - 1} - C_0^{-1}  m_\eps^{4/\alpha}  \Big) \di \haus^2 \di s  \notag \\
  &\quad\leq
     \int_0^t \int_{ \Gamma_h}  \Big( C_\lambda(\vartheta_\flat^h\EEE)^{4/\alpha}
    +   2^{4/\alpha} \EEE    \lambda ( h^4 + m_\eps^{4/\alpha} ) -C_0^{-1}   m_\eps^{4/\alpha} \Big) \di \haus^2 \di s   \leq C  h^5.
\end{align}
\EEE Recall  the definitions \EEE of $\xi^{\rm reg}_{\eps, \alpha}$ in \eqref{dissipation:trucated2} and $\xi^{(\alpha) }$ in \eqref{diss_rate_truncated}, respectively\EEE.  W\EEE e have $(\xi^{\rm reg}_{\eps, \alpha})^{4 / \alpha} \leq (\xi^{(\alpha)})^{4/\alpha} \leq \xi$.
Hence, \EEE by Young's inequality with powers $4/\alpha$ and $4/(4-\alpha)$  and constant $\lambda \in (0, 1)$, as well as the definition of $\OOO\mathcal W_h^{\rm in, \alpha}\EEE$ in \eqref{toten} \EEE we can estimate
\begin{align}\label{reg:est:3}
  &\int_0^t \int_{\Omega_h}
    \xi^{\rm reg}_{\eps, \alpha}(\nabla w_\eps,  \partial_t \EEE \nabla w_\eps, \vartheta_\eps)
    \chi'(m_\eps)
  \di x \di s \nonumber \\
  &\quad\leq  \lambda \EEE \int_0^t \int_{\Omega_h}
      \xi(\nabla w_\eps,  \partial_t \EEE \nabla w_\eps, \vartheta_\eps)
    \di x \di s
    +  C_\lambda \EEE \int_0^t \int_{\Omega_h}
      (h^4+ m_\eps^{4/\alpha})
    \di x \di s  \nonumber  \\
  &\quad\leq  \lambda \EEE \int_0^t \int_{\Omega_h}
      \xi(\nabla w_\eps,  \partial_t \EEE \nabla w_\eps, \vartheta_\eps)
    \di x \di s
    + C h^5 + C \int_0^t
      \mathcal{W}^{{\rm in}, \alpha}_h(w_\eps(s), \vartheta_\eps(s))
    \di s.
\end{align}
In view of \eqref{hcm}--\eqref{spectrum_bound_K}, \eqref{Linfty_strain_inverse_bound}, and  \eqref{bound:determinant}, $ \hcmnoh \EEE(\nabla w_\eps,\vartheta_\eps)$ is uniformly bounded  from below \EEE (in the eigenvalue sense) \OOO for a.e.~$t \in I$ and every $x \in \Omega$\EEE.
Moreover, with \OOO \eqref{chain_rule_nabla_inten}  and  \eqref{sec_deriv} we \EEE derive that
\begin{align*}
  \nabla  \chi'(m_\eps) \EEE
  &= \chi''(m_\eps) \nabla m_\eps  = \left(\tfrac{4}{\alpha} - 1\right)
    (h^{\alpha} + m_\eps)^{4/\alpha-2}
    \Big(
      \partial_F W^{\rm in}(\nabla w_\eps, \vartheta_\eps) ) \GGG : \EEE \nabla^2 w_\eps
      - \vartheta_\eps \partial_\vartheta^2 W^{\rm cpl}(\nabla w_\eps, \vartheta_\eps) \nabla \vartheta_\eps
    \Big).
\end{align*}
Thus, employing \eqref{est:coupl}, \eqref{Linfty_strain_inverse_bound}, \OOO and \EEE \eqref{sec_deriv} we find that
\begin{equation}\label{hcm_lower_bound}
   \hcmnoh \EEE(\nabla w_\eps, \vartheta_\eps) \nabla \vartheta_\eps \cdot \nabla \chi'(m_\eps)
    \geq (\tfrac{4}{\alpha} - 1) (h^{\alpha} + m_\eps)^{4/\alpha-2}
    \big(
      C^{-1} \vert \nabla \vartheta_\eps \vert^2
      -C (\vartheta_\eps \wedge 1)
        \vert \nabla^2 w_\eps \vert
        \vert \nabla \vartheta_\eps \vert
    \big).
\end{equation}
By $\vartheta \wedge 1 \leq \vartheta^{1-4/(\alpha p)}$ for all $\vartheta \geq 0$  (recall $\alpha \ge 2$ and $p > 4$), \EEE \eqref{inten_lipschitz_bounds}, Young's inequality twice (firstly with power $2$ and constant $\lambda$, secondly with powers $p/(p-2)$ and $p/2$) we derive that
\begin{align}\label{needed_for_imporoved_weighted_l2}
  (\vartheta_\eps \wedge 1)
  \vert \nabla^2 w_\eps \vert
  \vert \nabla \vartheta_\eps \vert
 & \leq \lambda \vert \nabla \vartheta_\eps \vert^2 + C_\lambda m_\eps^{2-8/(\alpha p)}  \vert \nabla^2 w_\eps \vert^2   = \EEE \lambda \vert \nabla \vartheta_\eps \vert^2
    + C_\lambda m_\eps^{2(p-2)/p} m_\eps^{4(\alpha-2)/(\alpha p)}
    \vert \nabla^2 w_\eps \vert^2 \notag \\
  &\leq \lambda \vert \nabla \vartheta_\eps \vert^2
    + C_\lambda
    \left(
      m_\eps^2
      + m_\eps^{2 - 4/\alpha} \vert \nabla^2 w_\eps \vert^p
    \right).
\end{align}
Combining the previous two estimates, using \OOO $4/\alpha - 2 \leq 0\EEE$  and \EEE choosing $\lambda < C^{-2} \EEE$   with $C$ as in \eqref{hcm_lower_bound}, we then get  by \eqref{mechanical},  \ref{H_bounds} \ZZZ and \eqref{toten} \EEE
\begin{align}\label{reg:est:5}
  - \int_0^t \int_{\Omega_h}
     \hcmnoh \EEE(\nabla w_\eps, \vartheta_\eps) \nabla \vartheta_\eps \cdot \nabla \chi'(m_\eps) \di  x \di s
  &\leq C \int_0^t \int_{\Omega_h}
    (h^{\alpha} + m_\eps)^{4/\alpha-2}
    (
      m_\eps^2
      + m_\eps^{2 - 4/\alpha} \vert \nabla^2 w_\eps \vert^p
    )
  \di x \di s \nonumber \\
  &\leq C \int_0^t \int_{\Omega_h}
 \big( m_\eps^{4/\alpha}
      + \vert \nabla^2 w_\eps \vert^p \big) 
    \di x \di s \nonumber \\
  &\leq   C \int_0^t \left(
      \mathcal{W}^{{\rm in}, \alpha}_h(w_\eps(s), \vartheta_\eps(s))
      + \mathcal M_{\OOO h \EEE}(w_\eps(s)) \right)
    \di s \nonumber \\
  & = \EEE C \int_0^t \mathcal{E}^{(\alpha)}_h (w_\eps(s), \vartheta_\eps(s)) \di s.
\end{align}\GGG Our \EEE next goal is to show that for any $\lambda \in (0, 1)$ there is a constant $C_\lambda$ independent of $h$ and $\eps$ such that
\begin{align}\label{apriori_adiabatic}
  &\int_0^t \int_{\Omega_h}
      |\partial_F W^{\rm cpl}(\nabla w_\eps, \vartheta_\eps) :  \partial_t \EEE \nabla w_\eps
      \chi'(m_\eps)|
    \di x \di s \notag \\
  &\quad\leq
  C_\lambda h^5 + C_\lambda \int_0^t \mathcal{W}^{{\rm in}, \alpha}_h (w_\eps(s),\vartheta_\eps(s)) \di s
    + \lambda \int_0^t \int_{\Omega_h}
      \xi(\nabla w_\eps, \partial_t \nabla w_\eps, \vartheta_\eps)
    \di x \di s. 
\end{align} \EEE
We  first deal with \EEE the case $\alpha \in (2,4)$.
By \eqref{avoidKorn}, \eqref{Linfty_strain_inverse_bound}, Young's inequality with powers \OOO $\alpha/(\alpha-2)$ \EEE and \OOO $\alpha/2$ \EEE and constant $\lambda \in (0, 1)$,   and the inequality $(a + b)^q \leq \GGG 2^q \EEE (a^q + b^q)$ for any $a, \, b \geq 0$ and $q \ZZZ =  4/\alpha - 1  \OOO\in (0, 1)$, \EEE we derive  
\begin{align*}
  &\int_0^t \int_{\Omega_h}
      |\EEE\partial_F W^{\rm cpl}(\nabla w_\eps, \vartheta_\eps) : \partial_t \nabla w_\eps\EEE
      \chi'(m_\eps) |\EEE
    \di x \di s \\
  &\quad\leq \int_0^t \int_{\Omega_h} \left(
      C_\lambda (\vartheta_\eps \wedge 1)^{\alpha/(\alpha-2)}
      + \lambda \xi( \nabla \EEE w_\eps,  \partial_t \nabla \EEE w_\eps, \vartheta_\eps)^{\alpha/4}
    \right)
    (h^\alpha + m_\eps)^{4/\alpha-1}
  \di x \di s \\
    &\quad\leq \GGG C \EEE \int_0^t \int_{\Omega_h} \left(
      C_\lambda (\vartheta_\eps \wedge 1)^{\alpha/(\alpha-2)}
      + \lambda \xi( \nabla \EEE w_\eps,  \partial_t \nabla \EEE w_\eps, \vartheta_\eps)^{\alpha/4}
    \right) \OOO (h^{4-\alpha} + m_\eps^{4/\alpha-1}) \EEE
  \di x \di s.
\end{align*}
\OOO Consequently, \GGG using \EEE $\vartheta \wedge 1 \leq \vartheta^{(\alpha-2)/\alpha}$ for $\vartheta \geq 0$, \EEE \eqref{inten_lipschitz_bounds} and Young's inequality with powers $4/\alpha$ and $4/(4-\alpha)$
\begin{align*}
  & \int_0^t \int_{\Omega_h}
      |\EEE\partial_F W^{\rm cpl}(\nabla w_\eps, \vartheta_\eps) : \partial_t \nabla w_\eps
      \chi'(m_\eps)|\EEE
    \di x \di s \EEE\\
   &\quad \leq C \int_0^t \int_{\Omega_h} \left(
      C_\lambda (m_\eps h^{4-\alpha} + m_\eps^{4/\alpha})
      +  \lambda \xi(\nabla w_\eps, \partial_t \nabla w_\eps, \vartheta_\eps)^{\alpha/4}(h^{4 - \alpha} + m_\eps^{(4 - \alpha)/\alpha}) \EEE
    \right) \di x \di s \\
  &\quad\leq C \int_0^t \int_{\Omega_h}
   \Big(    C_\lambda (h^4 +  m_\eps^{4/\alpha})\EEE
      +  \lambda \xi( \nabla \EEE w_\eps,  \partial_t \nabla \EEE w_\eps, \vartheta_\eps) \Big)
    \di x \di s \\
  &\quad\leq  C_\lambda \EEE h^5
    +  C_\lambda \EEE \int_0^t \mathcal{W}^{{\rm in}, \alpha}_h (w_\eps(s),\vartheta_\eps(s)) \di s
    +  C \EEE \lambda \int_0^t \int_{\Omega_h}
      \xi( \nabla \EEE w_\eps,  \partial_t \nabla \EEE w_\eps, \vartheta_\eps)
    \di x \di s. 
\end{align*}
 This applied for $\lambda/C$ in place of $\lambda$ gives \EEE \eqref{apriori_adiabatic} for $\alpha \in (2, 4)$. \EEE
In the case $\alpha = 2$, we  similarly \EEE derive by \eqref{avoidKorn}, \EEE \eqref{Linfty_strain_inverse_bound}, and Young's inequality with power $2$ and constant $\lambda \in (0,1)$ \EEE that
\begin{align*}
      &\int_0^t \int_{\Omega_h}
          |\EEE\partial_F W^{\rm cpl}(\nabla w_\eps, \vartheta_\eps) :  \partial_t \nabla \EEE w_\eps
          \chi'(m_\eps) |\EEE
        \di x \di s \\
      &\quad\leq C\int_0^t \int_{\Omega_h}
          (\vartheta_\eps \wedge 1) \xi( \nabla \EEE w_\eps,  \partial_t \nabla \EEE w_\eps, \vartheta_\eps)^{ 1/2}
        (h^{ 2} + m_\eps)
      \di x \di s \\
      &\quad\leq  C \EEE \int_0^t \int_{\Omega_h} \Big(
           C_\lambda (m_\eps^2 + h^4) \EEE
          +  \lambda \xi( \nabla \EEE w_\eps,  \partial_t \nabla \EEE w_\eps, \vartheta_\eps) \Big)
        \di x \di s  \\
      &\quad\leq  C_\lambda h^5 \EEE +  C_\lambda \int_0^t  \mathcal{W}^{\rm in,2}_{h} \EEE (w_\eps(s),\vartheta_\eps(s)) \di s
        +  C \EEE \lambda \int_0^t \int_{\Omega_h}
          \xi( \nabla \EEE w_\eps,  \partial_t \nabla \EEE w_\eps, \vartheta_\eps)
        \di x \di s .
    \end{align*} \EEE
 In a similar fashion,  \ZZZ again using  \eqref{Linfty_strain_inverse_bound}, \EEE Young's inequality with power $2$ and constant $\lambda \in (0,1)$, and $\vartheta \wedge 1 \leq \vartheta^{2/\alpha}$ for $\vartheta \geq 0$ we also get
\begin{align}\label{eq: auch noch}
&\int_0^t \hspace{-0.05cm} \int_{\Omega_h} \hspace{-0.05cm}
          |\partial_F W^{\rm cpl}(\nabla w_\eps, \vartheta_\eps) :  \partial_t \nabla  w_\eps
        |         \di x \di s \notag \\ &\quad\le  C_\lambda \int_0^t \mathcal{W}^{{\rm in}, \alpha}_h (w_\eps(s),\vartheta_\eps(s)) \di s   +  \lambda \int_0^t \hspace{-0.05cm} \int_{\Omega_h} \hspace{-0.05cm}
      \xi( \nabla \EEE w_\eps,  \partial_t \nabla \EEE w_\eps, \vartheta_\eps)
    \di x \di s. 
        \end{align}
        \EEE    We are ready to conclude.
    Combining \eqref{balancetestwithchi}--\eqref{reg:est:3}, \EEE \eqref{reg:est:5}--\eqref{apriori_adiabatic}, and choosing   $\lambda = 1/3$  \EEE yields
    \begin{align*}
      \mathcal{W}^{{\rm in}, \alpha}_h(w_\eps(t),\vartheta_\eps(t))
      &\leq
        C h^5
        + C \mathcal{W}^{{\rm in}, \alpha}_h (w_\eps(0), \vartheta_\eps(0))
        + C \int_0^t \mathcal{E}^{(\alpha)}_h(w_\eps(s), \vartheta_\eps(s)) \di s \\
      &\phantom{\leq}\quad
          +   \frac{2}{3} \EEE  \int_0^t \int_{\Omega_h}
            \xi(\nabla w_\eps,  \partial_t \EEE \nabla w_\eps, \vartheta_\eps)
          \di x \di s . \EEE
    \end{align*}
  Summing the above inequality  and \EEE \eqref{energybalanceregularized},  using \eqref{eq: auch noch} for  $\lambda = 1/3$, \ZZZ and \eqref{improved_bounds_energy_initial} \EEE  we can derive similarly to \OOO the proof of \eqref{for once later} \EEE in Step~1
    \begin{equation*}
       E^{(\alpha)}(t) \EEE
        \leq E^{(\alpha)}(0) + C h^5 + C \int_0^t   \tilde{g}_h(s)  \big( E^{(\alpha)}(s) + h^5 \big) \di s,
    \end{equation*}
 \EEE   with   $\tilde{g}_h$ as in   \eqref{for once later}, where  $E^{(\alpha)}$ is defined as in \eqref{def_Ealpha} with $\mathcal{E}^{(\alpha)}_h$ in place of $\mathcal{E}_h = \mathcal{E}^{(4)}_h$. \EEE
    Then, Gronwall's inequality,  the assumption \eqref{improved_bounds_energy_initial}, \EEE and \eqref{forceestimate} lead to \eqref{bound:total}.

    \emph{Step 4 (Bound on temperature):} \EEE
It remains to show \eqref{bound:temperaturel1}.  In this regard\EEE, by Hölder's inequality \OOO with powers $(4-\alpha)/\alpha$ and $4/\alpha$\EEE, \eqref{inten_lipschitz_bounds}, \ZZZ \eqref{toten}, \EEE and \eqref{bound:total},  we have that \EEE
    \begin{align*}
    \int_{\Omega_h} \vert \vartheta_\eps(t) \vert \di x &\leq \OOO|\EEE\Omega_h\OOO|\EEE^{(4-\alpha)/4} \left(\int_{\Omega_h} \vert \vartheta_\eps(t) \vert^{4/\alpha} \di x \right)^{\alpha/4} \\
    &\leq C h^{(4-\alpha)/4} (\mathcal{E}^{(\alpha)}_h(w_\eps(t),\vartheta_\eps(t)))^{\alpha/4} \leq C h^{(4-\alpha)/4} h^{5  \alpha/4 \EEE} = C h^{1+\alpha}
    \end{align*}
\ZZZ for a.e.\ $t \in I$, \EEE    as desired.  
    \end{proof}

     W\EEE e next address  a priori \EEE bounds on the dissipation.
    \begin{lemma}[Bounds on the dissipation and  the \EEE strain rate]\label{lem: first a priori estimates2}
    Given $h > 0$, let $w_0^h \in \mathcal W_\id$, $\vartheta_0^h \in L^2_+(\Omega_h)$, $ g\EEE^{3D}_h \in W^{1, 1}(I; L^2_+(\Omega_{ h}))$, and  $\vartheta_\flat^h\EEE \in \ZZZ L^{2}(I;L^2_+(\Gamma_h)) \EEE$ such that   \eqref{improved_bounds_energy_initial}--\eqref{improved_bounds_external} hold. \EEE
    Then, there exist constant\OOO s \EEE $C>0$ and $h_0 \in (0,1]$ such that for all $\eps \in (0, 1)$, $h \in (0,   h_0 \EEE]$, and any weak solutions $(w_\eps, \vartheta_\eps)$ of the regularized problem in the sense of Definition~\ref{def:weak_formulation_reg} it holds that
    \begin{subequations}
    \begin{align}
      \int_I \int_{\Omega_h}
          R(\nabla w_\eps,  \partial_t \EEE \nabla w_\eps, \vartheta_\eps)
        \di x \di t
      &\leq C h^5, \label{bound:dissipation} \\
      \Vert  \partial_t \EEE \nabla w_\eps \Vert_{L^2(I\times \Omega_h)}
      &\leq C h^{3/2}. \label{bound:strainrate}
    \end{align}
    \end{subequations}
    \end{lemma}

    \begin{proof}
    The argument  follows along the lines of \EEE \cite[Lemma 6.2]{MielkeRoubicek20Thermoviscoelasticity}, \OOO but we additionally need  to ensure the independence of the constants on the \ZZZ  thickness $h$. \EEE 
    Let us first show \eqref{bound:dissipation}.
    By  \eqref{diss_rate}, \EEE  \eqref{energybalanceregularized}, and the nonnegativity of the mechanical energy it follows that
    \begin{align}\label{mech_bound_diss}
      &2\int_I \int_{\Omega_h}
        R(\nabla w_\eps,  \partial_t \EEE \nabla w_\eps, \vartheta_\eps) \di x \di t
      \notag \\
      &\quad \leq \mathcal{M}_h(w_\eps(0))
        + \int_I \int_{\Omega_h} \OOO g^{3D}_h \EEE \partial_t  ({w_\eps})_3 \EEE \dx  \di t
        -  \int_I \int_{\Omega_h}
          \partial_F W^{\rm cpl} (\nabla w_\eps, \vartheta_\eps)
            :  \partial_t \EEE \nabla w_\eps
        \di x \di t.
    \end{align}
     By the estimate in \eqref{eq: auch noch} for  $\lambda =1/4$ \EEE and \eqref{bound:total} we \OOO then derive \EEE
    \begin{align}\label{dont remove 2}
      \int_I \int_{\Omega_h}
        |\partial_F W^{\rm cpl} (\nabla w_\eps, \vartheta_\eps) :  \partial_t \EEE \nabla w_\eps|
      \di x \di t
      &\leq  C \int_{I} \mathcal{W}^{{\rm in}, \alpha}_h (w_\eps(\OOO t ),\vartheta_\eps(\OOO t ))  \di \OOO t \EEE
        + \frac{1}{4} \int_I \int_{\Omega_h}
            \xi(\nabla w_\eps,  \partial_t \EEE \nabla w_\eps, \vartheta_\eps)
          \di x \di t \notag \\
      &\leq C h^5 +  \frac{1}{2} \int_I \int_{\Omega_h}
        R(\nabla w_\eps,  \partial_t \EEE \nabla w_\eps, \vartheta_\eps)
      \di x \di t.
    \end{align}
 We now \OOO employ the generalized Korn's inequality from \EEE Theorem~\ref{pompescaling}\ref{item:kornsecondthin_dirichlet} for $u =    \partial_t  w_\eps(t)$ and $z = w_\eps(t)$. \GGG To this end, notice \EEE that by \eqref{bound:total} \GGG and  \EEE\ref{H_bounds} we have  $\Vert \nabla^2 w_\eps(t) \Vert^p_{L^p(\Omega_h)}\leq C h^5$ \EEE  for \GGG a.e.~$t \in I$. \GGG Moreover, \eqref{Linfty_strain_inverse_bound} and Poincaré's inequality imply that $w_\eps(t) \in W^{2,p}(\Omega_h;  \R^3)  $ for a.e.~$t \in I$. 
    Hence, \EEE \eqref{Linfty_strain_inverse_bound}--\eqref{bound:determinant} and $\partial_t w_\eps = 0$ on $\Gamma_D^h$ ensure that all assumptions of the generalized Korn's inequality are satisfied. 
  Thus, by using Young's inequality with power $2$ and constant $\OOO\lambda$ for some $  \lambda \in (0,1)$, Poincaré's inequality, \ref{D_quadratic}, \ref{D_bounds}, and \eqref{forceonthindomain} we derive that
\begin{align}\label{forceboundcalc2}
\int_I \int_{\Omega_h}  g^{3D}_h(s)   \partial_t   ({w_\eps})_3   \dx  \di t &\leq C_{\lambda} h^{-2} \int_I \Vert g^{3D}_h(t) \Vert_{L^2(\Omega_h)}^2 \di t + C   \lambda h^2 \int_I \Vert \partial_t \nabla w_\eps (t) \Vert_{L^2(\Omega_h)}^2 \di t \notag \\
&\leq C_{\lambda} h^5 + C   \lambda  \int_I \int_{\Omega_h}  R(\nabla w_\eps,   \partial_t   \nabla w_\eps, \vartheta_\eps) \di x \di t.
\end{align}    
 Thus, choosing $\lambda = (2C)^{-1}$ \OOO for $C$ as above  and \OOO using  \eqref{improved_bounds_energy_initial} \ZZZ along with \EEE \eqref{mech_bound_diss}--\eqref{forceboundcalc2}  \OOO yields  \eqref{bound:dissipation}.
Eventually, we can employ the generalized Korn's inequality in the version of Theorem~\ref{pompescaling}\ref{item:kornsecondthin_dirichlet}  once again \OOO to  obtain \eqref{bound:strainrate}. \EEE
    \end{proof}

    \subsection{Improved a priori estimates for the temperature  in the regularized setting\EEE}\label{sec: 4.3}
     In this section, \ZZZ we \EEE derive \EEE a priori estimates  for \EEE the temperature $\vartheta_\eps$ and the internal energy $m_\eps = \mathcal{W}^{\rm in} (w_\eps,\vartheta_\eps)$.       As a preliminary step, we  state and prove a \EEE version of the anisotropic Gagliardo-Nirenberg interpolation \OOO on thin domains\EEE, see \cite[Lemma 4.2]{MielkeNaumann}  for the corresponding statement on  a fixed domain. \EEE

    \begin{lemma}[Anisotropic Gagliardo-Nirenberg inequality on thin domains]\label{lem:gagliardonirenberg}
      For every $r \in (1, 3)$ there exists a constant $C_r>0$ such that for all $\vphi \in L^\infty(I; L^1(\Omega_h)) \cap L^r(I; W^{1, r}(\Omega_h))$ it holds that
      \begin{equation}\label{GL_thin_domain}
        \Vert \vphi \Vert_{L^{ 4r/3}(I \times \Omega_h)}
        \leq C_r h^{3/(4r) -1/4}
          \Vert \vphi \Vert_{L^{\infty}(I;L^{1}(\Omega_h))}^{1/4}
            \Big(
              h^{-1} \Vert \vphi \Vert_{L^{\infty}(I;L^{1}(\Omega_h))}
              + h^{-1/r} \Vert \nabla \vphi \Vert_{L^r(I \times \Omega_h)}
            \Big)^{3/4}.
      \end{equation}
    \end{lemma}

    \begin{proof}
       Let $r$ and $\vphi$ be as in the statement. \EEE
      We reason by rescaling.
       Consider \EEE $\tilde \vphi \in L^\infty(I; L^1(\Omega)) \cap L^r(I; W^{1, r}(\Omega))$ given by $\tilde \vphi(t, x) \defas \vphi(t, x', hx_3)$ for every $(t, x) \in I \times \Omega$.
      Applying \cite[Lemma 4.2]{MielkeNaumann} for $N=3$,   $\theta = 3/4$, \EEE   $s = p = 4r/3$, $s_1 = \infty$, $p_1 = 1$, and $s_2 = p_2 = r$ leads to the existence of a constant $C_r$ independent of  $\vphi$ \EEE and $h$ such that
      \begin{equation}\label{GL_thick_domain}
        \Vert \tilde \vphi \Vert_{L^{ 4r/3}(I \times \Omega)}
        \leq C_r
          \Vert \tilde \vphi \Vert_{L^{\infty}(I;L^{1}(\Omega))}^{ 1/4}
            \Big(
              \Vert \tilde \vphi \Vert_{L^{\infty}(I;L^{1}(\Omega))}
              + \Vert \nabla \tilde \vphi \Vert_{L^r(I \times \Omega)}
            \Big)^{ 3/4}.\EEE
      \end{equation}
       A \EEE change of coordinates  yields \EEE
      $        \lVert \vphi \rVert _{L^{ 4r/3\EEE}(I \times \Omega_h)}
          = h^{ 3 / (4r)} \lVert \tilde \vphi \rVert _{L^{4r/3}(I \times \Omega)}$,         $\lVert \vphi \rVert _{L^\infty(I; L^1(\Omega_h))}
          = h \lVert \tilde \vphi \rVert _{L^\infty(I; L^1(\Omega))}$, and $\lVert \nabla \tilde \vphi \rVert _{L^r(I \times \Omega)}
          \leq  h^{-1/r}  \lVert \nabla \vphi \rVert _{L^r(I \times \Omega_h)}$.       This along with  \eqref{GL_thick_domain} gives \EEE \eqref{GL_thin_domain}.
    \end{proof}

    \begin{lemma}[Improved bounds on the temperature]\label{lem:improvedregularityregularized}
    Given $h \in (0, 1]$, let $w_0^h \in \mathcal W_\id$, $\vartheta_0^h \in L^2_+(\Omega_h)$, $ g\EEE^{3D}_h \in W^{1, 1}(I; L^2_+(\Omega_h\EEE))$, and  $\vartheta_\flat^h\EEE \in \ZZZ L^{2}(I;L^2_+(\Gamma_h)) \EEE$   such that \eqref{improved_bounds_energy_initial}--\eqref{improved_bounds_external} hold. \EEE     Then, for every $q \in [1, 5/3)$ and $r \in [1, 5/4)$ there exist constants $C_q>0$, $C_r>0$,  $C>0$, and $h_0 \in (0,1)$ \EEE such that for all $h \in (0,  h_0 \EEE]$, $\eps \in (0,1)$, and all weak solutions $(w_\eps,\vartheta_\eps)$ of the regularized problem in the  sense \EEE of Definition~\ref{def:weak_formulation_reg} it holds that
    \begin{subequations}
    \begin{align}
      \Vert \vartheta_\eps \Vert_{L^q(I\times \Omega_h)}
        + \Vert m_\eps \Vert_{L^q(I\times \Omega_h)}
        &\leq C_q h^{\alpha + 1/q}, \label{reg:temperature} \\
      \Vert \nabla \vartheta_\eps \Vert_{L^r(I\times \Omega_h)}
        + \Vert \nabla m_\eps \Vert_{L^r(I\times \Omega_h)}
        &\leq C_r h^{\alpha + 1/r}, \label{reg:temperaturegradient}\\
      \Vert \partial_t \EEE m_\eps \Vert_{L^1(I; \OOO(\EEE H^3 \EEE (\Omega_h)\OOO)\EEE^*)}
        &\leq C  h^{\alpha + 1}, \label{forAubin-Lion}
    \end{align}
    \end{subequations}
  where again \OOO we have \EEE set $m_\eps := \mathcal{W}^{\rm in} (w_\eps,\vartheta_\eps)$. \EEE    \end{lemma}
    \begin{proof}
    The proof is divided \GGG into \EEE three steps. In Step 1, we show a weighted  $L^2$-bound \EEE on $\nabla m_\eps$. As in \cite{MielkeRoubicek20Thermoviscoelasticity} or \cite{RBMFMK}, we employ special test functions \ZZZ used \EEE by {\sc Boccardo and Gallou\"et} for the regularity theory of parabolic equations with a measure-valued right-hand side, see e.g.\ \cite{BoccardoGallouet89Nonlinear}\EEE.
    \OOO With \EEE Lemma~\ref{lem:gagliardonirenberg} we \OOO then \EEE derive  the bound on $\nabla m_\eps$ in  \eqref{reg:temperaturegradient}, see Step 2.  T\EEE he last step  addresses \EEE the remaining bounds.

    \emph{Step 1 (Weighted $L^2$-bound on the gradient):}
    For $\eta \in (0, 1)$, let $\chi_\eta\colon \R_{\OOO +\EEE} \to \R_{\OOO +\EEE}$ be given by $\chi_\eta(0) = 0$ and $\chi_\eta'(s) = 1- (1+ h^{-\alpha} s)^{-\eta}$. \ZZZ For  \EEE all $s \geq 0$,  $\chi_\eta$ satisfies  
    \begin{equation*}
      \chi_\eta(s) \leq s, \qquad \chi_\eta''(s) = \eta h^{-\alpha} (1+ h^{-\alpha} s)^{-1-\eta} \in (0, h^{-\alpha} \ZZZ ). \EEE
    \end{equation*}
     Along the lines of Step~3 in the proof of Proposition~\ref{lem: first a priori estimates}, \OOO we can show that the chain rule \cite[Proposition 3.5]{MielkeRoubicek20Thermoviscoelasticity} applies to $\mathcal{J}(m_\eps) \defas \int_{\Omega_h} \chi_\eta(m_\eps) \di x$  and that $ \chi_\eta'(m_\eps) \ZZZ \in L^2(I;H^1(\Omega_h)) \EEE$  is a valid test function for  \eqref{weak_form_heat_unrescaled_reg}. This leads to \EEE the identity
    \begin{align}\label{chi_eta_test}
      & \int_{\Omega_h} \EEE \chi_\eta(m_\eps(T))  \di x \EEE -  \int_{\Omega_h} \EEE\chi_\eta(m_\eps(0))  \di x \EEE
       \notag \\
      &\quad= - \int_I \int_{\Omega_h} \chi_\eta''(m_\eps) \nabla m_\eps
        \cdot  \hcmnoh \EEE(\nabla w_\eps,\vartheta_\eps) \nabla \vartheta_\eps \di x \di t
        + \kappa \int_I \int_{\Gamma_h\EEE} (\vartheta_{\flat,\eps}- \vartheta_\eps) \chi_\eta'(m_\eps) \di \haus^2 \di t \notag \\
      &\phantom{\quad=}\quad + \int_I \int_{\Omega_h}
        \chi_\eta'(m_\eps) \Big(
          \xi_{\eps, \alpha}^{\rm reg} (\nabla w_\eps,  \partial_t \EEE \nabla w_\eps, \vartheta_\eps)
          + \partial_F W^{\rm cpl} (\nabla w_\eps, \vartheta_\eps) :  \partial_t \EEE \nabla w_\eps
        \Big) \di x \di t.
    \end{align}
    Using $\chi_\eta' \leq 1$, $\xi_{\eps, \alpha}^{\rm reg} \leq \xi$, \eqref{avoidKorn}, \eqref{Linfty_strain_inverse_bound}, \eqref{bound:temperaturel1}, $\vartheta_\eps \wedge 1 \leq \vartheta_\eps^{1/2}$, Young's inequality, \eqref{bound:dissipation}, and $\alpha \leq 4$ we derive that
    \begin{align*}
      &\int_I \int_{\Omega_h}
        \chi_\eta'(m_\eps) \Big(
          \xi_{\eps, \alpha}^{\rm reg}(\nabla w_\eps,  \partial_t \EEE \nabla w_\eps, \vartheta_\eps)
          + \partial_F W^{\rm cpl} (\nabla w_\eps, \vartheta_\eps) :  \partial_t \EEE \nabla w_\eps
        \Big) \di x \di t \\
      &\quad\leq
        \int_I \int_{\Omega_h}
          \xi(\nabla w_\eps,  \partial_t \EEE \nabla w_\eps, \vartheta_\eps)
        \di x \di t
        + C \int_I \int_{\Omega_h}
          (\vartheta_\eps \wedge 1)
          \ZZZ \big(\xi(\nabla w_\eps,   \partial_t   \nabla w_\eps, \vartheta_\eps)\big)^{1/2} \EEE
        \di x \di t \\
      &\quad\leq  C 
        \int_I \int_{\Omega_h}
          \xi(\nabla w_\eps,  \partial_t \EEE \nabla w_\eps, \vartheta_\eps)
        \di x \di t
        + C \int_I \int_{\Omega_h}
          \vartheta_\eps
        \di x \di t
       \leq C h^{ 1 + \alpha\EEE}.
    \end{align*}
    By Hölder's inequality,   $\vartheta_{\flat,\eps}\le \vartheta_\flat^h\EEE$, \ZZZ  and    \eqref{improved_bounds_external}  we find  
$ \Vert \vartheta_{\flat,\eps}  \Vert_{L^1(I; L^1(\Gamma_h))} \le C h^{1+\alpha}  $, see also \eqref{hatten wir schon}.  \EEE  By    \eqref{improved_bounds_energy_initial} and H\"older's inequality with exponent\OOO s $4/(4 - \alpha)$ and  $4/\alpha$ we \OOO also \EEE get $\Vert m_\eps(0) \Vert_{L^1(\Omega_h)} \le Ch^{1+\alpha}$.    \OOO Using the above estimates\EEE, by \eqref{chi_eta_test}, \EEE  $0 \leq \chi_\eta(s) \leq s$, $ 0 \le \chi_\eta' \leq 1$,   and $\vartheta_\eps \geq 0$ we see that
    \begin{align}\label{mixed_term_hcm_bound}
      &\int_I \int_{\Omega_h}
        \chi_\eta''(m_\eps) \nabla  m_\eps \EEE
        \cdot  \hcmnoh \EEE(\nabla w_\eps,\vartheta_\eps) \nabla \vartheta_\eps
      \di x \di t \notag \\
      &\quad\leq \int_{\Omega_h} \chi_\eta(m_\eps(0)) \di x
        + \int_I \int_{ \Gamma_h\EEE} \kappa  \vartheta_{\flat,\eps} \EEE \dH \dt \notag \\
      &\phantom{\quad\leq}\quad
        + \int_I \int_{\Omega_h}
          \chi_\eta'(m_\eps) \Big(
            \xi_{\eps, \alpha}^{\rm reg}(\nabla w_\eps,  \partial_t \EEE \nabla w_\eps, \vartheta_\eps)
            + \partial_F W^{\rm cpl}(\nabla w_\eps, \vartheta_\eps) :  \partial_t \EEE \nabla w_\eps
          \Big) \dx \dt
      \leq Ch^{ 1 + \alpha}.
    \end{align}
    Using the definition of $\mathcal K_h$ in \eqref{hcm}--\eqref{spectrum_bound_K}, and \eqref{Linfty_strain_inverse_bound}--\EEE\eqref{bound:determinant}, we see that
    \begin{equation}\label{spec_bound_hcm}
      \frac{1}{C} \leq  \hcmnoh \EEE(\nabla w_\eps, \vartheta_\eps) \leq C,
    \end{equation}
    where the inequalities are meant in the eigenvalue sense.
 Employing the chain rule we have
    \begin{equation}\label{relationwtheta}
      \nabla m_\eps
      = \partial_\vartheta W^{\rm in}(\nabla w_\eps, \vartheta_\eps) \nabla \vartheta_\eps + \partial_F W^{\rm in}(\nabla w_\eps, \vartheta_\eps) : \nabla^2 w_\eps.
    \end{equation}
    Thus, \eqref{sec_deriv}, \eqref{spec_bound_hcm}, Young's inequality with constant $\lambda$, \eqref{est:coupl},  \eqref{Linfty_strain_inverse_bound}, \EEE and \eqref{inten_lipschitz_bounds} yield
    \begin{align*}
        \tfrac{1}{C_0C} \vert \nabla m_\eps \vert^2  & \leq \frac{1}{\partial_\vartheta W^{\rm in}(\nabla w_\eps, \vartheta_\eps)}
        \nabla m_\eps
        \cdot  \hcmnoh \EEE(\nabla w_\eps, \vartheta_\eps) \nabla m_\eps \\
      &\quad= \nabla m_\eps
          \cdot  \hcmnoh \EEE(\nabla w_\eps, \vartheta_\eps) \nabla \vartheta_\eps
        + \frac{1}{\partial_\vartheta W^{\rm in}(\nabla w_\eps, \vartheta_\eps)}
          \nabla m_\eps \cdot  \hcmnoh \EEE(\nabla w_\eps, \vartheta_\eps)
          (
            \partial_F W^{\rm in}(\nabla w_\eps, \vartheta_\eps)
              : \nabla^2 w_\eps
          ) \\
      &\quad\leq \nabla m_\eps \cdot  \hcmnoh \EEE(\nabla w_\eps, \vartheta_\eps) \nabla \vartheta_\eps
        + c_0^{-1} \lambda \vert \nabla m_\eps \vert^2
        + C_\lambda (m_\eps \wedge 1)^2 \vert \nabla^2 w_\eps \vert^2.
    \end{align*}
     We choose \EEE $\lambda = C^{-1}C_0^{-1} c_0/2$ in the estimate above. Then, by \EEE the elementary  estimate \EEE
    \begin{equation*}
      \chi_\eta''(m_\eps)
      = \frac{\eta}{h^\alpha (1+ h^{-\alpha} m_\eps )^{1+\eta}}
      \leq \frac{\eta}{h^\alpha + m_\eps}
      \leq m_\eps^{-1},
    \end{equation*}
    Young's inequality with powers $p/(p-2)$ and $p/2$, and $s \wedge 1 \leq s^{(p-1)/p}$ for $s \geq 0$,  it follows that \EEE
    \begin{align*}
      \chi_\eta''(m_\eps) \vert \nabla m_\eps \vert^2
      &\leq C \Big(
        \chi_\eta''(m_\eps) \nabla m_\eps \cdot  \hcmnoh \EEE (\nabla w_\eps, \vartheta_\eps) \nabla \vartheta_\eps
        + m_\eps^{-1} m_\eps^{2(p-1)/p} \vert \nabla^2 w_\eps \vert^2
       \Big) \\
    &\leq C \Big(
      \chi_\eta''(m_\eps) \nabla m_\eps
        \cdot  \hcmnoh \EEE(\nabla w_\eps, \vartheta_\eps) \nabla \vartheta_\eps
        + m_\eps
        + \vert \nabla^2 w_\eps \vert^p
      \Big).
    \end{align*}
   Integrating the above inequality over $I \times \Omega_h$, we derive by \eqref{mixed_term_hcm_bound}, \eqref{inten_lipschitz_bounds}, \eqref{bound:temperaturel1}, \ref{H_bounds}, \eqref{bound:total}, and $\alpha \leq 4$  the following \EEE weighted \GGG $L^2$-bound \EEE on the gradient:
    \begin{equation}\label{weigthedpoincarebound}
      \int_I \int_{\Omega_h}
        \frac{\vert \nabla m_\eps \vert^2}{(1 + h^{-\alpha} m_\eps)^{1+\eta}} \dx \dt
        = \frac{h^\alpha}{\OOO \eta\EEE} \int_I \int_{\Omega_h} \chi_\eta''(m_\eps) |\nabla m_\eps|^2 \di x \di t
        \leq \frac{C}{\eta} h^\alpha(h^{ 1 + \alpha} + h^5) \leq \frac{C}{\eta} h^{ 1 + 2 \alpha}.
    \end{equation}

    \emph{Step 2 \GGG ($L^r$-bound \EEE on $\nabla m_\eps$):}
    By interpolation we \ZZZ  now \EEE derive \EEE an  \GGG $L^r$-bound \EEE on $\nabla m_\eps$ for $r \in (1, 5/4)$.
     Let us shortly write \EEE $p_\eps \defas h^{-\alpha} m_\eps$.
     Furthermore, we choose $\eta \defas (5-4r)/3$ in \eqref{weigthedpoincarebound}.
    Then, b\EEE y Hölder's inequality with powers \OOO $2/r$ and $2/(2-r)$ \EEE we  get \EEE  
        \begin{align}\label{regularityestimate1}
      \Vert \nabla p_\eps \Vert_{L^r(I\times \Omega_h)}^r
      &= \int_I \int_{\Omega_h} \vert\nabla p_\eps  \vert^r (1+p_\eps)^{-(1+\eta)r/2}  (1+p_\eps)^{2 (2-r) r/3} \di x \EEE \notag \\
      &\leq \left(
        \int_I \int_{\Omega_h} (1 + p_\eps)^{ 4r/3} \dx \dt
      \right)^{(2-r)/2}
      \left(   \int_I \int_{\Omega_h}
        \frac{ h^{-2\alpha} \vert    \nabla m_\eps \EEE  \vert^2}{(1+  h^{-\alpha} m_\eps \EEE)^{1+\eta}} \dx \dt
      \right)^{r/2} \notag \\
      & \REV \leq C_{r} h^{-\alpha r} h^{(1+2\alpha)r/2} \Vert 1 + p_\eps \Vert_{L^{ 4r/3}(I \times \Omega_h)}^{ 2(2-r)r/3} \notag  \EEE \\
      &\leq C_{r} h^{r/2} \Vert 1 + p_\eps \Vert_{L^{ 4r/3}(I \times \Omega_h)}^{ 2(2-r)r/3}.
    \end{align}
    \OOO Notice that  \eqref{bound:temperaturel1} and \eqref{inten_lipschitz_bounds} yield \EEE
    \begin{equation}\label{Linfty_L1_bound_peps}
      \Vert 1 + p_\eps \Vert_{L^\infty(I;L^1(\Omega_h))} \leq Ch.
    \end{equation}
    Consequently, applying Lemma \ref{lem:gagliardonirenberg} for $\vphi = 1 + p_\eps$ we discover that 
    \begin{align}\label{regularityestimate2}
      &\Vert 1 + p_\eps \Vert_{L^{ 4r/3}(I\times \Omega_h)} \notag \\
      &\quad\leq C_r h^{ 3/(4r) - 1/4}
        \Vert 1 + p_\eps \Vert_{L^\infty(I;L^1(\Omega_h))}^{ 1/4}
        \left(
          h^{-1} \Vert 1 + p_\eps \Vert_{L^\infty(I;L^1(\Omega_h))}
          + h^{-1/r} \Vert \nabla p_\eps \Vert_{L^r(I\times \Omega_h)}
        \right)^{ 3/4} \notag \\
      &\quad\leq C_r h^{ 3/(4r)} \left(
        1 +  h^{-3/(4r)} \Vert \nabla p_\eps \Vert_{L^{r}(I;L^{r}(\Omega_h))}^{ 3/4\EEE}
      \right).
    \end{align}
    \OOO Employing \EEE \eqref{regularityestimate1}, \eqref{regularityestimate2},
    \OOO and \EEE Young's inequality with powers \OOO $2/r$ and $2/(2-r)$ \EEE  and constant $\nu > 0$, \ZZZ we deduce \EEE
    \begin{align*}
      \Vert \nabla p_\eps \Vert_{L^r(I\times \Omega_h)}^r
      &\leq C_r h^{r/2} h^{(2-r)/2} \left( 1 + h^{-3/(4r)} \Vert \nabla p_\eps \Vert_{L^r(I\times \Omega_h)}^{ 3/4} \right)^{ 2(2-r)r/3\EEE}\\
      &\leq C_rh \left( 1 + h^{r/2 - 1} \Vert \nabla p_\eps \Vert_{L^r(I\times \Omega_h)}^{ (2-r)r/2\EEE} \right) \\
      &\leq C_r h + \ZZZ C_r \EEE h^{r/2} \Vert \nabla p_\eps \Vert_{L^r(I\times \Omega_h)}^{(2-r)r/2}
      \leq C_r h + C_{r, \nu} h + C_r \nu \Vert \nabla p_\eps \Vert_{L^r(I\times \Omega_h)}^r .
    \end{align*}
    \REV Thus, we choose $\nu>0$   such that $C_r \nu \leq 1/2$, rearrange the above inequalty, and  \EEE see that
    \begin{equation}\label{grad_inten_bound}
      \Vert \nabla p_\eps \Vert_{L^r(I\times \Omega_h)}
      \leq C_r h^{1/r}.
    \end{equation}
     Recalling $p_\eps = h^{-\alpha} m_\eps$, this shows the estimate for $\nabla m_\eps$ in \eqref{reg:temperaturegradient} \OOO for $r \in (1, 5/4)$\EEE,  whereas \EEE
 the case $r=1$ follows by Hölder's inequality.

    \emph{Step 3 (Proof of the remaining bounds):}
    Given $q \in ( 1\EEE, 5/3)$, we  now derive  $L^q$-bounds \EEE on the temperature  and \EEE the internal energy.
     Suppose first that $q \in (4/3 ,5/3)$ and  let $r = 3q/4 \in (1, 5/4)$. \EEE
    Then, by Lemma \ref{lem:gagliardonirenberg} applied \OOO for \EEE $\vphi = p_\eps$, \eqref{Linfty_L1_bound_peps}, and \eqref{grad_inten_bound} it follows that
    \begin{align*}
      \Vert p_\eps \Vert_{L^q(I \times \Omega_h)}   &\leq C_q h^{3/(4r) - 1/4} \Vert p_\eps \Vert_{L^{\infty}(I;L^{1}(\Omega_h))}^{ 1/4\EEE} \left(h^{-1}\Vert p_\eps \Vert_{L^{\infty}(I;L^{1}(\Omega_h))} + h^{-1/r} \Vert \nabla p_\eps \Vert_{L^{r}(I;L^{r}(\Omega_h))} \right)^{ 3/4} \\ &
      \leq C_q h^{1/q}.
    \end{align*}
    With \eqref{inten_lipschitz_bounds}  and $p_\eps = h^{-\alpha} m_\eps$, \EEE this establishes \eqref{reg:temperature},  also using H\"older's inequality if $q \le 4/3$. \EEE
    Due to \eqref{grad_inten_bound}, in order to show \eqref{reg:temperaturegradient}, it remains to control $\Vert \nabla \vartheta_\eps \Vert_{L^r(I \times \Omega_h)}$ for $r \in [\EEE 1, 5/4)$.
    By \eqref{relationwtheta},  \eqref{grad_inten_bound}, \EEE Hölder's inequality with powers $p/(p-r)$ and $p/r$, \eqref{sec_deriv}, \eqref{bound:total},   \eqref{Linfty_strain_inverse_bound}, \EEE \eqref{est:coupl}, and $s \wedge 1 \leq s^{(p-1)/p}$ for $s \geq 0$  it holds that
    \begin{align*}
      &\int_I \int_{\Omega_h} h^{-\alpha r}  \vert \nabla \vartheta_\eps \vert^r \dx \dt \nonumber \\
      &\quad\leq C \int_I \int_{\Omega_h} h^{-\alpha r} \vert \nabla  m_\eps \vert^r \dx \dt
        + \int_I \int_{\Omega_h} \vert
          h^{-\alpha} \partial_F W^{\rm in}(\nabla w_\eps, \vartheta_\eps)
        \vert^r \vert \nabla^2 w_\eps \vert^r \dx \dt \nonumber\\
      &\quad\leq C_r h
        +  C \EEE\left(
            \int_I \int_{\Omega_h} \vert
              h^{-\alpha}  (\vartheta_\eps \wedge 1) \EEE
            \vert^{pr/(p-r)} \dx \dt
          \right)^{(p-r)/p} \left(
            \int_I \int_{\Omega_h} \vert \nabla^2 w_\eps \vert^p \dx \dt
          \right)^{r/p} \nonumber \\
      &\quad\leq
        C_r h
        + C h^{5r/p} \left(
          \int_I \int_{\Omega_h} (h^{-\alpha} \vartheta_\eps^{(p-1)/p})^{pr/(p-r)} \dx \dt
        \right)^{(p-r)/p} \nonumber\\
      &\quad\leq C_r h + C h^{5r/p} h^{-\alpha r} \Big(
          \int_I \int_{\Omega_h} \vartheta\EEE_\eps^{r(p-1)/(p-r)} \dx \dt
        \Big)^{(p-r)/p}.
    \end{align*}
     As $p > 3$ and $r <5/4$, we have  $r(p-1)/(p-r) < 5(p-1)/(4(p-5/4)) \leq 5/3$. Hence, \OOO we can apply \EEE \eqref{reg:temperature}  for the power $q \defas r(p-1)(p-r)$ which leads to \EEE
    \begin{equation*}
      h^{5r/p} h^{-\alpha r} \left(
          \int_I \int_{\Omega_h} \vartheta\EEE_\eps^{r(p-1)/(p-r)} \dx \dt
        \right)^{(p-r)/p}
      \leq C_r h^{5r/p -\alpha r + \alpha r(1 - 1/p) + 1 - r/p\EEE}= C_r h^{1+ (4-\alpha)r/p} \leq C_r h,
    \end{equation*}
    where we have used $\alpha \leq 4$.  This concludes the proof of \eqref{reg:temperaturegradient}. \EEE

\ZZZ It remains to show \eqref{forAubin-Lion}. To this end, we \EEE test \eqref{weak_form_heat_unrescaled_reg} with  an element \OOO $\vphi \in L^\infty(I;H^3(\Omega_h)) \subset L^\infty(I;W^{1,\infty}(\Omega_h))$  of the dual  satisfying $\Vert \vphi \Vert_{L^\infty(I;H^3(\Omega_h))} \leq 1$, \REV yielding
\begin{align*}
  \int_I \REV \langle  \partial_t m_\eps , \vphi \rangle \EEE \dt
    &= - \int_I \int_{\Omega_h}  \hcmnoh \EEE(\nabla w_\eps,\vartheta_\eps) \nabla  \vartheta_\eps \cdot \nabla  \vphi\EEE \dx \dt
      + \kappa\EEE \int_I \int_{ \Gamma_h\EEE} ( \vartheta_{\flat,\eps}- \vartheta_\eps ) \vphi\EEE \dH \dt \notag \\
    &\phantom{=}\quad + \int_I \int_{\Omega_h} \Big(\xi^{\rm reg}_{\eps, \alpha}(\nabla w_\eps,  \partial_t \EEE \nabla w_\eps, \vartheta_\eps ) + \partial_F W^{\rm cpl} (\nabla w_\eps, \vartheta_\eps) :  \partial_t \EEE \nabla w_\eps\Big) \vphi\EEE \di x \di t. \notag \\
          \end{align*}
 $s \wedge 1 \le \sqrt{s}$ for $s \ge 0$, \REV \eqref{avoidKorn}, \eqref{Linfty_strain_inverse_bound}, \eqref{diss_rate}, and Young's inequality with power $2$ imply that
\begin{align*}
  \partial_F W^{\rm cpl} (\nabla w_\eps, \vartheta_\eps) :  \partial_t  \nabla w_\eps \leq C (\vartheta_\eps)^{1/2} R (\nabla w_\eps, \partial_t \nabla w_\eps,\vartheta_\eps )^{1/2} \leq C \vartheta_\eps + C R (\nabla w_\eps, \partial_t \nabla w_\eps,\vartheta_\eps )
\end{align*}
a.e.\ in $I \times \Omega$. \EEE
Thus, we get by \EEE \eqref{spec_bound_hcm}, $\xi^{\rm reg}_{\eps,\alpha} \leq \xi$,  \eqref{diss_rate}, $\vartheta_{\flat,\eps} \leq \vartheta_\flat^h\EEE$,   and a trace estimate 
    \begin{align*}
      \int_I \REV \langle  \partial_t m_\eps , \vphi \rangle \EEE \dt
            &\leq
          C \lVert \nabla \vartheta\EEE_\eps \rVert _{L^1(I \times \Omega_h)} \lVert \nabla \vphi \rVert _{L^\infty(I \times \Omega_h)}
          +  C \EEE \lVert \vartheta_\flat^h\EEE \rVert _{L^{ 1\EEE}(I; L^1(\Gamma_h\EEE))} \lVert  \vphi \rVert _{L^\infty(I \times \Gamma_h)} \notag \\
          &\phantom{\leq}\quad + C \lVert \vartheta_\eps \rVert _{L^1(I; W^{1,1}(\Omega_h))} \lVert  \vphi \rVert _{L^\infty(I \times \Gamma_h)} 
         + C \int_I \int_{\Omega_h} R(\nabla w_\eps, \partial_t\EEE \nabla w_\eps, \vartheta_\eps) \di x \di t \lVert  \vphi \rVert _{L^\infty(I \times \Omega_h)}.  \notag\\
            \end{align*}
Then, \REV  employing  \eqref{bound:dissipation},  \eqref{improved_bounds_external},    \eqref{reg:temperature}--\eqref{reg:temperaturegradient}, $\Vert \vphi \Vert_{L^\infty(I;H^3(\Omega_h))} \leq 1$, and a Sobolev embedding        \ZZZ we conclude 
    \begin{align}\label{eq: similar arg}
      \int_I \REV \langle  \partial_t m_\eps , \vphi \rangle \EEE \dt         &\leq C(
          h^{ 1 + \alpha}
          + h^5 ) \leq C h^{ 1 + \alpha}
      \end{align}
\EEE This shows \eqref{forAubin-Lion}  by the arbitrariness of $\vphi$\EEE.
    \end{proof}

    \begin{remark}[Improved temperature bounds for $\alpha < 4$]\label{rem:improved_temp_bounds}
    \OOO We remark that, in the case $\alpha < 4$, the estimates \eqref{reg:temperature} and \eqref{reg:temperaturegradient} hold for a larger  class of   values \EEE of $q$ and $r$  than stated in \EEE Lemma \ref{lem:improvedregularityregularized}.
    In fact, testing \eqref{weak_form_heat_unrescaled_reg} with $\vphi \defas \chi'(m_\eps)$ for  $\chi(s) =   \alpha / 4 (h^\alpha + s)^{4/\alpha}$ \EEE as in the proof of  Proposition~\ref{lem: first a priori estimates}, and \ZZZ using \eqref{balancetestwithchi}, \eqref{reg:est:1}--\eqref{reg:est:3}, \EEE \eqref{bound:total}, \eqref{bound:dissipation}, and $\mathcal{E}^{(\alpha)}_h(w_0^h,\vartheta_0^h) \leq C_0 h^5$  it follows that \EEE
    \begin{equation}\label{further_improvement_heat_cond_upper_bound}
      \int_I \int_{\Omega_h}  \hcmnoh \EEE(\nabla w_\eps,\vartheta_\eps) \nabla \vartheta_\eps \cdot \nabla \chi'(m_\eps) \dx \dt \leq C h^5.
    \end{equation}
    Moreover, by \eqref{hcm_lower_bound}, \eqref{needed_for_imporoved_weighted_l2}, and \eqref{bound:total} we derive that
    \begin{align*}
      &\int_I \int_{\Omega_h}  \hcmnoh \EEE(\nabla w_\eps,\vartheta_\eps) \nabla \vartheta_\eps \cdot \nabla \chi'(m_\eps) \dx \dt \\
      &\quad\geq c_\alpha \int_I \int_{\Omega_h}
        (h^\alpha + m_\eps)^{4/\alpha - 2} |\nabla \vartheta_\eps|^2 \di x \di t
        - C \int_I \int_{\Omega_h} m_\eps^{4/\alpha} \di x \di t
        - C \int_I \int_{\Omega_h} |\nabla^2\EEE w_\eps|^p \di x \di t \\
      &\quad\geq c_\alpha \int_I \int_{\Omega_h}
        (h^\alpha + m_\eps)^{4/\alpha - 2} |\nabla \vartheta_\eps|^2 \di x \di t
        -C h^5
    \end{align*}
    for a constant $c_\alpha$  only \EEE depending on $\alpha$.
    With  \eqref{inten_lipschitz_bounds} and  \eqref{further_improvement_heat_cond_upper_bound} \EEE this leads to an improved weighted  $L^2$-bound \EEE on the temperature gradient,  namely \EEE
    \begin{equation*}
      \int_I \int_{\Omega_h} \frac{\vert h^{-\alpha} \nabla  \vartheta_\eps \vert^2}{(1 + h^{-\alpha} \vartheta_\eps)^{2-4/\alpha}} \leq C h.
    \end{equation*}
     This is in fact an improvement for $\alpha < 4$ as $2 - 4/\alpha < 1 \leq 1 + \eta$, see also \eqref{weigthedpoincarebound}. \EEE
    By a similar argument as in the proof of Lemma \ref{lem:improvedregularityregularized} we can \OOO then \EEE show that \eqref{reg:temperature} and \eqref{reg:temperaturegradient} hold true for $q = \frac{20}{3\alpha}$ and $r = \frac{20}{3\alpha + 4}$.
    We omit the details as this \OOO has \EEE already been discussed in \cite[Remark 3.21]{RBMFMK}  (To compare to \cite{RBMFMK}, replace $\alpha$ by $\alpha/2$, $d$ by $3$, and $\eps$ by $h$.) \EEE
    \end{remark}

    \subsection{Proof of Proposition~\ref{prop:Existenceof3dsolutions}}\label{sec: 4.4}

    We are ready to  prove  \OOO the \EEE a priori estimates for the \OOO nonregularized system. \EEE
    \begin{proof}[Proof of Proposition~\ref{prop:Existenceof3dsolutions}]
      Employing $\xi^{\rm reg}_{\alpha, \eps} \leq \xi^{(\alpha)}$, see   \eqref{diss_rate_truncated} \ZZZ and \eqref{dissipation:trucated2}, \EEE one can show in the same manner as   in  \EEE \cite[Proprosition 6.4]{MielkeRoubicek20Thermoviscoelasticity} that solutions $(w_\eps,\vartheta_\eps)$ from Proposition~\ref{prop:existenceregularized} converge, up to selecting a subsequence, to a weak solutions $(w^h, \vartheta^h)$ of the nonregularized problem in the sense of Definition~\ref{def:weak_formulation},  where we have
\begin{align}\label{the convergences}
   w_\eps &\rightharpoonup w^h \qquad \quad \text{ weakly* in } L^\infty(I; W^{2,p}(\Omega_h;\R^3))  \ZZZ  \ \  \  \text{ and weakly in } \EEE  H^1(I;H^1(\Omega_h;\R^3)) \EEE , \notag\\   
     \nabla w_\eps  &\to \nabla w^h \qquad \,\text{  strongly in \EEE}  L^\infty(I \times \Omega_h;\R^{3 \times 3})  , \notag \\
      \vartheta_\eps &\to \vartheta^h \qquad \quad \,\text{  strongly   in \EEE}  L^q(I\times \Omega_h)  \text{ for all }  1 \le \ZZZ q \EEE < \tfrac{5}{3}. 
     \end{align}
      It remains to ensure that the a priori bounds  stated in  Proposition~\ref{lem: first a priori estimates}, \EEE Lemma~\ref{lem: first a priori estimates2}, and Lemma~\ref{lem:improvedregularityregularized} are preserved  after \EEE taking the limit $\eps \to 0$.  For convenience, we address here the bounds for $(w^h,\vartheta^h)$, but they clearly transfer to $(y^h,\theta^h)$ as stated in Proposition~\ref{prop:Existenceof3dsolutions} by a \OOO change  of variables. \EEE

       By  \ref{H_bounds}, \ref{H_regularity}, \EEE and \eqref{inten_lipschitz_bounds}    we discover by  a  standard lower semicontinuity argument that \EEE the bounds \eqref{bound:total}, \eqref{bound:dissipation}, \eqref{bound:strainrate}, and \eqref{reg:temperature}--\eqref{reg:temperaturegradient} pass over to the limiting solution $(w^h, \vartheta^h)$  as $\eps \to 0$ resulting in  \eqref{boundres:mech}--\eqref{boundres:strainrate} \EEE and  \eqref{boundres:temperature}--\eqref{boundres:temperaturegrad}, respectively\EEE.  We omit details and just mention that for \eqref{boundres:mech}  \ZZZ it \EEE is important that, due to \eqref{the convergences}, the convergence $w_\eps(t)  \to   w^h(t)$ in $W^{1,\infty}(\Omega_h;\R^3)$ and $\vartheta_\eps(t) \to \vartheta^h(t)$ in $L^1(\Omega_h)$ hold for a.e.\ $t \in I$. \EEE

   Let us finally show \eqref{forAubin-Lion3}.      Given \EEE $h \in (0, 1)$,  we can define for \EEE a.e.~$t \in I$ the distribution $\sigma^h(t)$ by
      \begin{align*}
         \langle \sigma^h(t),\varphi \rangle &\defas - \int_{\Omega_h}
         \hcmnoh \EEE(\nabla w^h, \vartheta^h) \nabla \vartheta^h \cdot \nabla \vphi
        + \big(
          \xi^{(\alpha)}(\nabla w^h,  \partial_t \EEE \nabla w^h, \vartheta^h)
          + \partial_F W^{\rm{cpl}}(\nabla w^h, \vartheta^h) : \partial_t\EEE\nabla w^h
        \big) \vphi \di x \notag \\
        &\phantom{\defas}\quad
          + \kappa  \int_{ \Gamma_h \EEE} (\vartheta_\flat^h\EEE - \vartheta^h) \vphi \di \haus^2, \qquad
        \text{for every } \varphi \in  H^3 \EEE(\Omega_h),
           \end{align*}
              where all functions appearing on the right-hand side are evaluated at $t$.
      Then, as $(w^h, \vartheta^h)$ is a weak solution in the sense of Definition \ref{def:weak_formulation}, we see that for every $\psi\EEE \in C^\infty_{\OOO c\EEE} (I)$ and $\vphi \in C^\infty( \overline{\Omega_h})\EEE$ it holds that
      \begin{equation*}
        \int_I \langle \sigma^h(t), \varphi \rangle \psi\EEE(t)\dt
        = - \int_I \int_{\Omega_h} m^h \partial_t \psi\EEE(t) \varphi \dx  \dt,
      \end{equation*}
\ZZZ where $m^h \defas W^{\rm{in}}(\nabla w^h, \vartheta^h)$. \EEE \OOO The \EEE arbitrariness of $\vphi$ \OOO implies that \EEE the weak time derivative of $ m^h$ coincides in the distributional sense with $\sigma^h$ for a.e.~$t \in I$. Thus, it is left to show that $\sigma^h \in L^1(I;  H^3\EEE (\Omega_h)^*)$.
     In this regard, \ZZZ as shown above, \EEE $(w^h, \vartheta^h)$ satisfies \EEE the bounds \eqref{boundres:mech}--\eqref{boundres:dissipation} and \eqref{boundres:temperature}--\eqref{boundres:temperaturegrad},  up to scaling.
     With the definition of $\sigma$, \EEE  \eqref{avoidKorn}, Young's inequality, a trace estimate,  \eqref{improved_bounds_external}, and  a bound on $ \hcmnoh \EEE$ derived as in  \eqref{spec_bound_hcm} \EEE we find $\Vert \sigma^h \Vert_{L^1(I; \OOO(\EEE H^3 \EEE (\Omega_h)\OOO)\EEE^*)} \leq C h^{\alpha+1}$. \ZZZ This \EEE  concludes the proof of \eqref{forAubin-Lion3}, \ZZZ again up to scaling. \EEE  We refer to \eqref{eq: similar arg} for a similar argument. \EEE
    \end{proof}

    \section{Passage to the  two-dimensional \EEE limit}\label{sec:proofofcompactness}

 In this section we \ZZZ prove \EEE our main results, namely \EEE   Proposition~\ref{prop:compactness} and Theorem~\ref{maintheorem}.
    As before,  we write \EEE $\Omega  = \EEE \Omega' \times (-1/2, 1/2)$, $\Gamma = \GGG \Gamma' \EEE \times (-1/2, 1/2)$, and $\Gamma_D = \Gamma_D' \times (-1/2, 1/2)$. 

    \subsection{Rigidity}\label{sec:rigidity } 
    We start  this \EEE section by  proving \EEE the following rigidity result.  
    \begin{lemma}[Rigidity]\label{lemma:rigidity}
 Let   $\OOO(\EEE(y^h,\theta^h)\OOO)_h\EEE$   be a sequence of weak solutions to \eqref{weak_form_mech_res} and \eqref{weak_form_heat_res} as given in Proposition \ref{prop:Existenceof3dsolutions}.      Then, for sufficiently small $h$ there exists a  map \EEE $R^h \in L^\infty(I; H^1(\Omega'; SO(3)))$  such that \EEE
    \begin{subequations}
    \begin{align}
      \Vert \nabla_h y^h - R^h \Vert_{L^\infty(I;L^2(\Omega))} &\le C h^2 ,  \label{rig:closetorot}\\
      \Vert \nabla_h y^h - \Id \Vert_{L^\infty(I;L^2(\Omega))}
      + \Vert \nabla' R^h \Vert_{L^\infty(I;L^2(\Omega'))} &\le C h, \qquad
      \Vert R^h -\Id  \Vert_{{L^\infty(I;L^q(\Omega'))}} \leq C_q h, \label{rig:closetoid}\\
        \Vert \nabla_h y^h -\Id  \Vert_{L^\infty(I\times\Omega)}
      + \Vert R^h -\Id  \Vert_{L^\infty(I \times \Omega')} & \le Ch^{4/p}, \label{rig:closetoidinfty}
    \end{align}
    \end{subequations}
    where  $q \in [1, \infty)$, \EEE $C_q$ is a constant only depending on $q$ and $\Omega$, and  where \EEE  we have   extended $R^h$ to $I \times \Omega$ via $R^h(t, x) \defas R^h(t, x')$.
     Finally, \EEE setting $s= 1+ (3-8/p)^{-1} \in [1,2)$,  it holds that \EEE 
    \begin{subequations}
    \begin{align}
      \Vert  \partial_t \EEE \nabla_h y^h \Vert_{L^2(I\times \Omega)} & \leq C h, \label{bound:strainratenew}\\
         \Vert \sym \big( (\nabla_h y^h)^T  \partial_t \EEE \nabla_h y^h \big) \Vert_{L^2(I\times \Omega)} & \leq C h^2, \label{bound:strainraterotated-another one}\\
      \Vert \sym \big( (R^h)^T  \partial_t \EEE \nabla_h y^h \big) \Vert_{L^s(I\times \Omega)} & \leq C h^2, \label{bound:strainraterotated}\\
      \Vert \sym \big(\partial_t\EEE \nabla_h y^h \big) \Vert_{L^s(I\times \Omega)} &\leq C h^2. \label{boundres:strainratesym}
    \end{align}
    \end{subequations}
    \end{lemma}

    \begin{proof}
    \textit{Step 1 (Proof of \eqref{rig:closetorot}--\eqref{rig:closetoidinfty}):} \ZZZ     By \eqref{boundres:mech}, for $h$ sufficiently small we have  \EEE
    \begin{equation*}
     \esssup_{t \in I}   \mathcal{M}(y^h(t))  \ZZZ \le Ch^{-4}. \EEE
    \end{equation*}
 With this bound,  \EEE
    the proof of \eqref{rig:closetorot}--\eqref{rig:closetoidinfty} for fixed $t \in I$ can be found in \cite[Lemma 4.2]{FK_dimred}  (\EEE see also  \cite{FJM_hierarchy,lecumberry} \EEE for further details) with constants $C$ and $C_q$ that can be chosen uniformly in  $t$.  (Note that the scaling  $h^\alpha$ in \cite{FK_dimred} is replaced by  $h^{4/p}$ in \eqref{rig:closetoidinfty}. This is due to the fact that in our model the prefactor of the second gradient term is $h^{-4}$, whereas in \cite{FK_dimred} it is $h^{-\alpha p}$, see \cite[Equation~(2.14)]{FK_dimred}.) \EEE
     Moreover, the map $t \mapsto R^h(t)$ is measurable \EEE as a careful inspection of the proof of \cite[Theorem 6]{FJM_hierarchy} shows that $R^h(t,x')$ may defined as the nearest-point projection onto $SO(3)$ of
    \begin{align*}
    \int_{-1/2}^{1/2} \int_{x'+ (-h,h)\EEE^2} \frac{1}{h^2} \psi\left(\frac{x'-z'}{h}\right)\nabla_h y^h(t,z',z_3) \di z' \di z_3,
    \end{align*}
    $\psi$  being \EEE a standard mollifier.

    \textit{Step 2 (Proof of \eqref{bound:strainratenew}--\eqref{boundres:strainratesym}):}  First, \eqref{bound:strainratenew} has already been shown in \eqref{boundres:strainrate}. \GGG Moreover,  \EEE \eqref{bound:strainraterotated-another one} follows  by \OOO combining  \ref{D_quadratic}--\ref{D_bounds} and \eqref{boundres:dissipation}. \EEE        Let us \OOO now \EEE show \eqref{bound:strainraterotated}--\EEE\eqref{boundres:strainratesym}.
    We first note that $s = 1+(3-8/p)^{-1} \in [1,2)$ as $p>4$.
    With $2s/(2-s) = 2 + 4 (s-1)/(2-s)$,  \eqref{rig:closetorot}\EEE, and  \eqref{rig:closetoidinfty} \EEE we  then \EEE derive that
    \begin{align}\label{strainrateforHölder}
      \int_I \int_\Omega \vert h^{-1} (\nabla_h y^h - R^h) \vert^{\frac{2s}{2-s}} \dx \dt
      &\leq
        \lVert h^{-1} (\nabla_h y^h - R^h) \rVert _{L^2(I \times \Omega)}^2
        \lVert h^{-1}(\nabla_h y^h - R^h) \rVert _{L^\infty(I \times \Omega)}^{4\frac{s-1}{2-s}} \notag \\
      &\leq C  h^{2 + 4(4/p-1)\frac{s-1}{2-s}}
      = C,
    \end{align}
    where we have used $1 + 2(4/p-1)\frac{s-1}{2-s} = 0$ \OOO by our definition of $s$\EEE.  Consequently, by the triangular inequality, Hölder's inequality with powers $2/(2 - s)$ and $2/s$\EEE,  the definition of $s$\EEE, \eqref{bound:strainratenew}\OOO--\eqref{bound:strainraterotated-another one}, \EEE and  \eqref{strainrateforHölder} we derive that \EEE
    \begin{align*}
      &\Vert \sym((R^h)^T  \partial_t \EEE  \nabla_h \EEE y^h) \Vert_{L^s(I \times \Omega)} \\
      &\quad\leq
        \Vert \sym((R^h - \nabla_h y^h)^T  \partial_t \EEE  \nabla_h \EEE y^h) \Vert_{L^s(I \times \Omega)}
         + \Vert \sym((\nabla_h y^h)^T  \partial_t \EEE  \nabla_h \EEE y^h) \Vert_{L^s(I \times \Omega)} \\
      &\quad\leq
        C \Vert R^h - \nabla_h y^h \rVert_{L^{\frac{2s}{2 - s}}(I \times \Omega)}
        \Vert  \partial_t \EEE  \nabla_h \EEE y^h \rVert_{L^2(I \times \Omega)}
        + C \Vert \sym((\nabla_h y^h)^T  \partial_t \EEE  \nabla_h \EEE y^h) \Vert_{L^2(I \times \Omega)} \leq C h^2,
    \end{align*}
    which is \eqref{bound:strainraterotated}.
    Again, by the triangular inequality, Hölder's inequality with powers $2/(2-s)$ and $2/s$,   \eqref{rig:closetoid},    \eqref{bound:strainratenew}, \ZZZ and \EEE \eqref{bound:strainraterotated} we have that
    \begin{align*}
      \Vert \sym (  \partial_t \EEE  \nabla_h \EEE  y^h) \Vert_{L^s(I \times \Omega)}
      &\leq \Vert \sym((\Id - (R^h)^T)  \partial_t \EEE  \nabla_h \EEE  y^h) \Vert_{L^s(I \times \Omega)}
        + \Vert \sym((R^h)^T  \partial_t \EEE  \nabla_h \EEE  y^h) \Vert_{L^s(I \times \Omega)} \\
      &\leq C \Vert \Id - R^h \Vert_{L^{\frac{2s}{2 - s}}(I \times \Omega)}
        \Vert  \partial_t \EEE \nabla_h \EEE  y^h \Vert_{L^2(I \times \Omega)} + Ch^2 \leq C h^2.
    \end{align*}
     This shows  \eqref{boundres:strainratesym} and concludes the proof. \EEE
    \end{proof}

    \subsection{Compactness}\label{sec:compactness}
    Recall the definitions of $u^h$, $v^h$, \EEE and $\mu^h$ in \eqref{def:displacements} and \eqref{def:scaledtemp}.
    We are ready to prove Proposition~\ref{prop:compactness}.

    \begin{proof}[Proof of Proposition~\ref{prop:compactness}]
    We prove the statement for the sequence of solutions $\OOO(\EEE(y^h, \theta^h)\OOO)\EEE_h$ from Proposition~\ref{prop:Existenceof3dsolutions}  satisfying \EEE \eqref{boundres:mech}--\eqref{forAubin-Lion3}   and   \eqref{rig:closetorot}--\eqref{boundres:strainratesym} for a map $R^h \in L^\infty(I;H^1(\Omega';SO(3)))$.   For convenience, in this proof we only show the regularity  
\begin{align}\label{eq: reguregu}
u \in L^\infty(I; H^1(\Omega';\R^2)) \quad \text{ and } \quad v \in L^\infty(I; H^2(\Omega'))
\end{align}
for the displacements, deferring the regularity of the time derivatives to Lemma \ref{lem:conv_strain_stress} below. \EEE   
    
      The proof  consists of \EEE five steps.
     In the first  step\EEE, we investigate the convergence of $(\EEE h^{-1} (R^h - \Id))_h$,  leading to   the  compactness statement \EEE for $(u^h)_h$ and $(v^h)_h$ in Step 2. In Step 3, we relate the limit of $(h^{-1} (R^h - \Id))_h$ with the limit of $( v^h)_h$.         This allows us \EEE to verify the boundary conditions of the limit of $(v^h)_h$ in Step 4.
     Finally, we address the convergence of the temperatures $(\mu^h)_h$ in Step 5. \EEE

    \emph{Step 1 (Limit of $(\frac{R^h - \Id}{h})$):}
    As a preliminary step, we investigate the convergence of $A^h \defas h^{-1}(R^h - \Id)$. 
\ZZZ The following argument is similar to the one in \cite[Proof of Theorem 2.1, Step 2]{AMM_vKaccel}. Yet, we have a slightly different control on the time derivative. \EEE     By \eqref{rig:closetoid} there exists $A \in L^\infty(I; H^1(\Omega';  \R^{3 \times 3} \EEE ))$ such that, up to selecting a subsequence,     \begin{equation}\label{convAh}
      A^h \weaklystar A \qquad \text{weakly* to } L^\infty(I; H^1(\Omega';  \R^{3 \times 3} \EEE )).
    \end{equation}
    In the following, we will improve the above weak* convergence to strong convergence,  i.e., \EEE
    \begin{equation}\label{convAhimproved}
      A^h \to A \qquad \text{strongly in } L^{ q \EEE}(I \times \Omega';  \R^{3 \times 3} \EEE  ) \text{ for any }  q \EEE \in [1, \infty).
    \end{equation}
 This is based on showing that   for any $t_1, \, t_2 \in I$ with $0 < t_1 < t_2 < T$ it holds  that    
    \begin{equation}\label{lennart show}
    \limsup\limits_{ s\to 0} \sup\limits_{j} \int_{t_1}^{t_2} \Vert A^{h_j}(t+s) - A^{h_j}(t) \Vert_{\REV ( H^{1}(\Omega'))^* \EEE} \di t = 0,
    \end{equation}
     where  $(h_j)_j$ is \OOO an arbitrary  sequence converging to $0$\EEE.
 Then, \EEE as $H^1(\Omega'; \R^{3 \times 3})$ embeds  compactly into $L^{ q}(\Omega'; \R^{3 \times 3})$ for any $ q \EEE \in [1, \infty)$ and  $(A^h)_h$ is bounded in  $L^\infty(I; H^1(\Omega';  \R^{3 \times 3}  ))$  the desired convergence \eqref{convAhimproved} follows by using \cite[Theorem 6]{Simon}, see also \cite[Theorem 2.5]{AMM_vKaccel}. \EEE

 Let us show \eqref{lennart show}. \EEE      Using \EEE again \eqref{rig:closetoid},  we see that \EEE the sequence $(h^{-1}(\nabla_h y^h - \Id))_h$ is bounded in $L^\infty(I;L^2(\Omega; \R^{3 \times 3}))$.
    Moreover, $(h^{-1} \partial_t \EEE\nabla_h y^h)_h$ is bounded in $L^2(I; \REV (H^1(\Omega;\R^{3 \times 3}))^* \EEE)$ due to \eqref{bound:strainratenew}.
    \GGG Due to \EEE the compact embedding of \EEE $L^2(\Omega;  \R^{3 \times 3}\EEE)$ into $\REV (H^1(\Omega;\R^{3 \times 3}))^* \EEE$, \QQQ the \EEE Aubin-Lions lemma \GGG implies that \EEE $(h^{-1}(\nabla_h y^h - \Id))_h$ is  pre\EEE compact in $L^\infty(I;\REV ( H^{1}(\Omega))^* \EEE)$.  
    Hence, with  \cite[Theorem 2]{Simon} it follows for any $t_1, \, t_2 \in I$ \OOO satisfying \EEE $0 < t_1 < t_2 < T$ that \EEE
    \begin{equation}\label{strain_displacement_conv}
      \int_{t_1}^{t_2}
      \Vert h^{-1} (\nabla_h y^h(t+s) - \nabla_h y^h(t)) \Vert_{\REV ( H^{1}(\Omega))^* \EEE} \di t \to 0 \qquad \text{as } s \to 0 \text{, uniformly in } h.
    \end{equation}
     Fix \EEE $\eps>0$ and  consider a sequence $(h_j)_j$ converging to $0$\EEE.   Then\OOO,  we get \EEE
    \begin{equation*}
    \limsup\limits_{ s\to 0} \max_{h_j \geq \eps} \EEE \int_{t_1}^{t_2} \Vert A^{h_j}(t+s) - A^{h_j}(t) \Vert_{\REV ( H^{1}(\Omega'))^* \EEE} \di t = 0
    \end{equation*}
     since this convergence holds for any $A^{h_j}  \in \EEE L^\infty(I;H^1(\Omega'))$ and we are taking the maximum over a finite \OOO set\EEE. For every $h_j <\eps$ instead,  \EEE    \eqref{rig:closetorot} and the triangular inequality imply  
    \begin{align*}
      &\int_{t_1}^{t_2} \Vert A^{h_j} (t+s) - A^{h_j}(t) \Vert_{\REV ( H^{1}(\Omega'))^* \EEE} \di t \\
      &\quad=  \int_{t_1}^{t_2} \Vert {h}_j^{-1} (\nabla_{h_j} y^{h_j}(t+s) - \nabla_{h_j} y^{h_j}(t)) \Vert_{\REV ( H^{1}(\Omega))^* \EEE} \di t  \\
       &\phantom{\quad=}\quad + {h}_j^{-1} \int_{t_1}^{t_2}
       \Vert R^{h_j}(t+s) - \nabla_{h_j} y^{h_j}(t+s) \Vert_{L^2(\Omega)}
        + \Vert R^{h_j}(t)  - \nabla_{h_j} y^{h_j}(t) \Vert_{L^2(\Omega)} \di t \\
     &\quad\leq    \int_{t_1}^{t_2} \Vert {h}_j^{-1} (\nabla_{h_j} y^{h_j}(t+s) - \nabla_{h_j} y^{h_j}(t)) \Vert_{\REV ( H^{1}(\Omega))^* \EEE} \di t + C \eps.
    \end{align*}
    Thus, sending $\eps \to 0$ and using \eqref{strain_displacement_conv} results in  \eqref{lennart show}. \ZZZ Therefore, we have shown \GGG \eqref{convAhimproved}. \EEE

       Next, \EEE notice that
    \begin{equation}\label{toderivecompu}
      -\frac{(A^h)^T A^h}{2}
      = -\frac{((R^h)^T - \Id)(R^h - \Id)}{2h^2}
      = -\frac{\Id - (R^h)^T - R^h + \Id}{2h^2}
      = \sym \left( \frac{R^h - \Id}{h^2} \right).
    \end{equation}
    As $(A^h)_h$ is bounded in $ L^\infty(I; \EEE L^{ 2 q}(\Omega'; \R^{3 \times 3}))\EEE$  by \eqref{rig:closetoid} \EEE we see that
    \begin{equation*}
      \Vert \sym(A^h) \Vert_{ L^\infty(I; \EEE L^{ q}( \Omega'))\EEE}
       = h \left\Vert \frac{(A^h)^T A^h}{2} \right\Vert_{L^\infty(I; L^q(\Omega'))} \EEE
      \leq C h \Vert A^h \Vert_{L^\infty(I; L^{2q}(\Omega'))}^2
      \EEE
      \leq C h \to 0. 
    \end{equation*}
    This shows that $A$ is skew-symmetric  for a.e.~$(t, x') \in I \times \Omega'$, \EEE and therefore, \ZZZ due to \GGG \eqref{convAhimproved}, \EEE for any $ q\EEE \in [1, \infty)$ and $i, \, j \in \{1,2,3\}$ we have that
    \begin{equation*}
      \left( \sym \left( \frac{R^h - \Id}{h^2} \right) \right)_{ij}
      = -\frac{A^h e_i \cdot A^h e_j}{2}
      \to - \frac{A e_i \cdot A e_j}{2}
      = \frac{(A^2)_{ij}}{2} \qquad \GGG \text{strongly in } \EEE L^{ q}(I \times \Omega').
    \end{equation*}

    \emph{Step 2 (Compactness for $(u^h)_h$ and $(v^h)_h$):}
    Let $s$ be as in \eqref{convuhdot}.
    Since $u^h(t,x'\EEE) = 0$ for a.e.\ $(t,x'\EEE) \in I \times \Gamma_D'$ and thus $ \partial_t \EEE u^h(t,x'\EEE) = 0$ for a.e.~$(t,x'\EEE) \in I \times \Gamma_D'$, we obtain by  \eqref{def:displacements}, \EEE Korn's inequality, Jensen's inequality, and \eqref{boundres:strainratesym}\OOO: \EEE
    \begin{align*}
      \Vert  \partial_t \EEE \nabla'\EEE u^h \Vert_{L^s(I\times \Omega')}
      &\leq C \Vert \sym ( \partial_t \EEE \nabla'\EEE u^h) \Vert_{L^s(I\times \Omega')}
      = C h^{-2} \Bigg\Vert \int_{-1/2}^{1/2} \sym(\partial_t  \nabla'  y^h) \EEE \di x_3 \Bigg\Vert_{L^s(I\times \Omega')} \\
      &\leq C h^{-2} \Vert \sym( \partial_t  \nabla_h \EEE y^h) \Vert_{L^s(I\times\Omega)} \leq C.
    \end{align*}
 Thus,  Poincaré's inequality  \GGG yields \EEE $\partial_t\EEE u^h \rightharpoonup \tilde u$ weakly in $L^s(I; W^{1,s}(\Omega'; \R^2))$ for some $\tilde u \in L^s(I; W^{1,s}(\Omega'; \R^2))$, \EEE up to selecting a subsequence.
    We proceed similarly with $(u^h)_h$.
    By the boundedness of $(A^h)_h$ in $L^\infty(I; H^1(\Omega'; \R^{3 \times 3}\EEE))$, \eqref{rig:closetorot}, and \eqref{toderivecompu} we derive that
    \begin{equation*}
      \left\Vert \sym \left( \frac{\nabla_h y^h - \Id}{h^2} \right) \right\Vert_{L^\infty(I; L^2(\Omega))}
      \leq \left\Vert \sym \left( \frac{\nabla_h y^h - R^h}{h^2} \right) \right\Vert_{L^\infty(I; L^2(\Omega))}
      + C \Vert A^h \Vert_{L^\infty(I; L^4(\Omega'))}^2 \leq C.
    \end{equation*}
    Hence, Korn-Poincaré's inequality  and  $u^h(t,x'\EEE) = 0$ for a.e.\ $(t,x'\EEE) \in I \times \Gamma_D'$  \EEE yields that $(u^h)_h$ is bounded in $L^\infty(I; H^1(\Omega';\R^2))$.
    Possibl\OOO y p\EEE assing to a subsequence,  we can find \EEE $u \in L^\infty(I; H^1(\Omega';\R^2))$ such that $u^h \weaklystar u$ weakly* in $L^\infty(I; H^1(\Omega';\R^2))$.
    It is then standard to prove \ZZZ that \EEE $\tilde u =  \partial_t \EEE u$.
    In particular, we have shown \eqref{convuh}--\EEE\eqref{convuhdot}.

We now address compactness for  $(v^h)_h$.
    By \OOO the definition of $v^h$ in \EEE  \eqref{def:displacements}, \EEE \eqref{rig:closetoid}, \OOO Jensen's inequality, \EEE  $v^h(t,x'\EEE) = 0$ for  a.e.~$(t,x'\EEE) \in I \times \Gamma_D'$, \EEE and Poincaré's \EEE inequality it holds that $(v^h)_h$ is bounded in $L^\infty(I; H^1(\Omega'))$.
    Hence, there exists $v \in L^\infty(I; H^1(\Omega'))$ such that, up to selecting a subsequence, $v^h \weaklystar v$ weakly* in $L^\infty(I; H^1(\Omega'))$.
    Similarly,  by  \eqref{bound:strainratenew} we have  that $(\partial_t  \nabla'\EEE v^h)_h$ is bounded in $L^2(I \times \Omega';  \R^2)\EEE$.
    Again, Poincaré's inequality yields $ \partial_t \EEE v^h \rightharpoonup  \partial_t \EEE v$ weakly in $L^2(I; H^1(\Omega'))$, up to   a subsequence.     This concludes the proof of \eqref{convvh}--\EEE\eqref{convvhdot}.

 The convergences \eqref{convuh}--\eqref{convvhdot} and \eqref{eq: nonlinear boundary conditions} also imply that the first two boundary conditions in \eqref{eq: bccond} are satisfied. Our next goal is to complete the proof of \eqref{eq: bccond} and to show  $v \in L^\infty(I; H^2(\Omega'))$. For this, we first need some additional properties of $A$. \EEE

    \textit{Step 3 (Characterization of $A$):}  
     Notice that by the definition of $u^h$ and $A^h$ we have \EEE for a.e.~$(t,x') \in I \times \Omega'$
    \begin{align*}
      h\partial_2 u_1^h
      = \frac{1}{h} \int_{-1/2}^{1/2} \partial_2 y^h_1 \di x_3
      = \frac{1}{h} \int_{-1/2}^{1/2}  (\partial_2 y^h_1 - R^h_{12} ) \di x_3
         + A^h_{12}.
    \end{align*}
    Therefore, by \eqref{rig:closetorot} and \eqref{convuh} we derive that $(\ZZZ h^{-1} \EEE A_{12}^h)_h$ is bounded in $L^\infty(I; L^2(\Omega'))$.
    This shows
    \begin{equation}\label{A12}
      A_{12} = 0 \qquad \text{for a.e.~} (t,x') \in I \times \Omega'.
    \end{equation}
     Similarly, \EEE for $i \in \{1, 2\}$ and a.e.~$(t,x') \in I \times \Omega'$  we have \EEE
    \begin{equation*}
      \partial_i v^h
      = \int_{-1/2}^{1/2} \frac{\partial_i y_3^h - R^h_{3i}}{h} \di x_3 + A^h_{3i}.
    \end{equation*}
     Using \eqref{rig:closetorot} and t\EEE aking the limit $h \to 0$ on both sides leads  to \EEE
    \begin{equation}\label{A3i}
      \partial_i v = A_{3i} \qquad \text{for a.e.~} (t,x') \in I \times \Omega' \text{ and } i \in \{1,2\}.
    \end{equation}
      The skew-symmetry of $A$, \eqref{A12}, and \eqref{A3i} then  allow \OOO us  to represent $A$ in terms of $v$, namely \EEE
    \begin{equation}\label{A_repr}
      A
      = e_3 \otimes \begin{pmatrix} \nabla' v \\ 0 \end{pmatrix}
        - \begin{pmatrix} \nabla' v \\ 0 \end{pmatrix} \otimes e_3
      = \begin{pmatrix}
        0 & 0 & -\partial_1 v \\
        0 & 0 & -\partial_2 v \\
        \partial_1 v & \partial_2 v & 0
      \end{pmatrix}.
    \end{equation}
   As \EEE $A \in L^\infty(I; H^1(\Omega';  \R^{3 \times 3}\EEE))$  (see \EEE \eqref{convAh}) this \EEE implies $v \in L^\infty(I; H^2(\Omega'))$.

    \emph{Step 4 (Trace of  $\nabla' v$ on $\Gamma_D'$):}     Our next goal is to derive  the trace condition \EEE
    \begin{equation}\label{boundarynablav}
      \nabla' v = 0 \qquad \text{a.e.~on }   I \EEE \times \Gamma_D'.
    \end{equation}
   To this end,  let us define \EEE
    \begin{equation*}
    \mathcal Z^h(t,x') \defas
      \int_{-1/2}^{1/2} x_3 \left(
          y^h(t,x',x_3) -
          \begin{pmatrix} x' \\ hx_3 \\ \end{pmatrix}
        \right) \di x_3
    \end{equation*}
 Repeating the proof in  \cite[Corollary 1]{FJM_hierarchy}, in particular by following \cite[ Equation~(100)\EEE]{FJM_hierarchy}, we get that 
    \begin{align*}
      \frac{1}{h^2} \mathcal Z^h(t) \ZZZ  \rightharpoonup \EEE \frac{1}{12} A(t) e_3
        &= -\frac{1}{12} \begin{pmatrix} \nabla' v(t) \\ 0 \end{pmatrix} \quad
          \text{weakly in }  H^1(\Omega';\R^3)
          \end{align*}
          for a.e.\ $t \in I$, where  \EEE the  equality   is a direct consequence of \eqref{A_repr}.     Notice that by construction $\mathcal Z^h (t,x') = 0$  for a.e.~$(t,x') \in I \times \Gamma_D'$. Consequently, by the above convergence and the compactness of the trace operator from $H^1(\Omega';\R^3)$ to $ L^2(\OOO \Gamma'; \R^3)$,  \eqref{boundarynablav} follows. \EEE

    \emph{Step 5 (Compactness for the temperature and its gradient):}
    Using the definition of $\mu^h$ in \eqref{def:scaledtemp}  and \EEE \eqref{boundres:temperature}--\EEE\eqref{boundres:temperaturegrad}  we see that \EEE
    \begin{equation*}
      \sup_{h \in (0, 1]} \Big(
        \Vert \mu^h \Vert_{L^q(I \times  \Omega' \EEE)}
        + \Vert \OOO\nabla'\EEE \mu^h \Vert_{L^r(I\times  \Omega' \EEE)}
      \Big) < \infty
    \end{equation*}
    for any $q \in [1,5/3)$ and $r \in [1,5/4)$.
    This directly leads to the \GGG convergence \EEE in \eqref{convmuhnabla}, up to selecting a subsequence\EEE.
    The strong convergence in \eqref{convmuh} is more delicate.  Here, we follow the lines of \cite[Lemma~4.2]{RBMFMK}. \EEE First, we \EEE show that, up to selecting a subsequence, \EEE
    \begin{equation}\label{theta_h_conv}
       h^{-\alpha}  \theta^h \EEE \to \tilde \theta \qquad \text{strongly in } L^q(I \times \Omega),
    \end{equation}
     for any $q \in [1, 5/3)$  for some $\tilde \theta \in L^q(I \times \Omega)$.      Set $\eta^h \defas h^{-\alpha} \zeta^h = h^{-\alpha} W^{\rm in}(\nabla_h y^h , \EEE \theta^h)$.
    \EEE
    By \eqref{boundres:temperature}--\eqref{forAubin-Lion3} the sequence $(\eta\EEE^h)_h$ is uniformly bounded in $L^r(I; W^{1,r}(\Omega))$ for any $r \in [1, 5/4)$ and that $(\partial_t\eta\EEE^h)_h$ is uniformly bounded in $L^1(I;\OOO( H^3 \EEE (\Omega)\OOO)\EEE^*)$.
 \ZZZ    Fix $\tilde r \in  (1, \frac{15}{7})$. \EEE     By  the \EEE Rellich-Kondrachov theorem, there exists $ s\EEE \in [1,5/4)$ such that the embedding $W^{1,s}(\Omega) \subset\subset L^{\tilde r}(\Omega)$ is compact.
     As \EEE $L^{\tilde  r}(\Omega) \subset \OOO( H^3 \EEE (\Omega)\OOO)\EEE^*$, we derive by  the \GGG Aubin-Lions \EEE lemma that $\eta\EEE^h \to \eta$  strongly in $L^{s}(I; L^{\tilde r}(\Omega))$ 
    for some $\eta\EEE \in L^{s}(I; W^{1,s}(\Omega)\GGG)$.   \EEE \ZZZ In  particular, this implies  $\eta^h \to \eta$ in measure on $I \times \Omega$. 
     \EEE Given any $q \in [1, 5/3)$, the sequence $(\eta^h)_h$ is equiintegrable in \GGG $L^q(I \times \Omega)$ \EEE by \eqref{boundres:temperature} (applied for a larger exponent less than $5/3$), and then Vitali's convergence theorem implies   
  \begin{equation}\label{phconv1}
      \eta\EEE^h \to \eta\EEE\qquad {\rm strongly \ in \ }L^{q}( \GGG I \times \Omega \EEE).
    \end{equation}     \EEE
 To show \eqref{theta_h_conv}, we now transfer the \OOO convergence  of \EEE  the sequence $(\eta^h)_h$ to the sequence $(h^{-\alpha}\theta^h)_h$.
    We \EEE first note that for any $F \in GL^+(3)$, the map $W^{\rm in }(F, \cdot)$ is invertible with
    \begin{equation*}
      (\EEE W^{\rm in}(F, \cdot)^{-1})'\EEE(m) = c_V\big(F, W^{\rm in}(F, \cdot)^{-1}(m)\big)^{-1} \leq c_0^{-1}
    \end{equation*}
    for every $m > 0$, where we recall the definition of $c_V$ in \ref{C_heatcap_cont}  and  the  \EEE bound \OOO in \EEE \eqref{sec_deriv}.
     Using \EEE the definition of $\zeta^h$ we can write $\theta\EEE^h  = W^{\rm in}(\nabla_h y^h, \cdot)^{-1}(\zeta^h)$.
    Then,  by \EEE the fundamental theorem of calculus, a change of variables, \ZZZ and the fact that $W^{\rm in}(\nabla_h y^h, \cdot)^{-1}(0)= 0$ \GGG (see the discussion below \eqref{inten_lipschitz_bounds}) \EEE  it follows that \EEE
    \begin{align*}
       h^{-\alpha} \theta\EEE^h
      = h^{-\alpha} \int_0^{\zeta^h} (\EEE W^{\rm in}(\nabla_h y^h, \cdot)^{-1})'\EEE(m) \di m
      &= h^{-\alpha} \int_0^{\zeta^h} c_V\big(\nabla_h y^h, W^{\rm in}(\nabla_h y^h, \cdot)^{-1}(m)\big)^{-1} \di m \\
      &= \int_0^{ \eta\EEE^h} c_V(\nabla_h y^h, W^{\rm in}(\nabla_h y^h, \cdot)^{-1}(h^\alpha m))^{-1} \di m.
    \end{align*}
    Let us now set $\tilde\theta\EEE \defas \overline c_V^{-1} \eta\EEE$, where $\overline c_V = c_V(\Id, 0)$  is as in \EEE \eqref{barcvandk}.
    By \eqref{sec_deriv} we  then \EEE derive that
    \begin{align*}
      \vert  h^{-\alpha} \theta\EEE^h - \tilde \theta\EEE \vert
      &\ZZZ = \EEE \left\vert
          \int_0^{\eta\EEE^h} c_V\big(\nabla_h y^h, W^{\rm in}(\nabla_h y^h, \cdot)^{-1}(h^\alpha m)\big)^{-1} \di m
          - \int_0^{  \eta\EEE} \overline c_V^{-1} \di m
        \right\vert \\
    &\leq \frac{1}{c_0} |\eta\EEE^h - \eta\EEE|
      +
          \int_0^{  \eta\EEE^h} \OOO\left|\EEE c_V\big(\nabla_h y^h, W^{\rm in}(\nabla_h y^h, \cdot)^{-1}(h^\alpha m)\big)^{-1} - \overline c_V^{-1} \OOO\right|\EEE\di m.
    \end{align*}
  The integrand of the second term is bounded by $2/c_0$, see \eqref{sec_deriv}, and thus the integral is bounded pointwise by $2\eta^h/c_0$. \ZZZ Then,  \EEE  $\eta^h \to \eta$ in $L^q(I \times \Omega)$, the continuity of $c_V$ at $(\Id, 0)$, \eqref{rig:closetoidinfty}, \OOO and   dominated convergence imply \eqref{theta_h_conv}.
     Now, notice that by \eqref{convmuhnabla} we must have for a.e.~$(t, x') \in I \times \Omega'$\OOO  \EEE
    \begin{equation}\label{mu_def}
      \mu(t, x') = \int_{-1/2}^{1/2} \tilde \theta(t, x', x_3) \di x_3.
    \end{equation}
        Hence, with \eqref{theta_h_conv} and Jensen's inequality \eqref{convmuh} follows.
 Eventually we note that the weak convergence of the scaled gradient $h^{-\alpha} \nabla_h \theta^h$, see \eqref{boundres:temperaturegrad}, implies that $ \tilde \theta  \ZZZ =   \overline c_V^{-1}\eta \EEE$ does not depend on $x_3$, i.e., $\mu = \tilde \theta$. \EEE
    \end{proof}
    
  The  following \EEE corollary collects some properties that have been established in the previous proof. \EEE
     
\begin{corollary}\label{a new corollary}
In the setting of Proposition~\ref{prop:compactness}, given the maps $(R^h)_h \subset L^\infty(I;H^1(\Omega';SO(3)))$ from Lemma \ref{lemma:rigidity}  the functions $A^h \defas \frac{R^h - \Id}{h}$ satisfy,  up to  a subsequence, for any $q \in [1,\infty)$
\begin{subequations}
\begin{align}
  A^h &\to A  \qquad \quad \text{weakly* in }    L^\infty(I; H^1(\Omega'; \R^{3 \times 3})) \text{ and strongly in } L^{ q \EEE}(I \times \Omega'; \R^{3 \times 3}), \label{convAhlem}\\
     h^{-1} \sym \left( A^h \right) &\to
   \frac{1}{2} \ZZZ A^2 \EEE  \quad   \ \text{ strongly in } L^{q}(I \times \Omega'; \R^{3 \times 3}), \label{convAhsquaredstatement} 
\end{align}
\end{subequations}
where the limit $A \in L^\infty(I; H^1(\Omega';   \R^{3 \times 3}   )) $ is characterized by 
\begin{equation}\label{sym_G''A_repr-new}
 A
  = e_3 \otimes \begin{pmatrix} \nabla' v \\ 0 \end{pmatrix}
    - \begin{pmatrix} \nabla' v \\ 0 \end{pmatrix} \otimes e_3 \quad \text{a.e.~in $I \times \Omega$}.
\end{equation}
 Moreover, \EEE
    \begin{equation}\label{theta_h_convXXX}
       h^{-\alpha}  \theta^h \EEE \to   \mu \qquad \text{strongly in } L^q(I \times \Omega) \quad \text{for any $q \in [1, 5/3)$,}
    \end{equation}
 for $\mu$ as given in Proposition~\ref{prop:compactness}. \EEE
\end{corollary}    
    
    \EEE
    
    The next lemma derives compactness properties of the internal energy. \EEE

    \begin{lemma}[Compactness  of internal energy\EEE]\label{corol:compactnesscorollary}
     Let $(\EEE(y^h,\theta^h))\EEE_h$ be a sequence of weak \GGG solutions \EEE to \eqref{weak_form_mech_res} and \eqref{weak_form_heat_res} in the sense of Definition \ref{def:weak_formulation}, such that \eqref{boundres:mech}--\eqref{forAubin-Lion3} and all   assumptions of Proposition~\ref{prop:Existenceof3dsolutions}  are satisfied. \ZZZ Let  $( \mu^h_0 )_h$ be the  rescaled versions of the initial temperatures $(\theta_0^h)_h$     given in  \eqref{def:scaledtemp}. We  \EEE suppose that $\mu_0^h \to \mu_0$  strongly in $L^2( \Omega')$ and that \eqref{well-preparednessinitial} holds. \EEE
     Then, the following holds true: \EEE
    \begin{subequations}
    \begin{align}
    h^{-\alpha}  \int_{-1/2}^{1/2} \EEE W^{\rm in}(\nabla_h y^h, \theta^h)  \di x_3 \EEE &\to \overline c_V \mu &&\text{strongly in } L^q(I\times \Omega) \ \text{for }q \in [1, 5/3), \label{convscaledinternal} \\
    h^{-\alpha}  \int_{-1/2}^{1/2} \EEE W^{\rm in}(\nabla_h y^h_0, \theta^h_0)  \di x_3 \EEE &\to \overline c_V \mu_0 &&\text{strongly in } L^1( \Omega \EEE), \label{convscaledinternalinitial}
    \end{align}
    \end{subequations}
    as $h \to 0$, where  as before  $\overline c_V = c_V(\Id, 0)$.
    \end{lemma}

    \begin{proof}
     Notice that \EEE \eqref{convscaledinternal}  has \EEE been addressed in Step~5 of the proof of Proposition \ref{prop:compactness}, see   \ZZZ \eqref{phconv1} \EEE and \eqref{mu_def}, and use the identities $\eta^h = h^{-\alpha} W^{\rm in}(\nabla_h y^h  , \EEE \theta^h)$ and  $\overline c_V \tilde\theta =  \eta$. \EEE
    
  To see  \eqref{convscaledinternalinitial}, we recall the relation $  \mu^h_0(x')
  = \frac{1}{h^\alpha}   \int_{-1/2}^{1/2} \theta^h_0( x', x_3) \di x_3   $ by  \eqref{def:scaledtemp}. Then, \OOO using \EEE the fundamental theorem of calculus and a change of variables,  we get  for  a.e.~$(x',x_3) \in \Omega$
     \begin{align*}
   h^{-\alpha} W^{\rm in}(\nabla_h y^h_0, \theta^h_0) - \overline c_V   \mu^h_0  =  \int_0^{ h^{-\alpha} \theta^h_0\EEE}   \partial_\vartheta W^{\rm in}(\nabla_h y^h_0,h^\alpha s) \di s - \overline c_V \mu^h_0.
\end{align*}
Taking the integral over $x_3$ and using Fubini's theorem,    we get \GGG for \GGG a.e.~$x' \in \Omega'$ \EEE
\begin{align*}
     h^{-\alpha}  \int_{-1/2}^{1/2} W^{\rm in}(\nabla_h y^h_0, \theta^h_0)   \di x_3 -  \overline c_V \mu^h_0 & = \int_0^{\mu_0^h} \Big( \int_{-1/2}^{1/2}  \partial_\vartheta W^{\rm in}(\nabla_h y^h_0,h^\alpha s) \di x_3   - \overline c_V  \Big) \di s \\ &\quad  +   \int_{-1/2}^{1/2} \int_{\mu_0^h}^{h^{-\alpha} \theta^h_0}  \Big( \partial_\vartheta W^{\rm in}(\nabla_h y^h_0,h^\alpha s) - \overline c_V  \Big)   \di s  \di x_3, 
    \end{align*}
    where we also used $ \int_{-1/2}^{1/2} \int_{\mu_0^h}^{h^{-\alpha} \theta^h_0} 1 \di s  \di x_3 =0$ and the fact that $\mu_0^h$ is independent of $x_3$. We observe that the \OOO absolute values of the  integrands are bounded by \eqref{sec_deriv} and that $\partial_\vartheta W^{\rm in}(\nabla_h y^h_0,h^\alpha s) \to \overline c_V $ pointwise in $\Omega$ by  the continuity of $c_V$ at $(\Id, 0)$  and \eqref{rig:closetoidinfty}. Recall $\mu_0^h \to \mu_0$   in $L^2(\Omega')$ \OOO by assumption  and  \eqref{well-preparednessinitial} which gives that $\int_{-1/2}^{1/2}|h^{-\alpha} \theta^h_0 - \mu_0^h  | \EEE\di x_3 \to 0 $ in $L^2(\Omega')$. Then,  \eqref{convscaledinternalinitial} follows by dominated convergence. 
    \end{proof}

\subsection{Convergence of strain and stress}  

Given any matrix $M \in \R^{3 \times 3}$, we will write $M''$ for the upper-left $(2 \times 2)$-submatrix of $M$. In the next lemma, we address the convergence of the rescaled strain and stress tensors.

\begin{lemma}[Convergence of \OOO rescaled \EEE strain and stress]\label{lem:conv_strain_stress}\EEE
\ZZZ Suppose that all assumptions of Proposition \ref{prop:Existenceof3dsolutions} hold and that $(\EEE(y^h,\theta^h))_h$ is a sequence of solutions satisfying  \eqref{boundres:mech}--\eqref{forAubin-Lion3}. Moreover, we assume that  \eqref{rig:closetorot}--\eqref{boundres:strainratesym} \EEE hold for a sequence of rotations $(R^h)_h \subset \EEE L^\infty(I;H^1(\Omega';SO(3)))\EEE$. \EEE
We define
\begin{equation}\label{def:AhGh}
  G^h \defas \frac{(R^h)^T \nabla_h y^h - \Id}{h^2}.
\end{equation}
Then, \ZZZ there exists $G \in L^\infty(I; L^2(\Omega; \R^{3\times 3}))$ such that, \EEE up to selecting a \ZZZ subsequence (not relabeled), \EEE the following convergences hold true: \EEE
\begin{subequations}
\begin{align}
  G^h &\weaklystar G && \text{weakly* in } L^\infty(I; L^2(\Omega; \R^{3\times 3})), \label{convGh} \\
   \frac{1}{2h^2} \left((\nabla_h y^h)^T \nabla_h y^h - \Id\right) &\rightharpoonup \sym(G) && \text{weakly in }  H^1(I; L^2(\Omega; \R^{3\times 3})), \EEE \label{convstrainrateweak}
\end{align}  
\end{subequations}
where  
\begin{equation}\label{sym_G''A_repr}
\sym(G'')
  = \sym(\nabla' u)
    + \frac{1}{2} \nabla'v \otimes \nabla' v
    - x_3 (\nabla')^2 v
\end{equation}
 a.e.~in \ZZZ $I \times \Omega$. \EEE
\GGG Furthermore, \EEE  we have \EEE $ \partial_t \EEE u \in L^2(I; H^1(\Omega'; \R^2))$,  $\partial_t\EEE v \in L^2(I; H^2(\Omega'))$,  and \EEE
\begin{equation}\label{sym_dotG''}
  \partial_t \sym(G'')
  =  \sym( \partial_t \EEE \nabla' u) + \partial_t \EEE \nabla' v \odot \nabla' v
    - x_3 (\nabla')^2  \partial_t \EEE v
\end{equation} for a.e.~$(t, x) \in I \times \Omega$. \EEE
Finally,    for any  $ r \in [1, 5/3)$ \EEE  the following convergences hold true:
\begin{subequations}
\begin{align}
  h^{-2} \partial_F W^{\rm el}(\nabla_h y^h)
    &\weaklystar \C_{W^{\OOO \rm el\EEE}}^3 \sym(G) \qquad
    &&\text{weakly* in } L^\infty(I; L^2(\Omega; \R^{3\times 3})), \label{conv:stress:el} \\
  h^{-2} \partial_F W^{\rm cpl}(\nabla_h y^h,\theta^h)
    &\weakly \ZZZ \mu \EEE   \mathbb{B}^{(\alpha)}
    &&\text{weakly  in } L^{ r}(I \times \Omega; \R^{3\times 3}), \label{conv:stress:cpl}\\
  h^{-3} \partial_G H(\nabla_h^2 y^h)
    &\to 0
    &&\text{strongly in \EEE} L^\infty(I; L^1(\Omega; \R^{3 \times 3 \times 3})), \label{secondgradientvanishes} \\
  h^{-2}\partial_{\dot F} R (\nabla_h y^h,  \partial_t \EEE \nabla_h y^h, \theta^h)
  &\weakly \C_R^3 \partial_t  \sym(G)
  &&\text{weakly in } L^2(I\times \Omega; \R^{3 \times 3}), \label{conv:stress:visc}
\end{align}
\end{subequations}
where  $\C_{W^{\rm el}}^3$, $\C_R^3$, and $\mathbb B^{(\alpha)}\EEE$ are as in \eqref{eq:quadraticformsnotred}--\eqref{alpha_dep},  and  \GGG $\mu$ as given in Proposition~\ref{prop:compactness}. \EEE 
\end{lemma}

\EEE We remark  at this point \EEE that due to  the higher order regularization\OOO, \EEE no linear growth condition on $\partial_F W^{\rm el}(F)$  is required,  in contrast to, e.g.,  \cite[Equation~(1.19)]{MP_thin_plates} or \cite[Equation~(2.1)]{AMM_vKaccel}.

\begin{proof} 
We divide the  proof in four steps. In the first step,  we prove \eqref{convGh}--\eqref{convstrainrateweak} \ZZZ and \eqref{sym_G''A_repr}. \ZZZ Here, for \EEE the convergence of $(G^h)_h$ we argue along the lines of \cite[Proof of Theorem~1, Step~4]{AMM_vKaccel}. \EEE
Then, convergence  of \EEE the rescaled elastic and coupling stress is shown in Step 2.
Step 3  is concerned with \EEE the rescaled viscous stress. \ZZZ This \EEE eventually allows us to  show \ZZZ the characterization in  \eqref{sym_dotG''} and \EEE the regularity of $\partial_t \EEE u$ and $\partial_t\EEE v$ in Step 4. \ZZZ The latter \EEE in turn also concludes the proof of the compactness statement in Proposition~\ref{prop:compactness}, \ZZZ see \eqref{eq: reguregu}. \EEE

\textit{Step 1 (Compactness for $(G^h)_h$ and characterization of the limit):}  First, note that \eqref{rig:closetorot} directly  gives \EEE \eqref{convGh} by weak compactness.  Next, we \EEE  characterize $G''$ in terms of $u$ and $v$.
In this regard, we  show that  $G''(t, x', \cdot)$ is \OOO affine for \EEE a.e.~$(t,  x')  \in I \times \Omega'$. \EEE By the fundamental theorem of calculus and the definition of $G^h$ in \eqref{def:AhGh}, we discover for a.e.~$(t, x) \in I \times \Omega$, $\ell > 0$ sufficiently small, and $i \in \{1,2\}$  that \EEE
\begin{align*}
  R^h(t,x') \frac{G^h(t,x',x_3 + \ell) - G^h(t,x',x_3)}{\ell} e_i
  &= \frac{\partial_i y^h(t, x', x_3 + \ell) - \partial_i y^h(t, x', x_3)}{h^2 \ell} \\
  &= \partial_i \Bigg( \frac{1}{\ell} \int_0^\ell \frac{h^{-1} \partial_3 y^h (t,x',x_3 + \tilde \ell)}{h} \, {\rm d} \tilde \ell \Bigg).
\end{align*}
On the one hand, by \eqref{rig:closetoidinfty} and \eqref{convGh}, the left-hand side converges to $\ell^{-1}(G(t,x',x_3 + \ell) - G(t,x',x_3)) e_i$ weakly* in $L^\infty(I; L^2(\Omega; \R^3))$.
On the other hand, by adding $\partial_i \big( (R^h-\Id)e_3 - R^he_3 \big) = 0$ to the equation above, we see by \eqref{rig:closetorot} and \eqref{convAhlem} that  the term after the second equal sign \EEE  converges to $\partial_i A(t,x') e_3$ weakly* in $L^\infty(I; \REV (H^{1}(\Omega; \R^3))^* \EEE)$.
Thus, we have for a.e.~$(t,x) \in I \times \Omega$, $\ell > 0$ sufficiently small, and $i \in \{1, 2\}$ that
\begin{equation*}
  \frac{G(t,x',x_3 + \ell) - G(t,x',x_3)}{\ell} e_i
  = \partial_i A(t,x') e_3 . \EEE
\end{equation*}
This proves  that \EEE the difference quotient on the left-hand side  is independent \EEE of $x_3$. \EEE
 In particular, \EEE there exists $\bar G \in L^\infty(I; L^2({\Omega'}; \R^{3\times3}))$ such that for a.e.~$(t, x) \in I \times \Omega$ and $i \in \{1, 2\}$
\begin{equation*}
  G(t,x',x_3) e_i = \bar G(t,x') e_i + x_3 \partial_i A(t,x') e_3.
\end{equation*}
By  \eqref{sym_G''A_repr-new}   \EEE this leads to
\begin{equation}\label{repG}
  G(t,x',x_3)_{ji} = \bar G(t,x')_{ji} - x_3 \partial_{ji} v(t,x')
\end{equation}
  a.e.\  on \EEE  $I \times \Omega$ and  for \EEE $ i, \, j \in \{1,2\}$.
In order to identify the symmetric part of $\bar G''$, \ZZZ we \EEE employ \EEE the identity
\begin{align}\label{to pass limit}
  \int_{-1/2}^{1/2} \sym((R^h G^h)'') \di x_3
  &= \int_{-1/2}^{1/2} \sym(h^{-2}(\nabla_h y^h - \Id))'' \di x_3
    - \int_{-1/2}^{1/2} \sym(h^{-2}(R^h - \Id))'' \di x_3 \notag \\
  &= \sym(\nabla' u^h)
    - \int_{-1/2}^{1/2} \sym(h^{-2}(R^h - \Id))'' \di x_3
\end{align}
a.e.\  on \EEE $I \times \Omega'$. \ZZZ Using \EEE 
\begin{align*}
  \vert R^h G^h - G^h \vert = \vert h^{-2}  (\Id - (R^h)^T)(\nabla_h y^h - R^h)  \vert \leq   h \vert h^{-1} (\Id - (R^h)^T)\vert \vert h^{-2} (\nabla_h y^h - R^h) \vert, \EEE
\end{align*}
\ZZZ the \EEE \OOO Cauchy-Schwarz \EEE inequality,  and \EEE \eqref{rig:closetorot}--\EEE\eqref{rig:closetoid}, we see that the limits of $(\sym(R^hG^h))_h$ and $(\sym(G^h))_h\EEE$  must \EEE coincide.
Then, by  \ZZZ  \eqref{convuh},  \eqref{convAhsquaredstatement}, \eqref{sym_G''A_repr-new}, and \eqref{convGh} \EEE    we can pass to the limit \OOO $h \to 0$ \EEE in \eqref{to pass limit} to get 
\begin{equation*}
  \sym (\bar G'')
  = \sym(\nabla' u) + \frac{1}{2} \nabla'v \otimes \nabla'v
\end{equation*}
a.e.~in $I \times \Omega'$.
We conclude \eqref{sym_G''A_repr} by using \eqref{repG}.

\ZZZ We conclude this step with the proof of \eqref{convstrainrateweak}. \EEE \OOO By \eqref{rig:closetorot} and \eqref{rig:closetoidinfty} we can estimate
\begin{equation*}
  \|\nabla_h y^h - R^h\|_{L^\infty(I; L^4(\Omega))}^2
  \leq \|\nabla_h y^h - R^h\|_{L^\infty(I \times \Omega)} \|\nabla_h y^h - R^h\|_{L^\infty(I; L^2(\Omega))}
  \leq C h^{2 + 4/p}.
\end{equation*}
Hence, \GGG using \EEE
\begin{align*}
\frac{1}{2h^2} \left((\nabla_h y^h)^T \nabla_h y^h - \Id\right) = \frac{1}{2h^2}(\nabla_h y^h - R^h)^T (\nabla_h y^h -R^h) +\sym(G^h) 
\end{align*}
a.e.~in $I \times \Omega$ \OOO and  \eqref{convGh} we derive that 
\begin{align}\label{stvernantbad}
   \frac{1}{2h^2} \left((\nabla_h y^h)^T \nabla_h y^h - \Id\right) \GGG \weaklystar \EEE \sym(G) \qquad\text{weakly* in } L^\infty(I; L^2(\Omega; \R^{3\times 3})).
\end{align}
Employing \eqref{bound:strainraterotated-another one}, we find some $P\in L^{2}(I \times \Omega; \R^{3 \times 3})$ such that, up to a subsequence,
\begin{align}\label{eq: P def}
h^{-2} \sym ((\nabla_h  y^h)^T  \partial_t  \ZZZ \nabla_h \EEE  y^h) \rightharpoonup P \qquad \text{ weakly in   $ L^{2}(I \times \Omega; \R^{3 \times 3})$}. 
\end{align}
Taking the time derivative on the left-hand side in \eqref{stvernantbad}, we obtain the left-hand side of \eqref{eq: P def}, which gives $P = \partial_t\sym(G)$ and \OOO concludes the proof of  \eqref{convstrainrateweak}.

\textit{Step 2 (Proof of \eqref{conv:stress:el}--\eqref{secondgradientvanishes}):}
We   now derive \EEE compactness results for the sequence of elastic stresses. 
 Using \eqref{rig:closetoidinfty} and the definition of $G^h$  in \eqref{def:AhGh}, we   get $\|h^2 G^h\|_{L^\infty(I \times \Omega)} \le  \| \nabla_h y^h - R^h \|_{L^\infty(I \times \Omega)} \leq C h^{4/p}$.
 \EEE Hence, \EEE by \ref{W_regularity},   $\partial_F W^{\rm el} (\Id) = 0$ (see \ref{W_lower_bound_spec}), and  the symmetry of \ZZZ $\C_{W^{\rm el}}^3$, see \eqref{eq:quadraticformsnotred}, \EEE a Taylor expansion  yields for every $ \varphi \EEE \in L^1(I; L^2(\Omega; \R^{3 \times 3}))$: \EEE
\begin{align*}
&\left\vert \int_I \int_\Omega  \big( h^{-2} \partial_F W^{\rm el}(\Id+ h^2 G^h) - \C_{W^{\rm el}}^3 \sym(G) \big)  \varphi \EEE  \di x \dt \right\vert \\
&\quad  \leq \EEE
\left\vert \int_I \int_\Omega \big( \partial^2_{F^2} W^{\rm el}(\Id) G^h - \C_{W^{\rm el}}^3 G \big)  \varphi \EEE  \di x \dt \right\vert
+ C \int_I \int_\Omega \vert (R^h)^T \nabla_h y^h  - \Id \EEE \vert \vert G^h \vert \vert  \varphi \EEE \vert \di x \dt \\
&\quad \leq
\left\vert \int_I \int_\Omega \big( \partial^2_{F^2} W^{\rm el}(\Id) G^h - \C_{W^{\rm el}}^3 G \big) \varphi \EEE  \di x \dt \right\vert
 + C \Vert \nabla_h y^h - R^h \Vert\EEE_{L^\infty(I \times \Omega)} \Vert G^h \Vert_{L^\infty(I;L^2(\Omega))} \Vert  \varphi \EEE  \Vert_{L^1(I;L^2(\Omega))}.
\end{align*}
In view of \eqref{convGh}, \eqref{rig:closetoidinfty}, \ZZZ and the definition of $\C_{W^{\rm el}}^3$, \EEE we deduce  by the arbitrariness of $\varphi$ \EEE
\begin{equation}\label{weak_conv_elastic_stress}
  h^{-2} \partial_F W^{\rm el}(\Id+ h^2 G^h)
  \weaklystar \C_{W^{\rm el}}^3 \sym(G) \qquad \text{weakly* in } L^\infty(I; L^2(\Omega;\R^{3\times 3}))  \text{ as } h \to 0 \EEE.
\end{equation}
Furthermore, notice that by \ref{W_frame_invariace} we can write
\begin{equation*}
  \partial_F \elpot(\nabla_h y^h)
  = R^h \partial_F \elpot((R^h)^T \nabla_h y^h)
  = R^h \partial_F W^{\rm el}(\Id + h^2 G^h).
\end{equation*}
With  \eqref{weak_conv_elastic_stress} \EEE and \eqref{rig:closetoidinfty} this \ZZZ gives \EEE to \eqref{conv:stress:el}.

We proceed with the coupling stress.
By a similar reasoning as in the proof of Lemma \ref{est:coupl} we \GGG derive by using \EEE the continuous extension of $\partial_{F\vartheta} W^{\rm cpl}$ to $GL^+(3)\times \R_+$ in \ref{C_heatcap_cont}, $\partial_F W^{\rm cpl} (\Id, 0 ) =  0$ (see \ref{C_zero_temperature}), the first two bounds of \ref{C_bounds}, \ZZZ  and \EEE the fundamental theorem of calculus \OOO that \EEE $\vert \partial_F W^{\rm cpl}(\Id + F,  \vartheta \EEE) \vert \ZZZ \leq C \vert F \vert +  C (1+|F|) | \vartheta|\EEE$.  
 By  \eqref{theta_h_convXXX}  we find   $h^{ - 2} \theta^h \weakly  \delta_{2\alpha} \ZZZ \mu \EEE $ \EEE weakly in $L^q(I \times \Omega)$ for any $q \in [1,5/3)$, where $\delta_{2\alpha}$ denotes the Kronecker delta. Thus,  \eqref{convGh},  \EEE  \ref{C_regularity}, \OOO the definition of $\mathbb{B}^{(\alpha)}$ in  \eqref{alpha_dep}, \ZZZ \eqref{rig:closetoidinfty}, \EEE and \cite[Proposition 2.3]{MP_thin_plates} imply that
\begin{equation*}
  h^{-2} \partial_F W^{\rm cpl}(\Id+ h^2 G^h,   \theta\EEE^h)
  \weakly \ZZZ  \mu \EEE \, \mathbb{B}^{(\alpha)} \qquad \text{weakly in }  L^q  (I \times \Omega;\R^{3\times 3}).
\end{equation*}
As before, we use \ref{C_frame_indifference} and \eqref{rig:closetoidinfty} to deduce \eqref{conv:stress:cpl}.

We  now address \EEE  the hyperelastic stress. By \ref{H_bounds} and \eqref{boundres:mech} it holds that $ \int_\Omega \vert \nabla_h^2 y^h \vert^p \dx \leq C h^4$ for a.e.\ $t \in I$. Consequently, by \ref{H_bounds} and Hölder's inequality with powers $p/(p-1)$ and $p$ we derive that
\begin{equation*}
  \int_\Omega \vert \partial_G H(\nabla_h^2 y^h) \vert \di x
  \leq C \int_\Omega |\nabla_h^2 y^h|^{p-1} \di x
  \leq C \lVert \nabla_h^2 y^h \rVert_{L^{p}(\Omega)}^{p-1}
  \leq C h^{4- 4/p}
\end{equation*}
for almost every $t\in I$ and thus \eqref{secondgradientvanishes} follows, as $p>4$.

\emph{Step 3 (Proof of \eqref{conv:stress:visc}):}
 We now consider the sequence of \EEE viscous stresses and characterize their limit in terms of $G$. 
  Recall by \eqref{chain_rule_Fderiv} that one can write
\begin{equation}\label{viscousstressexpansion}
  \partial_{\dot F} R(\nabla_h y^h, \partial_t  \ZZZ \nabla_h \EEE  y^h, \theta^h)
  = 2 \nabla_h y^h
      D\big((\nabla_h y^h)^T \nabla_h y^h, \theta^h\big)
      \big((\partial_t \ZZZ \nabla_h \EEE  y^h)^T \nabla_h y^h + (\nabla_h y^h)^T  \partial_t \ZZZ \nabla_h \EEE  y^h\big).
\end{equation}
Our goal \OOO now  is to show that
\begin{equation}\label{weaklimitdissipation}
  h^{-2} \partial_{\dot F} R(\nabla_h y^h, \partial_t \ZZZ \nabla_h \EEE  y^h, \theta^h)
  \weakly 4 D(\Id, 0) \partial_t \sym(G)
  = \C_R^3 \partial_t \sym(G) \qquad \text{weakly in } L^{2}(I \times \Omega; \R^{3 \times 3}),
\end{equation}
where  \ZZZ the second identity follows from \eqref{DDDDD}. \EEE We first note that by \eqref{rig:closetoidinfty} we have $\nabla_h y^h \to \Id$ uniformly on $I \times \Omega$.
Moreover, by \eqref{theta_h_convXXX}  it holds that $\theta^h \to 0$ in $L^{1}(I \times \Omega)$, and hence, up to selecting a subsequence, pointwise a.e.~in $I \times \Omega$.
Thus,  by \ref{D_bounds} and dominated convergence  it follows that 
\begin{equation*}
  \nabla_h y^h D((\nabla_h y^h)^T \nabla_h y^h, \theta^h) \to D(\Id,0) \qquad
  \text{strongly in } L^{ q}(I \times \Omega; \R^{3 \times 3 \times 3  \times 3 })
\end{equation*}
for any $ q \in [1, \infty)$.   Hence, \OOO \eqref{eq: P def},  \eqref{viscousstressexpansion},  and  the \ZZZ fact that $P = \partial_t\sym(G)$ \OOO lead to the convergence in \eqref{weaklimitdissipation}, but \ZZZ only \OOO in the space  $L^{\OOO\tilde p}(I \times \Omega; \R^{3 \times 3})$ for any $\OOO\tilde p\in[1,2)$. Eventually, using that the sequence in \eqref{weaklimitdissipation}
is bounded in   $L^2(I \times \Omega; \R^{3 \times 3})$ by  \eqref{viscousstressexpansion}, \ref{D_bounds}, \eqref{rig:closetoidinfty}, and \eqref{bound:strainraterotated-another one}, we get by weak compactness that the convergence also holds \REV weakly \EEE in $L^{2}(I \times \Omega; \R^{3 \times 3})$.   \ZZZ This \EEE concludes   the proof of   \eqref{conv:stress:visc}.\EEE

\textit{Step 4 \GGG (Improved \EEE regularity of \EEE $ \partial_t \EEE u$ and $ \partial_t \EEE v$):}
It remains to show \ZZZ the characterization  \eqref{sym_dotG''} and \EEE that $ \partial_t \EEE u \in L^2(I; H^1(\Omega'; \R^2))$ and $ \partial_t \EEE v \in L^2(I; H^2(\Omega'))$.
We have already shown  that \EEE $\partial_t \sym(G)  \in \EEE L^2(I \times \Omega; \R^{3 \times 3})$.
 Moreover, i\EEE n Proposition~\ref{prop:compactness} (\EEE see \eqref{convuhdot} and \eqref{convvhdot}) \EEE we have  also proved \EEE that $ \partial_t \EEE \nabla'\EEE u \in L^s(I \times \Omega'\EEE)$ for some $s > 1$ and $ \partial_t \EEE \nabla'\EEE v \in L^2(I \times \Omega'\EEE)$.
With   \eqref{sym_G''A_repr}, this shows $x_3 (\nabla')^2 v \in W^{1, 1}(I ; L^1(\Omega; \R^{2 \times 2}))\EEE$.
Consequently,  using \eqref{sym_G''A_repr},  \eqref{sym_dotG''}  holds true \EEE for a.e.~$(t, x) \in I \times \Omega$.
On the one hand, multiplying both sides of \eqref{sym_dotG''}  with \EEE $-x_3$ and integrating over $x_3 \in (-1/2, 1/2)$  leads to \EEE
\begin{equation*}
  \frac{1}{12} | \partial_t \EEE (\nabla')^2 v|^2 = \left|
    \int_{-1/2}^{1/2} x_3 \partial_t \sym(G'') \di x_3
  \right|^2
  \leq \int_{-1/2}^{1/2} |\partial_t \sym(G'')|^2 \di x_3 ,
\end{equation*}
where we used Jensen's inequality in the second step.  As \EEE $\partial_t \sym(G) \in L^2(I \times \Omega; \R^{3 \times 3})$,  we get $  \frac{1}{12} | \partial_t \EEE (\nabla')^2 v|^2 \in L^1(I \times \Omega') $. \EEE With \eqref{convvhdot} this shows $ \partial_t \EEE v \in L^2(I; H^2(\Omega'; \R^2))$.

On the other hand, integrating \eqref{sym_dotG''} over $x_3 \in (-1/2, 1/2)$ we find  that \EEE
\begin{equation*}
  \sym( \partial_t \EEE \nabla'\EEE u) =
    \int_{-1/2}^{1/2} \partial_t \sym(G'') \di x_3
    -  \partial_t \EEE \nabla' v \odot \nabla'\EEE v.
\end{equation*}
\ZZZ Note that  $ \partial_t \EEE v \in L^2(I; H^2(\Omega'; \R^2))$ implies \OOO  $\nabla' v \in L^\infty(I; H^1(\Omega; \ZZZ \R^{2})) \OOO $ and $\partial_t \nabla' v \in L^2(I; H^1(\Omega; \ZZZ \R^{2})) \OOO$. \ZZZ Thus, \OOO by Sobolev embedding in space we get that $\nabla' v \in L^\infty(I; L^4(\Omega; \R^{2})$ and $\partial_t \nabla' v \in L^2(I; L^4(\Omega; \R^{2 })$, respectively. This directly \GGG yields \EEE $ \partial_t \EEE \nabla' v \odot \nabla'\EEE v \in L^2(I \times \Omega'; \R^{2 \times 2})$.  T\EEE he previous equality,   $\partial_t \sym(G) \in L^2(I \times \Omega; \R^{3 \times 3})$, \EEE and Korn's inequality  eventually lead to \EEE $ \partial_t \EEE u \in L^2(I; H^1(\Omega; \R^2))$.
Note that we have used $u(t, \cdot) = 0$  a.e.~on \EEE $\Gamma_D'$ for a.e.~$t \in I$, see also Proposition~\ref{prop:compactness}.
\end{proof}

\subsection{Convergence of solutions}\label{sec:convergencesol}
In this subsection, we   prove our main theorem.  As a preparation, we recall an energy balance in the \GGG three-dimensional setting. \EEE
 
    \begin{remark}[Energy balance of rescaled solutions]
     Let $(y^h,\theta^h)$ be as in Remark \EEE \ref{rem:weakformresc}.  
     Then, for a.e.~$t \in I$ it holds that \EEE
\begin{align}\label{balance3D}
       \mathcal{M}(y^h(t)) \EEE -   \mathcal{M}(y^h_0) \EEE   &=  \int_0^t \int_{\Omega} f^{3D}_h(t)  \partial_t \EEE y^h_3 \dx \di s  -   \int_0^t  \int_{\Omega}  2 R(\nabla_h y^h,  \partial_t \nabla_h   y^h, \theta^h) \dx  \di s  \EEE\notag\\
      &\phantom{=}\quad \ZZZ - \EEE \int_0^t \int_\Omega \partial_F W^{\rm cpl}(\nabla_h y^h, \theta^h):  \partial_t \EEE \nabla_h y^h \dx \di s.
    \end{align}
    Indeed, this can be seen formally by testing \eqref{weak_form_mech_res} with $ \partial_t \EEE \ZZZ \nabla_h \EEE  y^h$ and  using \EEE a chain rule,  as well as \eqref{eq: free energy} and \eqref{diss_rate}. \EEE
    For a rigorous derivation, we refer to \cite[Proof of Proposition 5.1, Step 3]{MielkeRoubicek20Thermoviscoelasticity}, relying on the \ZZZ  chain \EEE rule \cite[Proposition 3.6]{MielkeRoubicek20Thermoviscoelasticity}.
    \end{remark}

\begin{proof}[Proof of Theorem~\ref{maintheorem}]

As in the proof of Proposition~\ref{prop:compactness}, we can assume that  there exists \EEE a sequence of solutions $(\EEE(y^h,\theta^h ))\EEE_h$ in the sense of  Definition \ref{def:weak_formulation} satisfying \EEE \eqref{boundres:mech}--\eqref{forAubin-Lion3}, and \eqref{rig:closetorot}--\eqref{boundres:strainratesym} for a sequence of rotations $(R^h)_h \subset \EEE L^\infty(I;H^1(\Omega'; SO(3)))$. \GGG In particular, we can make use of the properties given in Proposition~\ref{prop:compactness}, Corollary~\ref{a new corollary}, Lemma~\ref{corol:compactnesscorollary}, and Lemma~\ref{lem:conv_strain_stress}.  \EEE

 T\EEE he proof  is divided \GGG into \EEE  seven \EEE steps.
In the first  step\EEE, we use \eqref{weak_form_mech_res}   to further characterize the limiting stresses. After  proving that the \EEE skew-symmetric part of th\OOO e r\EEE escaled  limiting \EEE stress is of lower order (Step 2), we derive the  limiting \EEE mechanical equations in Step 3 and Step 4.  Step 5 is devoted to \OOO deriving an appropriate \EEE energy balance in the \GGG two-dimensional \EEE setting. \EEE In order to prove the convergence of the heat equation, we  first need to show \EEE that \eqref{convstrainrateweak} holds with \emph{strong} convergence (Step 6). Only then, we can pass to the limit in the heat equation (Step 7).

\textit{Step 1 (Further characterization of the limiting  stresses):}
 Our first goal \OOO is  to characterize the different parts of the limiting stress  \EEE
\begin{align}\label{defEmat}
 \Sigma \EEE \coloneqq \CW^3 \sym(G) + \mu \mathbb B^{(\alpha)}  + \CD^3 \partial_t\sym(G),
\end{align}
 where the \OOO constant  tensors \OOO appearing above  are defined below \eqref{eq:quadraticformsnotred} and in  \eqref{alpha_dep}, and $G$ \GGG is  given as \EEE in Lemma \ref{lem:conv_strain_stress}. \OOO More precisely, w\EEE e  aim to \EEE show   \EEE
\begin{align}\label{CWCD}\hspace{-0.3cm}
  \CW^3 \sym(G) = \begin{bmatrix}
    \CW^2 \sym(G'') & 0 \\
    0 & 0
  \end{bmatrix},  \ \ 
  \CD^3 \partial_t \sym(G) = \begin{bmatrix}
   \CD^2 \partial_t\sym(G'') & 0 \\
   0 & 0
  \end{bmatrix}, \ \ 
  \mathbb{B}^{(\alpha)} = \begin{bmatrix}
     (\mathbb{B}^{(\alpha)})''  & 0 \\
    0 & 0
  \end{bmatrix},
\end{align}
where the tensors  $\CW^2$, $\C_R^2 \in \R^{2 \times 2 \times 2 \times 2}$   are \ZZZ introduced \EEE below \eqref{def:quadraticforms}, and the $2 \times 2$-submatrix $\sym(\ZZZ G'') \EEE $  has \EEE been characterized in \eqref{sym_G''A_repr}.

The third \OOO equality  of \eqref{CWCD} follows from \OOO the  assumption \eqref{heattensorassumption}\OOO. Hence, it remains  to characterize $  \CW^3 \sym(G)$ and $  \CD^3 \partial_t \sym(G)$.  As a preliminary step, we show that
\begin{equation}\label{Ezerocolrow}
 \Sigma  e_3 = 0 \quad\text{a.e.~in $I \times \Omega$.}
\end{equation}
In \OOO this regard\EEE, \EEE notice that $ \vphi_y \EEE (t, x',x_3) \coloneqq \int_0^{x_3} \Phi(t, x', \tilde x_3) \di \tilde x_3$ is an admissible test function  in \EEE \eqref{weak_form_mech_res} for any $\Phi \in C^\infty(I; C^\infty_c(\Omega; \R^3))$.
Consequently, using $ \vphi_y \EEE $ in \eqref{weak_form_mech_res}, dividing the equation by $h$ and employing \eqref{forcebound}  as well as \EEE \eqref{conv:stress:el}--\eqref{conv:stress:visc} we discover that
\begin{equation*}
  \int_I \int_\Omega \Sigma e_3 \cdot \Phi \di x \di t = 0.
\end{equation*}
By the arbitrariness of $\Phi$ this  shows  \eqref{Ezerocolrow}.  \EEE

Condition \eqref{quadraticformsassumption}, \ref{W_frame_invariace}, and \ref{D_quadratic} imply for $i \in \{ 1, 2, 3 \}$ and $k, \, l \in \{ 1, 2 \}$ that  
\begin{equation*}
  (\C_S^3)_{3ikl} = (\C_S^3)_{i3kl} = (\C_S^3)_{kli3} = (\C_S^3)_{kl3i} = 0 \qquad \text{for } S \in \{ W^{\rm el}, R\}.
\end{equation*}
In particular, by mapping a matrix $F \in \R^{3 \times 3}_\sym$ to the vector $\tilde F^T = ( F_{11} , 2F_{12} , F_{22} , 2F_{13} , 2F_{23} , F_{33})^T$,
we can identify the 4th-order tensors $\CW^3$ and $\CD^3$ with matrices $ \tilde \C_{W^{\rm el}}^3$, $\tilde \C_R^3 \in \R^{6 \times 6}$ \OOO given by \EEE
\begin{align*}
  \tilde \C_{W^{\rm el}}^3 &= \begin{bmatrix}
    A_1 & 0 \\
    0 & A_2
  \end{bmatrix}, &
  \tilde \C_R^3 &= \begin{bmatrix}
    B_1 & 0 \\
    0 & B_2
  \end{bmatrix},
\end{align*}
 respectively, \EEE
where $A_1, \, A_2, \, B_1, \, B_2 \in \R^{3 \times 3}$ such that for any $F \in \R^{3\times 3}_\sym$ it holds that
\begin{align*}
  F: \C_S^3 F = \tilde F \cdot \tilde \C_S^3 \tilde F \qquad {\rm for } \ S \in \{ W^{\rm el} , R\}.
\end{align*}
Note that the matrices $A_1$, $A_2$, $B_1$, and $B_2$ are invertible due to the positive definiteness of $\CW^3$ and $\C_R^3$\OOO, respectively\EEE.  Due to \eqref{quadraticformsassumption}, in the above sense we can identify the 2nd-order tensors  $A_1$ and $B_1$ with the reduced 4th-order tensors  $\C_{W^{\rm el}}^2$  and  $ \C_R^2$, respectively. \EEE
Combining  these facts \EEE with \eqref{defEmat},  \eqref{Ezerocolrow},   and \eqref{heattensorassumption} leads to the following system of ODEs:
\begin{equation}\label{corkequation}
   \partial_t \begin{pmatrix}
    2\sym(G)_{13} \\
    2\sym(G)_{23} \\
    \sym(G)_{33} \\
  \end{pmatrix} = - B_2^{-1} A_2 \begin{pmatrix}
    2\sym(G)_{13} \\
    2\sym(G)_{23} \\
    \sym(G)_{33} \\
  \end{pmatrix}  \qquad {\rm for \ a.e.\ \ } (t,x) \in  I \times \EEE \Omega.
\end{equation}
\ZZZ Next, \GGG we \EEE check that the initial values satisfy 
\begin{equation}\label{eq: initiii val}
  \sym(G(0))_{13} = \sym(G(0))_{23} = \sym(G(0))_{33} = 0 \qquad \text{for a.e.~} x \in \Omega.
\end{equation}\EEE 
\ZZZ To this end, we  investigate  convergence at initial time.  By \cite[Lemma 5.3, Theorem 5.6]{FK_dimred} (see also \cite{FJM_hierarchy}) we find a sequence $(R^h_0)_h \subset H^1(\Omega';SO(3))$ satisfying
\begin{align}\label{convi-propi}
 \Vert \nabla_h y^h_0 - R^h_0 \Vert_{L^2(\Omega)} &\le C h^2
\end{align}
such that $G^h_0 = \OOO h\ZZZ^{-2} ( (R^h_0)^T \nabla_h y^h_0 - \Id)$ satisfies $G^h_0 \rightharpoonup G_0$ weakly in $L^2(\Omega;\R^{3\times 3})$ for some $G_0 \in L^2(\Omega;\R^{3\times 3})$, and it holds that 
$$ \liminf\limits_{h \to 0}  h^{-4} \int_\Omega W^{\rm el} (\nabla_h y^h_0 ) \dx + h^{-4 } \int_\Omega H(\nabla_h^2 y^h_0) \dx  \geq \int_\Omega    Q^3_{W^{\rm el}}  \big( \sym ( G_0 ) \big) \di x\geq\int_\Omega    Q^2_{W^{\rm el}}   \big( \sym ( G_0'' ) \big) \di x. $$
Moreover, we have  $G_0'' = \sym(\nabla' u_0)
    + \frac{1}{2} \nabla'v_0 \otimes \nabla' v_0
    - x_3 (\nabla')^2 v_0$ for the initial displacements $u_0$ and $v_0$. In particular, using $\int_{-1/2}^{1/2} x_3 \di x_3 = 0$, $\int_{-1/2}^{1/2} x_3^2 \di x_3 = 1/12$, and \eqref{eq: phi0}, this implies $\int_\Omega    Q^2_{W^{\rm el}}   \big( \sym ( G_0'' ) \big) \di x = \phi^{\rm el}_0(u_0,v_0)$. Therefore, using \eqref{well-preparednessinitial} we obtain
$$ \phi^{\rm el}_0(u_0,v_0)    \geq \int_\Omega    Q^3_{W^{\rm el}}  \big( \sym ( G_0 ) \big) \di x \ge  \int_\Omega    Q^2_{W^{\rm el}}   \big( \sym ( G_0'' ) \big) \di x  =  \phi^{\rm el}_0(u_0,v_0). $$
\EEE
Thus, all \GGG inequalities \ZZZ  turn out to be   equalities.  As $Q^3_{W^{\rm el}}$ is positive definite on $\R^{3 \times 3}_{\rm sym}$, this shows
 \begin{equation}\label{initialrelation1}
  \sym(G_0)_{13} = \sym(G_0)_{23} = \sym(G_0)_{33} = 0 \qquad \text{for a.e.~} x \in  \Omega.
\end{equation} 
\ZZZ We now transfer this property to $G(0)$.    Using the identity 
\begin{align}\label{iiii}
(2h)^{-2}  \big((\nabla_h y^h)^T \nabla_h y^h - \Id\big) = (2h)^{-2} (\nabla_h y^h - R^h)^T (\nabla_h y^h -R^h) +\sym(G^h) 
\end{align}
along with \eqref{rig:closetorot}, \GGG \eqref{rig:closetoidinfty}, \ZZZ \eqref{convGh}, and \eqref{convstrainrateweak} shows that the left-hand side of \eqref{iiii} \GGG is bounded in \ZZZ $L^\infty(I; L^2(\Omega; \R^{3\times 3})) \cap H^1(I; L^2(\Omega; \R^{3\times 3}))$. Then, using the \EEE Aubin-Lions lemma we \GGG get, up to a subsequence, \EEE that $ 
(2h)^{-2} \big((\nabla_h y^h)^T \nabla_h y^h - \Id\big) \to \sym(G) $
strongly in $C(I; \REV ( H^{1}(\Omega))^* \EEE)$,  and thus the initial values satisfy $
(2h)^{-2} \left((\nabla_h y^h_0)^T \nabla_h y^h_0 - \Id\right) \to \sym(G(0))$ \ZZZ in \EEE $\REV ( H^{1}(\Omega))^* \EEE$, up to a subsequence. \ZZZ Using \eqref{iiii} for $y_0^h$ and $R^h_0$ in place of $y^h$ and $R^h$, and \eqref{convi-propi}  \EEE yields $\sym(G(0)) = \sym(G_0)$ in $\REV ( H^{1}(\Omega))^* \EEE$. \ZZZ This \EEE together with \eqref{initialrelation1} shows \eqref{eq: initiii val}. \EEE

In Lemma \ref{lem:conv_strain_stress} we have shown that $\sym (G) \in L^\infty(I; L^2(\Omega;\R^{3 \times 3}))$ and $\partial_t \sym (G)\in L^2(I\times \Omega;\R^{3 \times 3})$.
Hence, we can \OOO multiply \EEE \eqref{corkequation} with $ \xi_G \defas (2 \sym (G)_{13},2 \sym (G)_{23}, \sym (G)_{33})^T$ \OOO and integrate over $\Omega$\EEE.
By the positive definiteness of $B_2^{-2} A_2$ this yields
\begin{align*}
\frac{{\rm d}}{{\rm d}t} \Vert \xi_G (t) \Vert_{L^2(\Omega)}^2 = 2 \int_\Omega  \partial_t \EEE \xi_G (t) \cdot \xi_G (t) \dx = \int_\Omega -2 B_2^{-1} A_2 \xi_G (t) \cdot \xi_G (t) \dx \leq - C \Vert \xi_G (t) \Vert_{L^2(\Omega)}^2 \leq 0
\end{align*}
for a.e.~ $t \in\EEE I$.
\OOO With Gronwall's inequality and \eqref{eq: initiii val} t\EEE his  shows \EEE $\xi_G  \equiv 0$.
\OOO Consequently, \eqref{defEmat} and \eqref{Ezerocolrow} \ZZZ along with   \eqref{heattensorassumption} \EEE imply that \eqref{CWCD} holds.

\textit{Step 2 (Convergence of  the \EEE rescaled stress):}
In order to prove the convergence of the mechanical equations, let us define
\begin{equation*}
  \Sigma^h \coloneqq h^{-2} (R^h)^T \big(
   \partial_F W^{\rm el}(\nabla_h y^h)
    + \partial_F W^{\rm cpl}(\nabla_h y^h,\theta^h)
    + \partial_{\dot F} R (\nabla_h y^h,  \partial_t \EEE \ZZZ \nabla_h \EEE  y^h, \theta^h)
  \big)
\end{equation*}
as well as \GGG its zeroth and first moments \EEE
\begin{align}\label{def:moments}
  \bar \Sigma^h(t, x') &\defas\EEE
    \int_{-\frac{1}{2}}^{\frac{1}{2}} \Sigma^h(t, x) \di x_3, &
  \hat \Sigma^h(t, x') &\defas\EEE
    \int_{-\frac{1}{2}}^{\frac{1}{2}} x_3 \Sigma^h(t, x) \di x_3,
\end{align}
respectively.  In a similar fashion, \ZZZ recalling $\Sigma$ defined in \eqref{defEmat}, \EEE  we define $ \bar \Sigma$ and $\hat \Sigma$ \OOO as the zeroth and first moment \EEE of the limiting strain\OOO, respectively\EEE. The goal of this step is to show that, for any  $q \in [1, 5/3)$, we have \EEE
    \begin{subequations}
    \begin{gather}
 R^h\Sigma^h  \rightharpoonup  \Sigma \  \text{ weakly in } L^q(I \times \Omega; \R^{3\times 3}),\EEE  \quad  R^h \bar \Sigma^h \rightharpoonup \bar \Sigma,  \ R^h \hat \Sigma^h \rightharpoonup \hat \Sigma \  \text{ weakly  in } L^q(I \times \Omega'; \R^{3\times 3}), \label{EhconvergenceRh}\\
 \bar \Sigma^h   \rightharpoonup \bar \Sigma, \quad\hat \Sigma^h \rightharpoonup \bar \Sigma  \quad   \text{weakly  in } L^q(I \times \Omega'; \R^{3\times 3}), \label{Ehconvergence}\\
h^{-1} \skw (\bar \Sigma^h) \to 0  \quad \ \ \ \text{ strongly in } L^1(I \times \Omega'; \R^{3\times 3}). \label{skewsymmetricvanish}
    \end{gather}
    \end{subequations}
By \eqref{conv:stress:el}, \eqref{conv:stress:cpl},   and \EEE \eqref{conv:stress:visc} we get \eqref{EhconvergenceRh}. Using also \eqref{rig:closetoidinfty}, we additionally derive \eqref{Ehconvergence}.  

We now show \eqref{skewsymmetricvanish}.  In this regard,   let us first write 
\begin{align}\label{skew} 
  (R^h)^T \partial_F W^{\rm el}(\nabla_h y^h)  & =   (R^h)^T \partial_F W^{\rm el}(\nabla_h y^h) (\nabla_h y^h)^T R^h  +  (R^h)^T \partial_F W^{\rm el}(\nabla_h y^h) \big(\Id -    (\nabla_h y^h)^T R^h\big).
\end{align} \EEE
B\EEE y \ref{W_frame_invariace}, it holds that $0= \partial_t \OOO(\EEE W^{\rm el}(e^{tS} F)\OOO)\EEE \vert_{t=0} = \partial_F W^{\rm el}(F)F^T:S$ for every $S \in \R^{3\times 3}_{\skw}$. Hence, $\partial_F W^{\rm el}(F) F^T$ is symmetric for every $F \in  GL^+(3) \EEE$,  and in particular \OOO also  $(R^h)^T \partial_F W^{\rm el}(\nabla_h y^h) (\nabla_h y^h)^T R^h  $ is symmetric. Thus, taking the   skew-symmetric \EEE \EEE parts of the matrices in \eqref{skew} we find 
\begin{align*}
{\rm skew} \big(  (R^h)^T \partial_F W^{\rm el}(\nabla_h y^h) \big)  & =    {\rm skew} \Big( (R^h)^T \partial_F W^{\rm el}(\nabla_h y^h) \big(\Id -    (\nabla_h y^h)^T R^h \big)\Big).   
\end{align*}
\OOO Then, \EEE by  \eqref{rig:closetorot} and  \EEE  \eqref{conv:stress:el}  it follows that \EEE
\begin{align}\label{skew:el}
\Vert {\rm skew} \big( (R^h)^T \partial_F W^{\rm el}(\nabla_h y^h)\big)\Vert_{L^\infty(I;L^1(\Omega))} &\leq  C \Vert \partial_F W^{\rm el}(\nabla_h y^h) \Vert_{L^\infty(I;L^2(\Omega))}  \Vert \big(  (R^h)^T \nabla_h y^h  - \Id \big) \Vert_{L^\infty(I;L^2(\Omega))}  \notag \\ &\leq Ch^4 .
\end{align}
We proceed similarly with the coupling term.  U\EEE sing \ref{C_frame_indifference} instead of \ref{W_frame_invariace}, we  can \EEE derive a similar identity for the  skew-symmetric \EEE part of $(R^h)^T \partial_F W^{\rm cpl}(\nabla_h y^h,\theta^h)$.  As the derivation is very similar to the one above, we omit further details. \QQQ Using \GGG Hölder's inequality with powers $4/3$ and $4$, we get \EEE
\begin{align}\label{skew-neu}
&\Vert {\rm skew} \big( (R^h)^T \partial_F W^{\rm cpl}(\nabla_h y^h,\theta^h)\big)\Vert_{L^1(\GGG I\times \Omega) \EEE}   \leq C \Vert   \partial_F W^{\rm cpl}(\nabla_h y^h,\theta^h) \Vert_{L^{\GGG 4/3 \EEE}(I \times \Omega)} \Vert (R^h)^T  \nabla_h y^h - \Id      \Vert_{L^{\GGG 4 \EEE}(I \times \Omega)}.
\end{align}
Moreover, \eqref{rig:closetoidinfty} and \eqref{rig:closetorot}  yield \EEE  
\begin{align*}
\Vert (R^h)^T  \nabla_h y^h - \Id  \big)    \Vert_{L^{\ZZZ 4}(I \times \Omega)} &\leq  \left( \int_I \int_\Omega \vert  (R^h)^T  \nabla_h y^h - \Id  \vert^{\ZZZ 2} \vert  (R^h)^T  \nabla_h y^h - \Id  \vert^2 \di x \di t \right)^{1/4}\\
&\leq  \Vert (R^h)^T  \nabla_h y^h - \Id  \Vert_{L^\infty(I \times \Omega)}^{\ZZZ 1/2} \left( \int_I \int_\Omega \vert (R^h)^T  \nabla_h y^h - \Id  \vert^2 \di x \di t \right)^{1/4}\\
&\leq C \ZZZ h^{2/p}h. \EEE
\end{align*}
\ZZZ Thus, \EEE from \eqref{skew-neu} and \eqref{conv:stress:cpl} \ZZZ  we get  \EEE 
\begin{align}\label{skew:cpl}
\Vert {\rm skew} \big( (R^h)^T \partial_F W^{\rm cpl}(\nabla_h y^h,\theta^h)\big)\Vert_{L^1 \GGG (I \times\Omega) \EEE}  = \GGG { o} \EEE(h^3). \EEE
\end{align}
Shortly writing $C^h \defas (\nabla_h y^h)^T \nabla_h y^h$ and $\dot C^h \defas ( \partial_t \EEE \ZZZ \nabla_h \EEE  y^h)^T \nabla_h y^h + (\nabla_h y^h)^T  \partial_t \EEE \ZZZ \nabla_h \EEE  y^h$, we have by \ref{D_quadratic} that $\skw (D(C^h,\theta^h) \dot C^h) = 0$. Thus, by \eqref{chain_rule_Fderiv} \ZZZ and \ref{D_quadratic} \EEE it holds that
\begin{align*}
 \skw \big( (R^h)^T \partial_{\dot F} R(\nabla_h y^h,  \partial_t \EEE \ZZZ \nabla_h \EEE  y^h, \theta^h)\big) & = \ZZZ 2 \,   \skw  ( (R^h)^T \nabla_h y^h -\Id) (D(C^h, \theta^h) \dot C^h)  +   2 \skw  (D(C^h, \theta^h) \dot C^h) \EEE   \\
  & =  2 \EEE \, \skw \big( h^2 G^h (D(C^h, \theta^h)  \dot C^h) \big) ,
\end{align*}
where $G^h$ is  as \EEE in \eqref{def:AhGh}.
 Consequently, \EEE by the Cauchy-Schwarz inequality,  \ref{D_bounds},  \eqref{bound:strainraterotated-another one}, \EEE and \eqref{convGh} it follows that
\begin{align}\label{skew:visc}
\Vert {\rm skew} \big( (R^h)^T \partial_{\dot F} R(\nabla_h y^h,  \partial_t \EEE \ZZZ \nabla_h \EEE  y^h, \theta^h)\big) \Vert_{L^1(I;L^1(\Omega))} \leq Ch^4.
\end{align}
With \eqref{skew:el}, \eqref{skew:cpl}, \eqref{skew:visc}, \EEE and \eqref{def:moments} we conclude \eqref{skewsymmetricvanish}.

\textit{Step 3 (Convergence  to \EEE  the first mechanical equation):} In this step, we prove that the triplet $(u,v,\mu)$ solves \eqref{eq: weak equation1}. \EEE
 Let us test \EEE \eqref{weak_form_mech_res} with
$ \vphi_y \EEE = (\varphi_u, 0)$ for $\varphi_u \in C^\infty(I \times \overline{\Omega'}; \R^2)$ with $\varphi_u = 0 $ on $I \times \Gamma_D'$ and divide  both sides \EEE by $h^2$, resulting in
\begin{align*}
 0  & =  \int_I \int_\Omega \Big( 
    R^h \Sigma^h : \nabla_h \vphi_y 
    + \frac{1}{h^2} \partial_G H(\nabla^2_h y^h) \cdddot \nabla^2_h \vphi_y  \Big) \di x \di t \EEE \\
  & =  \int_I \int_{\Omega'} (R^h \bar \Sigma^h)'' : \nabla' \varphi_u \di x' \di t
    + \frac{1}{h^2} \int_I \int_{\Omega}
      \partial_G H(\nabla^2_h y^h) \cdddot \nabla^2_h \vphi_y   \di x \di t, \EEE \\
\end{align*}
where   in the first step we used $(\vphi_y )_3 \equiv 0$ and  in the second step \EEE we used   the fact that $\varphi_u$ and $R^h$ do not depend on $x_3$.
After rearranging terms, we discover that
\begin{equation*}
  \left\vert \int_I \int_{\Omega'} (R^h \bar \Sigma^h)'' : \nabla' \varphi_u \di x' \di t \right\vert
  \leq  h^{-2}\int_I \int_\Omega \vert \partial_G H(\nabla_h^2 y^h) \vert \vert (\nabla')^2 \varphi_u \vert \di x \di t.
\end{equation*}
 Due to \EEE \eqref{EhconvergenceRh} and \eqref{secondgradientvanishes}, passing to  the limit $h \to 0$ we find
\begin{equation}\label{zeromomeq}
  \int_I \int_{\Omega'}   \bar \Sigma_{ij} \ZZZ \colon  \nabla' \varphi_u \EEE  \di x' \di t
  = 0.
\end{equation}
 By \eqref{defEmat}--\eqref{CWCD} and  \EEE  the characterization in  \eqref{sym_G''A_repr}--\eqref{sym_dotG''} we have that
\begin{equation}\label{E_repr}
\begin{aligned}
  \LLL \Sigma'' &= \CW^2 \big(
    \sym(\nabla' u)
    + \frac{1}{2} \nabla'v \otimes \nabla' v
    - x_3 (\nabla')^2 v
  \big)
  + \mu (\mathbb{B}^{(\alpha)})'' \\
  &\phantom{=}\quad + \CD^2 \big(
    \sym( \partial_t \EEE \nabla' u)
    +   \partial_t \EEE \nabla' v \odot \nabla' v 
    - x_3  \partial_t \EEE (\nabla')^2 v
  \big) \\
\end{aligned}
\end{equation}
 Thus,  by \eqref{zeromomeq}, \EEE \eqref{E_repr}, and the fact that $x_3$ only appears as a linear factor in \ZZZ \eqref{E_repr}, \EEE we find the equation
\begin{align}\label{zeromomeq2}
  \int_I \int_{\Omega'} \Big( \C_{W^{\rm el}}^2 \big(
    \sym(\nabla' u)
    + \frac{1}{2} \nabla'v \otimes \nabla' v
  \big)
  + \mu (\mathbb{B}^{(\alpha)})'' + \C_R^2 \big(
    \sym( \partial_t \EEE \nabla' u)
    +    \partial_t \EEE \nabla' v \odot \nabla' v 
  \big) \Big) : \nabla' {\varphi_u} \di x' \di t
  = 0.
\end{align}
This  concludes the proof of \EEE \eqref{eq: weak equation1}.

\textit{Step 4 (Convergence  to \EEE the second mechanical equation):}
We now derive the second  limiting \EEE mechanical equation \eqref{eq: weak equation2}.
In this regard, we test \eqref{weak_form_mech_res} with $ \vphi_y \EEE = (0,0,{\varphi_v})$ for ${\varphi_v} \in C^\infty(I\times  \overline{\Omega'})$  such that \EEE ${\varphi_v} = 0$ on $I \times \Gamma_D'$ and multiply both sides by $h^{-3}$, which leads to
\begin{align*}
 \int_I \int_\Omega
       h^{-1} \big( R^h \Sigma^h \big)  : \nabla_h \vphi_y  \dx \dt
      + h^{-3} \partial_G H(\nabla^2_h y^h) \cdddot \nabla^2_h \vphi_y 
    \di x \di t  =  h^{-3} \int_I \int_\Omega  f^{3D}_h (\vphi_y )_3 \di x \di t.
\end{align*}
Hence, by the definition of $\bar \Sigma^h$ in \eqref{def:moments}, the definition of  $\vphi_y$, \ZZZ and the fact that $f_h^{3D} $ is independent of the $x_3$-variable, \EEE it follows that
\begin{align*}
  \left\vert \int_I \int_{\ZZZ \Omega'} h^{-3} f_h^{3D}   \, \EEE {\varphi_v} \ZZZ \di x'  \EEE \di t - \int_I \int_{\Omega'} \sum_{i = 1}^2
    \frac{1}{h} ( R^h \bar \Sigma^h)_{3 i} \partial_i {\varphi_v} \di x' \di t \right\vert
  \leq h^{-3} \int_I \int_\Omega \vert \partial_G H(\nabla^2_h y^h) \vert \vert (\nabla')\EEE^2 {\varphi_v} \vert \di x \di t
     .
\end{align*}
By \eqref{secondgradientvanishes} the right-hand side tends to $0$ as $h \to 0$.
Hence, by \eqref{forcebound} we derive that
\begin{equation*}
  \lim_{h \to 0} \int_I \int_{\Omega'} \sum_{i = 1}^2 \frac{1}{h} (R^h \bar \Sigma^h)_{3i} \partial_i {\varphi_v} \di x' \di t
  = \int_I \int_{\Omega'} f^{2D} {\varphi_v} \di x' \di t.
\end{equation*}
Moreover,  considering again $A^h = h^{-1}(R^h - \Id)$, \EEE   by the identity $h^{-1} R^h \bar \Sigma^h = A^h \bar \Sigma^h + h^{-1} \bar \Sigma^h$, \eqref{convAhlem},   \eqref{Ehconvergence}, \OOO and the above limit, \EEE it follows that \EEE
\begin{equation}\label{lin_mech_sec_id1}
  \lim_{h \to 0} \int_I \int_{\Omega'} \sum_{i = 1}^2
    \frac{1}{h} \bar \Sigma^h_{3 i} \partial_i {\varphi_v} \di x' \di t
  = - \int_I \int_{\Omega'} \ZZZ \sum_{i,k = 1}^2 \EEE  A_{3k} \bar \Sigma_{k i}  \partial_i {\varphi_v} \di x' \di t
    + \int_I \int_{\Omega'} f^{2D} {\varphi_v} \di x' \di t.
\end{equation}
We next test \eqref{weak_form_mech_res} with $ \vphi_y \EEE (t, x) \coloneqq (x_3 \eta(t, x'), 0)$ for $\eta \in C^\infty(I\times  \overline{\Omega'};\R^2)$ with $\eta = 0 $ on $I \times \Gamma_D'$ and multiply both sides by $h^{-2}$  leading to \EEE
\begin{align*}
 \int_I \int_\Omega
      R^h \Sigma^h : \nabla_h  \vphi_y
    \di x \di t           + h^{-2} \int_I \int_\Omega
     \partial_G H(\nabla^2_h y^h) \cdddot \nabla^2_h  \vphi_y
    \di x \di t   = 0.
\end{align*}
 Notice that b\EEE y our choice of $ \vphi_y$ it holds that $\partial_{33}  \vphi_y = 0$ and thus $\nabla^2_h  \vphi_y = \bigo(\frac{1}{h})$  as $h \to 0$\EEE. 
Rewriting $\Sigma^h$  with \EEE $\hat \Sigma^h$ and $\bar \Sigma^h$  as in \EEE \eqref{def:moments}, we derive by \eqref{secondgradientvanishes} that
\begin{align*}
  &\Bigg\vert \int_I \int_{\Omega'} \sum_{i, j = 1}^2 ( R^h \hat \Sigma^h)_{ij} \partial_j \eta_i \di x' \di t
  + \int_I \int_{\Omega'}  \sum_{i = 1}^2 \frac{1}{h} ( R^h \bar \Sigma^h)_{i3} \eta_i \di x' \di t  \Bigg\vert   \leq   C \frac{1}{h^{2}} \int_I \int_\Omega \vert \partial_G H(\nabla^2_h y^h) \vert h^{-1} \di x \di t \to 0.
\end{align*}
Therefore, by \eqref{EhconvergenceRh}, the identity $h^{-1} R^h \bar \Sigma^h = A^h \bar \Sigma^h + h^{-1} \bar \Sigma^h$, \eqref{convAhlem}, \EEE and \eqref{skewsymmetricvanish}  we have \EEE
\begin{align*}
  \int_I \int_{\Omega'} \sum_{i, j = 1}^2  \hat \Sigma_{ij} \partial_j \eta_i \di x' \di t
  &= \lim_{h \to 0} \int_I \int_{\Omega'} \sum_{i, j = 1}^2 ( R^h \hat \Sigma^h)_{ij} \partial_j \eta_i \di x' \di t  = - \lim_{h \to 0} \int_I \int_{\Omega'} \sum_{i = 1}^2 \frac{1}{h} (R^h\bar \Sigma^h)_{i 3} \eta_i \di x' \di t \notag \\
  &= -\lim_{h \to 0} \left(\int_I \int_{\Omega'}  \sum_{i = 1}^2 \frac{1}{h}  \bar \Sigma^h_{i3} \eta_i \di x' \di t + \int_I \int_{\Omega'}  \sum_{i = 1}^2 \sum_{k=1}^3 A_{ik}^h \bar \Sigma_{k3}^h \eta_i \di x' \di t\right) \notag \\
     &= -\lim_{h \to 0} \int_I \int_{\Omega'}  \sum_{i = 1}^2 \frac{1}{h}  \bar \Sigma^h_{i3} \eta_i \di x' \di t
\end{align*}
where we  used that $\bar \Sigma_{k3} = 0$, see \EEE \eqref{Ezerocolrow}. \ZZZ  Then, \GGG \eqref{skewsymmetricvanish} \ZZZ shows
\begin{align*}
  \int_I \int_{\Omega'} \sum_{i, j = 1}^2  \hat \Sigma_{ij} \partial_j \eta_i \di x' \di t & =  -\lim_{h \to 0} \int_I \int_{\Omega'}  \sum_{i = 1}^2 \frac{1}{h}  \bar \Sigma^h_{3i} \eta_i \di x' \di t   -\lim_{h \to 0} \int_I \int_{\Omega'}  \sum_{i = 1}^2 \frac{1}{h} \big( \bar \Sigma^h_{i3} -  \bar \Sigma^h_{3i}  \big)  \eta_i \di x' \di t \notag \\
  & =  -\lim_{h \to 0} \int_I \int_{\Omega'}  \sum_{i = 1}^2 \frac{1}{h}  \bar \Sigma^h_{3i} \eta_i \di x' \di t.
\end{align*} \EEE
 U\EEE sing $\eta = \nabla' {\varphi_v}$  \ZZZ in the above relation, we then  derive by \eqref{lin_mech_sec_id1} \EEE 
 \begin{align}\label{firstmomec1}
 \int_I \int_{\Omega'} \ZZZ \sum_{i,k = 1}^2 \EEE  A_{3k} \bar \Sigma_{k i}  \partial_i {\varphi_v} \di x' \di t - \int_I \int_{\Omega'} \sum_{i, j = 1}^2 \hat \Sigma_{ij} \partial_{ji} {\varphi_v} \di x' \di t
    = \int_I \int_{\Omega'} f^{2D} {\varphi_v} \di x' \di t.
 \end{align} 
Then, using \eqref{sym_G''A_repr-new}, \eqref{E_repr}, \ZZZ $\int_{-1/2}^{1/2} x_3 \di x_3 = 0$, \EEE   $\int_{-1/2}^{1/2} x_3^2 \di x_3 = 1/12$,    the fact that $x_3$ only appears as a linear factor in $\Sigma$, \ZZZ and the identity $a^T M b  = \sum_{k,i=1}^2 a_k M_{ki} b_i = M \colon (a  \odot b) $ for $a,b \in \R^2$ and $M \in \R^{2 \times 2}_{\rm sym}$,  \EEE we can rewrite \OOO \eqref{firstmomec1} \EEE as   \EEE
\begin{equation}\label{firstmomec2}
\begin{aligned}
&\int_I \int_{\Omega'} \left( \C_{W^{\rm el}}^2 \big(
    \sym(\nabla' u)
    + \frac{1}{2} \nabla'v \otimes \nabla' v
  \big)
  + \mu (\mathbb{B}^{(\alpha)})''
   \right) : \nabla' v \odot \nabla' {\varphi_v} \di x' \di t\\
  &\quad+\int_I \int_{\Omega'} \C_R^2 \big(
    \sym( \partial_t \EEE \nabla' u)
    +  \partial_t \EEE \nabla' v \odot \nabla' v 
  \big) : \nabla' v \odot \nabla' {\varphi_v} \di x' \di t\\
 &\qquad= \int_I \int_{\Omega'} f^{2D} {\varphi_v} \di x' \di t
    - \int_I \int_{\Omega'} \frac{1}{12} \Big(
    \CW^2 (\nabla')^2 v + \CD^2  \partial_t \EEE (\nabla')^2 v
  \Big) : (\nabla')^2 {\varphi_v} \di x' \di t,
\end{aligned}
\end{equation}
 where we also used the symmetry of $(\mathbb{B}^{(\alpha)})''$. \EEE   This shows \eqref{eq: weak equation2}.

%

\textit{Step 5 (Energy balance in the \GGG two-dimensional \EEE setting):}  Before proceeding with the strong convergence of strains and the heat-transfer equation,  we   establish the energy balance \OOO \eqref{balance2Dmaintheorem}  in the  \GGG two-dimensional \EEE setting. \EEE 
We enlarge the class of test functions in \eqref{zeromomeq2} and   \eqref{firstmomec2}  in the following way: \EEE
Note  by Lemma \ref{lem:conv_strain_stress} that $v \in H^1(I;H^2({\Omega'}))$ and $u \in H^1(I;H^1({\Omega'};\R^2))$. Moreover, it can be shown that $\mu \in L^{10/3}(I \times \Omega)$ for $\alpha = 2$ (\EEE see Remark~\ref{rem:improved_temp_bounds}). \EEE  Since $\mathbb{B}^{(\alpha)} \neq 0$  if and only if \EEE $\alpha = 2$, by approximation \eqref{zeromomeq2} remains true for $\varphi_u \in L^2 (I;H^1({\Omega'}))$ satisfying $\varphi_u = 0$ a.e.~on $I \times \Gamma_D'$.  Similarly, \EEE  \eqref{firstmomec2} \EEE holds true for all $\varphi_v \in L^2(I;H^2({\Omega'}))$ with $\varphi_v = 0$  a.e.~in $I \times \Gamma_D'$\EEE.
  On the one hand, testing \eqref{zeromomeq2} with $\varphi_u = \indic_{[0, t]} ( \partial_t \EEE u_1, \partial_t \EEE u_2)$  for $t \in I$ \EEE results in
\begin{align*}
  0 &= \int_0^t  \left(\int_{\Omega'}  \C_{W^{\rm el}}^2 \big(
    \sym(\nabla' u)
    + \frac{1}{2} \nabla'v \otimes \nabla' v
  \big) :
    \sym( \partial_t \EEE \nabla' u)
  \d \ZZZ x' \EEE  \right) \di s
 + \int_0^t \int_{\Omega'} \mu (\mathbb{B}^{(\alpha)})'' : \sym( \partial_t \EEE \nabla' u)
     \, {\rm d} x' \di s  \\
    &\phantom{=}\quad + \int_0^t \int_{\Omega'} \C_R^2 \big(
    \sym( \partial_t \EEE \nabla' u)
    +   \partial_t \EEE \nabla' v \odot \nabla' v 
  \big) : \sym( \partial_t \EEE \nabla' u)
   \di x' \di s,
\end{align*}
where we have used that $(\mathbb{B}^{(\alpha)})''$ is symmetric.
On the other hand, testing   \eqref{firstmomec2} \EEE with $\varphi_v = \indic_{[0, t]}  \partial_t \EEE v$ yields
\begin{align*}
&\int_0^t \int_{\Omega'} \left( \C_{W^{\rm el}}^2 \big(
    \sym(\nabla' u)
    + \frac{1}{2} \nabla'v \otimes \nabla' v
  \big)
  + \mu (\mathbb{B}^{(\alpha)})'' \right) : \nabla' v \odot  \partial_t \EEE \nabla' v \di x' \di s\\
&\quad + \int_0^t \int_{\Omega'} \C_R^2 \big(
    \sym( \partial_t \EEE \nabla' u)
    +  \partial_t \EEE \nabla' v \odot \nabla' v 
  \big) : \nabla' v \odot  \partial_t \EEE \nabla' v \di x' \di s\\
 & \qquad=  \int_0^t \int_{\Omega'} f^{2D}  \partial_t \EEE v \di x' \di s
    - \frac{1}{12} \int_0^t \int_{\Omega'} \Big(
    \CW^2 (\nabla')^2 v + \CD^2  \partial_t \EEE (\nabla')^2 v
  \Big) :  \partial_t \EEE (\nabla')^2 v \di x' \di s.
\end{align*}
Summing the last two equations and \ZZZ again \EEE using $\int_{-1/2}^{1/2} x_3^2 \di x_3 = 1/12$, $\int_{-1/2}^{1/2} x_3 \di x_3 = 0$, and the \EEE chain rule \cite[Proposition 3.5]{MielkeRoubicek20Thermoviscoelasticity}   we find for   a.e.~$t \in I$ \EEE that
\begin{align*}
  &\int_0^t  \frac{{\rm d}}{{\rm d} s}  \frac{1}{2} \EEE \left(\int_\Omega  \C_{W^{\rm el}}^2 \big(
    \sym(\nabla' u)
    + \frac{1}{2} \nabla'v \otimes \nabla' v  - \EEE x_3 (\nabla')^2 v
  \big) : \big(
    \sym(\nabla' u)
    + \frac{1}{2} \nabla'v \otimes \nabla' v  - \EEE x_3 (\nabla')^2 v
  \big) \dx  \right) \di s \\
 & + \int_0^t \int_\Omega \C_R^2  (
    \sym( \partial_t \EEE \nabla' u)
    +  \partial_t \EEE \nabla' v \odot \nabla' v  - \EEE x_3  \partial_t \EEE (\nabla')^2 v  ) :  (
    \sym( \partial_t \EEE \nabla' u)
    +  \partial_t \EEE \nabla' v \odot \nabla' v  - \EEE x_3  \partial_t \EEE (\nabla')^2 v )
   \di x' \di s \\
  &+ \int_0^t \int_{\Omega'} \mu (\mathbb{B}^{(\alpha)})'' :
    \big( \sym( \partial_t \EEE \nabla' u) +\nabla' v \odot  \partial_t \EEE \nabla' v \big) \di x' \di s
  = \int_0^t \int_{\Omega'} f^{2D}  \partial_t \EEE v \di x' \di s.
\end{align*}  
  Recalling \eqref{eq: phi0}, in the limit $h \to 0$, this leads to the   energy balance  \EEE  
\begin{equation}\label{balance2D}
\begin{aligned}
  &\phi_0^{\rm el}(u(t),v(t)) - \phi_0^{\rm el}(u(0),v(0)) + \int_0^t \int_{\Omega}  \C_R^2  G''_{t}:G''_{t}  + \mu (\mathbb{B}^{(\alpha)})'' \OOO :  G''_{t}  \EEE \, {\rm d}x \di s =\int_0^t \int_{\Omega'} f^{2D}  \partial_t \EEE v \, {\rm d}x' \di s
\end{aligned}
\end{equation}
for  a.e.~$t\in I$,  where we use the shorthand $ G''_t:= \partial_t \sym(G'')$  for the time derivative identified \EEE in \eqref{sym_dotG''}. The balance also holds for every $t \in I$ since the regularity of $u$ and $v$ imply that $u \in C(I;H^1(\Omega'))$ and $v \in C(I;H^2(\Omega'))$. 
Moreover, it \EEE can be expressed in a compact way by using an integration over the third coordinate and \eqref{sym_dotG''}, which gives \eqref{balance2Dmaintheorem}.

\textit{Step 6 (Strong convergence of the   symmetrized \GGG  strain rate\EEE):}  Before we derive the limit ing \EEE  heat-transfer equation,  \EEE we \OOO improve the convergence in \eqref{convstrainrateweak}. More precisely, we \GGG show \EEE that for the same subsequence as in \eqref{convstrainrateweak} \EEE \GGG we have \EEE
\begin{equation}\label{Chdot_strong_conv}
 h^{-2} \dot C^h \to 2 \, \partial_t \sym(G) \EEE \qquad \text{strongly in } L^2(I \times \Omega; \R^{3 \times 3})
\end{equation}
 \OOO where \EEE
\begin{equation*}
  \dot C^h \defas\EEE ( \partial_t \EEE \ZZZ \nabla_h \EEE  y^h)^T \nabla_h y^h + (\nabla_h y^h)^T  \partial_t \EEE \ZZZ \nabla_h \EEE  y^h.
\end{equation*} 
\OOO In this regard, \ZZZ again using the notation  $ G''_t= \partial_t \sym(G'')$, \EEE  we prove \EEE for a.e.~$t \in I$ \OOO the following four limits \EEE
\begin{subequations} 
\begin{align}
\ZZZ J_1 \EEE &\defas \liminf\limits_{h \to 0} h^{-4} \int_\Omega W^{\rm el} (\nabla_h y^h(t,x) ) \dx + h^{-4 } \int_\Omega H(\nabla_h^2 y^h(t,x)) \dx \geq {\phi}^{\rm el}_0(u(t),v(t)), \label{conv:strongphi_h} \\
\ZZZ J_2 \EEE &\defas \liminf\limits_{h \to 0}  h^{-4} \int_0^t  \int_\Omega 2 R (\nabla_h y^h, \partial_t  \nabla_h y^h, \theta^h) \dx \dt \EEE \geq \ZZZ \int_0^t \EEE \int_{\Omega'} \C_R^2\,    G''_{t}: G''_{t} \EEE \di x' \di s , \label{conv:strongdissi} \\
\ZZZ J_3 \EEE &\defas \lim\limits_{h \to 0} h^{-4} \int_0^t \int_\Omega \partial_F W^{\rm cpl}(\nabla_h y^h, \theta^h):  \partial_t \EEE \ZZZ \nabla_h \EEE  y^h \dx \dt = \ZZZ \int_0^t \EEE \int_{\Omega'} \mu (\mathbb{B}^{(\alpha)})'' : G''_{t} \di x' \di s, \label{conv:strong:coupl} \\
\ZZZ J_4 \EEE &\defas \lim\limits_{h \to 0}  h^{-4}\int_0^t \int_{\Omega} f^{3D}_h  \partial_t \EEE y^h_3 \dx \di s =  \int_0^t f^{2D}  \partial_t \EEE v \di x' \di s. \label{conv:strong:forceprod}
\end{align}
\end{subequations}
 In this regard, notice that \EEE \eqref{conv:strongphi_h} is addressed in \cite[Theorem 5.6]{FK_dimred}.
The representation in \eqref{diss_rate}, \eqref{convstrainrateweak}, $\C_R^3= 4D(\Id,0)$ \ZZZ (see \eqref{DDDDD}), \EEE \eqref{eq:quadraticformsnotred}, \eqref{CWCD},  and \EEE a standard lower semicontinuity argument  lead to \EEE \eqref{conv:strongdissi}.

 Let us now show \EEE \eqref{conv:strong:coupl}.
 W\EEE e first investigate the case $\alpha > 2$.
 Given \EEE $s \in (2/\alpha, 10/(3\alpha) \wedge 1)$, i\EEE n view of Remark~\ref{rem:improved_temp_bounds} and $2s > 1$, we  derive \EEE $\Vert \theta_h \Vert_{L^{2s}(I \times \Omega)} \leq C h ^{\alpha}$.
Thus, \eqref{avoidKorn}, \eqref{diss_rate}, \ref{D_bounds}, $\theta_h \wedge 1 \leq \theta_h^s$, \ZZZ the \EEE Cauchy-Schwarz inequality\EEE,   \eqref{rig:closetoidinfty}, and \eqref{bound:strainraterotated-another one} \EEE  imply that
\begin{align*}
  h^{-4} \int_0^t \int_\Omega
    \vert \partial_F W^{\rm cpl}(\nabla_h y^h, \theta^h) :  \partial_t \EEE \ZZZ \nabla_h \EEE  y^h \vert
    \dx \di s
  &\leq C h^{-4} \int_0^t \int_\Omega
  \vert \theta_h \wedge 1 \vert
  \vert \dot C^h \vert \dx \di s \notag \\
  &\leq C h^{-2} \Vert \theta_h \Vert_{L^{2s}(I \times \Omega)}^s
  \leq C h^{(s - 2/\alpha)\alpha}.  
\end{align*}
\ZZZ Consequently, as $s > \frac{2}{\alpha}$, this term vanishes  as $ h \to 0$. \EEE  We \EEE now deal with the case $\alpha = 2$.
 With \EEE the identity \eqref{writecplviadotc}, it holds that 
\begin{equation}\label{ident_coupl}
  \partial_F W^{\rm cpl}(\nabla_h y^h, \theta^h):  \partial_t \EEE \ZZZ \nabla_h \EEE  y^h =  \frac{1}{2} (\nabla_h y^h)^{-1}  \partial_F W^{\rm cpl} (\ZZZ \nabla_h y^h, \EEE \theta^h)  : \dot C^h, \EEE
\end{equation}
 where $C^h \ZZZ = \EEE (\nabla_h y^h)^T \nabla_h y^h$. \EEE  
 W\EEE e now show

\begin{equation}\label{strongcplstressrewritten}
h^{-2}   (\nabla_h y^h)^{-1}  \partial_F W^{\rm cpl} (\nabla_h y^h,\theta^h)  \to \mu \mathbb B^{(2)}\qquad \text{strongly in } L^2(I \times \Omega).
\end{equation}
\EEE
Indeed, by the fundamental theorem of calculus \ZZZ and \EEE a change of variables  we find that
\begin{equation*}
 h^{-2} (\nabla_h y^h)^{-1} \partial_F W^{\rm cpl} ( \ZZZ \nabla_h y^h \EEE, \theta^h)
  =   (\nabla_h y^h)^{-1} \int_0^{ h^{\OOO -2\EEE}\theta^h \EEE} \partial_{F\theta} W^{\rm cpl} (\ZZZ \nabla_h y^h \EEE, h^2 s) \di s.
\end{equation*}
In view of the definition of  $\mathbb B^{(2)}$ in \eqref{alpha_dep}, \ZZZ and using \EEE \ref{C_bounds}, \EEE     \eqref{rig:closetoidinfty},  and \eqref{theta_h_convXXX} \EEE this quantity converges, up to selecting a subsequence, pointwise a.e.~in $I \times \Omega$ to $\mu \,   \mathbb  B^{(2)}\EEE$.  Similarly, \ZZZ using also that $\partial_F\cplpot(F, 0) = 0$ for all $F \in GL^+(3)$ by \ref{C_zero_temperature},  we get \EEE $    \vert
    h^{-2} (\nabla_h y^h)^{-1} \partial_F W^{\rm cpl} ( C^h \EEE, \theta^h)
  \vert  \leq C  h^{- \OOO 2} \theta^h \EEE$. \EEE  In the case $\alpha =2$, \ZZZ the sequence $h^{-2} \theta^h$ is bounded in $L^q(I \times \Omega)$ for some $q >2$, \EEE  see Remark~\ref{rem:improved_temp_bounds}. \EEE  Thus, the sequence \eqref{strongcplstressrewritten}  is $L^2(I\times \Omega)\EEE$-equiintegrable, and \ZZZ therefore \EEE  \eqref{strongcplstressrewritten}  follows by  Vitali's convergence theorem.  \EEE  Consequently, i\EEE n view of  \eqref{convstrainrateweak},  \eqref{ident_coupl}, \EEE \eqref{strongcplstressrewritten}, and weak-strong convergence\OOO, \EEE  we discover that \EEE
\begin{equation*}
 h^{-4} \partial_F W^{\rm cpl}(\nabla_h y^h, \theta^h) :  \partial_t \EEE \ZZZ \nabla_h \EEE  y^h
 \weakly \mu  \mathbb{B}^{(2)} \EEE : \partial_t {\rm sym}(G) \qquad \text{weakly in } L^1(I \times \Omega).
\end{equation*}
By our assumption  \eqref{heattensorassumption} \EEE on $\mathbb{B}^{(2)}$ this concludes the proof of \eqref{conv:strong:coupl}.
 Eventually, \EEE recalling \eqref{def:displacements}, i\EEE n view of \eqref{forcebound} and \eqref{convvhdot},  the limit \EEE \eqref{conv:strong:forceprod} follows by weak-strong convergence.

 Based on the energy  balance \EEE and \eqref{conv:strongphi_h}--\eqref{conv:strong:forceprod}, we will now show \eqref{Chdot_strong_conv}. \EEE
By \eqref{balance2D}, \eqref{conv:strong:forceprod}, \eqref{well-preparednessinitial}, \eqref{balance3D} \OOO with both sides divided by $h^{-4}$\EEE,  and \EEE \eqref{conv:strongphi_h}--\eqref{conv:strong:coupl} we find for a.e.~$t \in I$ that
\begin{align*}
  &{\phi}^{\rm el}_0(u(t),v(t)) + \int_0^t \int_{\Omega'} \C_R^2 G''_{t} : G''_{t} \di x' \di s
    + \int_0^t \int_{\Omega'} \mu (\mathbb{B}^{(\alpha)})'' : G''_{t}\di x' \di s \\
  &\quad= {\phi}^{\rm el}_0(u(0),v(0)) + \int_0^t \ZZZ \int_{\Omega'} \EEE f^{2D}  \partial_t \EEE v \di x' \di s \\
  &\quad= \lim\limits_{h \to 0} \left(
    h^{-4} \int_\Omega W^{\rm el} (\nabla_h y^h_0(x) ) \dx + h^{-4 } \int_\Omega H(\nabla_h^2 y^h_0(x)) \dx
    + h^{-4}\int_0^t \int_{\Omega}  f^{3D}_h \EEE  \partial_t \EEE y^h_3 \dx \di s
  \right) \\
  &\quad\geq \ZZZ J_1 + J_2 + J_3  \EEE \geq {\phi}^{\rm el}_0(u(t),v(t))
    + \int_0^t \int_{\Omega'} \C_R^2 G''_{t}:G''_{t} \di x' \di s
    + \int_0^t \int_{\Omega'} \mu (\mathbb{B}^{(\alpha)})'' : G''_{t} \di x' \di s.
\end{align*}
 As a consequence\EEE, the  inequality \EEE in \eqref{conv:strongdissi} must be  an \EEE equality.
 By \eqref{chain_rule_Fderiv} we \EEE can write
\begin{equation*}
 \frac{1}{2} (\nabla_h y^h)^{-1} \partial_{\dot F} R(\nabla_h y^h,  \partial_t \EEE \ZZZ \nabla_h \EEE  y^h, \theta^h)
  =  D( C^h \EEE, \theta^h)
     \dot C^h.
\end{equation*}
Thus,  with \eqref{rig:closetoidinfty} and \EEE  \eqref{conv:stress:visc} we find
\begin{equation*}
D( C^h\EEE, \theta^h)
    h^{-2} \dot C^h
  \weakly \frac{1}{2} \C_R^3 \partial_t  \sym(G)
  \qquad\text{weakly in } L^2(I\times \Omega; \R^{3 \times 3}).
\end{equation*}
 This along \EEE with \ref{D_quadratic}--\ref{D_bounds}, \eqref{diss_rate}, \eqref{convstrainrateweak}, \eqref{rig:closetoidinfty},   \eqref{theta_h_convXXX},    equality in  \eqref{conv:strongdissi}, \ZZZ and  dominated convergence yields \EEE
\begin{align*}
  &c_0 \int_I \int_\Omega \vert h^{-2} \dot C_h - 2 \partial_t \sym (G) \vert^2 \di x \di t \\
  &\quad\leq \int_I \int_\Omega D(C_h, \theta^h) \big(h^{-2} \dot C_h - 2 \partial_t \sym (G)\big)
    : \big(h^{-2}\dot C_h - 2 \partial_t \sym (G)\big) \di x \di t \\
  &\quad= h^{-4} \int_I \int_\Omega
   2R(\nabla_h y^h ,  \partial_t \EEE \ZZZ \nabla_h \EEE  y^h, \theta^h) \di x \di t
    - 4 \int_I \int_\Omega D(C_h, \theta^h) \partial_t \sym(G) : h^{-2} \dot C_h \di x \di t \\
  &\phantom{\quad=}\quad + 4 \int_I \int_\Omega
    D(C_h,\theta^h) \partial_t \sym(G) : \partial_t \sym(G) \di x \di t \to 0 \quad \text{as } h \to 0,
\end{align*}
where we again used \eqref{CWCD} and $\C_R^3 = 4D(\Id,0)$,  \ZZZ see   \eqref{DDDDD}.  \EEE This concludes the proof of \eqref{Chdot_strong_conv}.

\textit{Step 7 \GGG (Derivation \EEE of the limiting heat-transfer equation\EEE):}
It remains to derive  \eqref{eq:temperature2dweak}. \EEE
Recall \eqref{def:scaledtemp}  and l\EEE et $r \in [1, 5/4)$.
By \eqref{boundres:temperaturegrad} we have that
\begin{align}\label{rescaledmuhconv}
 h^{-\alpha} \EEE \nabla_h \theta\EEE^h \weakly g \qquad \text{weakly in }L^r(I \times \Omega; \R^3)
\end{align}
 for some $g = (g_1,g_2,g_3) \in L^r(I \times \Omega; \R^3)$. 
In view of \eqref{theta_h_convXXX} \GGG and Proposition~\ref{prop:compactness}, \EEE we find $\nabla \ZZZ \mu \EEE = (g_1, g_2, 0)$, \ZZZ and \EEE $g_1$ and $g_2$ \OOO do not  depend on $x_3$, \ZZZ i.e., \EEE
\begin{equation*}
  \nabla' \mu(t, x') =
  \begin{pmatrix} g_1(t, x') \\ g_2(t, x') \end{pmatrix} \qquad \text{for a.e.~} (t, x') \in I \times \Omega'.
\end{equation*}
Let us now define
\begin{equation*}
  \nu(t, x') \defas \int_{-1/2}^{1/2} g_3(t, x) \di x_3.
\end{equation*}
We continue by deriving a relation between $\nu$ and $\nabla' \mu$.
In this regard, let us test \eqref{weak_form_heat_res} with $\vphi(t, x) := x_3 \psi(t, x')$ for $\psi \in C^\infty(I \times \overline{\Omega'}\EEE)$ satisfying $\psi(T)=0$ and multiply the resulting equation by $h^{1 - \alpha}$.
By $\xi^{(\alpha)} \leq \xi$, \eqref{diss_rate}, \eqref{boundres:dissipation},  \ZZZ \eqref{conv:strong:coupl}, \EEE  a trace estimate, \eqref{boundres:temperature}, \eqref{boundres:temperaturegrad}, \eqref{assumption:temp}, and \eqref{well-preparednessinitial} this  leads to \EEE 
\begin{equation*}
 \left\vert
  \int_I \int_\Omega (
  \ZZZ   \mathcal{K}^h_{31} \EEE  h^{-\alpha}  \partial_1 \theta^h \EEE
    + \ZZZ \mathcal{K}^h_{32} \EEE  h^{-\alpha} \partial_2 \theta^h \EEE
    + \ZZZ \mathcal{K}^h_{33} \EEE h^{-1-\alpha} \partial_3 \theta^h \EEE
  ) \psi \di x \di t \right\vert \leq C h,
\end{equation*}
where we shortly wrote $\ZZZ \mathcal{K}^h \EEE \coloneqq \mathcal{K}(\nabla_h y^h, \theta^h)$.
Recall the definition of ${\mathbb K}$ above \eqref{barcvandk}. \EEE Moreover, recall that $\mu$  does not depend on  $x_3$. \EEE
Due to \eqref{spectrum_bound_K} and the strong \GGG convergence \EEE of $(\theta^h)_h$ and $(\nabla_h y^h)_h$, see \eqref{theta_h_convXXX}, and \eqref{rig:closetoidinfty}, respectively, we can use   dominated convergence   to get
\begin{align*}
\text{ $\mathcal{K}(\nabla_h y^h, \theta^h) \to \mathbb{K}$ \qquad strongly in $L^{\bar q}(I \times \Omega;\R^{3\times 3})$ for any $\bar q \in [1,\infty)$}.
\end{align*}
We pass to the limit  $h \to 0$ \EEE in the integral on the left-hand side above.  
 \EEE With Fubini's theorem this shows
\begin{equation*}
  \int_I \int_{\Omega'} (
     { {\mathbb K} \EEE}_{31} \partial_1 \mu
    +  { {\mathbb K} \EEE}_{32} \partial_2 \mu
    +  { {\mathbb K} \EEE}_{33} \nu
  ) \psi \di x' \di t
  = \int_I \int_{\Omega} (
    \mathbb{K}_{31} \partial_1 \mu
    + \mathbb{K}_{32} \partial_2 \mu
    + \mathbb{K}_{33} g_3
  ) \psi \di x \di t = 0.
\end{equation*}
 \EEE Note that $\mathbb{K}_{33} \geq c_0 > 0$ by \eqref{spectrum_bound_K}.
 Hence, b\EEE y the arbitrariness of $\psi$ we derive that
\begin{equation}\label{nu_reprnew}
  \nu = - \frac{
     { {\mathbb K} \EEE}_{31} \partial_1 \mu
    +  { {\mathbb K} \EEE}_{32} \partial_2 \mu
  }{ { {\mathbb K} \EEE}_{33}} \qquad \text{for a.e.~} (t,x') \in I \times \Omega'.
\end{equation}
We continue by testing \eqref{weak_form_heat_res} with $\varphi\EEE \in C^\infty(I \times \bar \Omega)$ independent of $x_3$ satisfying $\varphi\EEE(T) = 0$,  and \EEE divide the resulting equation by $h^\alpha$, to get  
\begin{equation}\label{last_test}
\begin{aligned}
  &h^{-\alpha}\int_I \int_\Omega
    \mathcal{K}(\nabla_h y^h, \theta^h)  \nabla_h \theta^h \cdot  \begin{pmatrix} \nabla' { \varphi \EEE} \\ 0 \end{pmatrix}
    - \Big(
        \xi^{(\alpha)}(\nabla_h y^h ,  \partial_t \EEE \ZZZ \nabla_h \EEE  y^h,\theta^h \big)
      + \partial_F W^{\rm{cpl}}(\nabla_h y^h, \theta^h) :  \partial_t \EEE \ZZZ \nabla_h \EEE  y^h
    \Big) { \varphi \EEE} \di x \di t \\
  &\  - h^{-\alpha}\int_I \int_\Omega
    W^{\rm{in}}(\nabla_h y^h,\theta^h)  \partial_t \EEE { \varphi \EEE} \di x \di t
    + \kappa \int_I \int_{\Gamma\EEE} h^{-\alpha}(\theta^h-\theta_\flat^h\EEE) { \varphi \EEE} \di \haus^2 \di t 
     = \EEE \int_\Omega h^{-\alpha}W^{\rm in}(\nabla_h y_0^h, \theta_0^h) { \varphi \EEE}(0) \di x.
\end{aligned}
\end{equation}
\ZZZ Note that  by  \eqref{rescaledmuhconv}--\eqref{nu_reprnew} \EEE we find
\begin{equation*}
h^{-\alpha}\int_I \int_\Omega
    \mathcal{K}(\nabla_h y^h, \theta^h)  \nabla_h \theta^h \cdot  \begin{pmatrix} \nabla' { \varphi \EEE} \\ 0 \end{pmatrix} \di x \di t\to \int_I \int_{\Omega'} \mathbb{K} \begin{pmatrix} \nabla' \mu \\ \nu \end{pmatrix}
    \cdot \begin{pmatrix} \nabla' { \varphi \EEE} \\ 0 \end{pmatrix} \di x' \di t \quad \text{as } h \to 0.
\end{equation*}
 Moreover, \EEE \eqref{diss_rate}, \eqref{diss_rate_truncated}, \ZZZ \eqref{DDDDD}, \EEE  \eqref{alpha_dep}, \EEE \eqref{Chdot_strong_conv}, \ref{D_bounds},   \EEE and \eqref{CWCD}  lead to \EEE
\begin{align*}
h^{-\alpha}\int_I \int_\Omega
\xi^{(\alpha)}(\nabla_h y^h ,  \partial_t \EEE \ZZZ \nabla_h \EEE  y^h,\theta^h \big)
      { \varphi \EEE} \di x \di t\to \int_I \int_{\Omega'} \GGG \big( \EEE   \CD^{2, \alpha} \EEE G''_{t}: G''_{t} \GGG \big) \EEE { \varphi \EEE} \di x \di t, \quad \text{as } h \to 0,
\end{align*}
 where we recall the shorthand $G''_{t} = \partial_t {\rm sym}(G'')$. \EEE Thus, \ZZZ \eqref{conv:strong:coupl}, \EEE
 \eqref{convscaledinternal}, \eqref{convscaledinternalinitial}, a trace estimate, \eqref{convmuhnabla},  and \EEE \eqref{assumption:temp}  allow \EEE us to pass to the limit $h \to 0$ in \eqref{last_test}, resulting in
\begin{align*}
  &\int_I \int_{\Omega'} \Big( \mathbb{K} \begin{pmatrix} \nabla' \mu \\ \nu \end{pmatrix}
    \cdot \begin{pmatrix} \nabla' { \varphi \EEE} \\ 0 \end{pmatrix} - \overline c_V \mu  \partial_t \EEE { \varphi \EEE} \Big) \di x' \di t
  -  \int_I \int_{\Omega'} \GGG \big( \EEE \CD^{2,\alpha} G''_{t}: G''_{t}  \GGG \big) \EEE{ \varphi \EEE} \di\ZZZ  x' \EEE \di t \\
  &\quad + \kappa \int_I \int_{ \Gamma'} \mu { \varphi \EEE} \di \haus^1(x') \di t
  = \kappa \int_I \int_{ \Gamma'}  \mu_\flat { \varphi \EEE} \di \haus^1(x') \di t + \overline c_V \int_\Omega \mu_0 { \varphi \EEE}(0) \di x'.
\end{align*}
Recalling the definition of $\tilde{\mathbb{K}}$ in \eqref{heattensorredd}, by \ZZZ the symmetry of $\mathbb{K}$, \EEE \eqref{nu_reprnew}, and \eqref{sym_dotG''},  the above equation \EEE further simplifies to
\begin{align*}
  &\int_I \int_{\Omega'} \Big( \tilde{\mathbb{K}} \nabla' \mu \cdot \nabla' { \varphi \EEE} - \overline c_V \mu  \partial_t \EEE { \varphi \EEE}  \Big)\di x' \di t     -  \frac{1}{12}\int_I \int_{\Omega'} \GGG \big( \EEE \CD^{2,\alpha}
		 \partial_t \EEE (\nabla')^2 v :
		 \partial_t \EEE (\nabla')^2 v  \GGG \big) \EEE{ \varphi \EEE} \di \ZZZ x' \EEE \di t \\
  &\quad- \int_I \int_{\Omega'} \CD^{2,\alpha} \Big(
      \sym( \partial_t \EEE \nabla' u)
      + \frac{1}{2}  \partial_t \EEE \nabla' v \odot  \partial_t \EEE \nabla' v
    \Big) : \Big(
      \sym( \partial_t \EEE \nabla' u)
      +  \partial_t \EEE \nabla' v \odot \nabla' v
    \Big) { \varphi \EEE} \di x' \di t  \\
  &\quad + \kappa \int_I \int_{\GGG \Gamma' \EEE} \mu { \varphi \EEE} \di \haus^1(x') \di t
  = \kappa \int_I \int_{\GGG \Gamma' \EEE} \mu_\flat { \varphi \EEE} \di \haus^1(x') \di t + \overline c_V \int_\Omega \mu_0 { \varphi \EEE}(0) \di x'. \EEE
\end{align*}  \EEE
This gives \eqref{eq:temperature2dweak}.

To conclude the proof, we note that  $u$ and $v$ satisfy   the initial conditions. In fact,   \eqref{convuh}--\eqref{convvhdot} and  \ZZZ the \EEE Aubin-Lions lemma imply strong convergence of $(u_h)_h$ and $(v_h)_h$ in  $C(I; L^{\bar r}(\Omega'))$, up to a subsequence, for some suitable $\bar r \in [1,+\infty)$.  
\end{proof}

\ZZZ

\section*{Acknowledgements} 
This work was funded by  the DFG project FR 4083/5-1 and  by the Deutsche Forschungsgemeinschaft (DFG, German Research Foundation) under Germany's Excellence Strategy EXC 2044 -390685587, Mathematics M\"unster: Dynamics--Geometry--Structure.

  \EEE

\typeout{References}

\end{document}